\documentclass{article}%
\usepackage{amssymb}
\usepackage{amsfonts}
\usepackage{amsmath}
\usepackage{graphicx}%
\numberwithin{equation}{section}
\numberwithin{equation}{section}
\numberwithin{equation}{section}
\def\e{\varepsilon}

\newtheorem{theorem}{Theorem}[section]

\newtheorem{definition}[theorem]{Definition}

\newtheorem{lemma}[theorem]{Lemma}

\newtheorem{proposition}[theorem]{Proposition}

\newtheorem{remark}[theorem]{Remark}

\newenvironment{proof}[1][Proof]{\noindent\textbf{#1.} }{\ \rule{0.5em}{0.5em}}
\begin{document}

\title{Relaxation for an optimal design problem with linear growth and perimeter
penalization }
\author{\textsc{Gra\c{c}a Carita}\thanks{CIMA-UE, Departamento de Matem\'{a}tica,
Universidade de \'{E}vora, Rua Rom\~{a}o Ramalho, 59 7000-671 \'{E}vora, Portugal.
E-mail: gcarita@uevora.pt},\textsc{ Elvira Zappale}\thanks{D.I.In.,
Universit\`{a} degli Studi di Salerno, Via Giovanni Paolo II, 132, 84084 Fisciano (SA)
Italy. E-mail: ezappale@unisa.it}}
\maketitle

\begin{abstract}
The paper is devoted to the relaxation and integral representation in the space of functions of bounded variation for an integral energy arising from optimal design problems. The presence of a perimeter penalization is also considered in order to avoid non existence of admissible solutions, besides this leads to an interaction in the limit energy.  
Also more general models have been taken into account.

\medskip


\end{abstract}

\section{Introduction}

The optimal design problem, devoted to find the minimal energy configurations of a mixture of two conductive materials, has been widely studied since the pioneering papers \cite{KS1,KS2, KS3}. It is well known that, given a container $\Omega$ and prescribing only the volume fraction of the material where it is expected to have a certain conductivity, an optimal configuration might not exist. To overcome this difficulty, Ambrosio and Buttazzo in \cite{AB} imposed a perimeter penalization and studied the following minimization problem
$$
\min\left\{\int_E (\alpha |Du|^2 +g_1(x,u))dx + \int_{\Omega\setminus E}(\beta |Du|^2+ g_2(x,u)) dx + \sigma P(E,\Omega):E \subset \Omega, u \in H^1_0(\Omega)\right\},
$$
finding the solution $(u,E)$ and describing the regularity properties of the optimal set $E$.

In this paper we are considering the minimization of a similar functional, where the energy density $|\cdot|^2$ has been replaced by more general $W_i$, $i=1,2$ without any convexity assumptions and with linear growth, and since the lower order terms $g_1(x,u)$ and $g_2(x,u)$ do not play any role in the asymptotics, we omit them in our subsequent analysis. The case of $W_i$, $i=1,2$, not convex with superlinear growth has been studied in the context of thin films in \cite{CZ}.

Thus, given $\Omega$  a bounded open subset of $\mathbb{R}^{N}$, we assume that
$W_{i}:\mathbb{R}^{d\times N}\rightarrow\mathbb{R}$ are continuous functions
such that there exist positive constants $\alpha,\beta$ for which
\begin{equation}
\alpha|\xi|\leq W_{i}(\xi)\leq\beta(1+|\xi|)\hbox{ for every }\xi\in
\mathbb{R}^{d\times N},\;\;\;\ i=1,2. \label{H1}%
\end{equation}



We consider the following optimal design problem%

\begin{equation}
\underset{%
\begin{array}
[c]{c}%
{\small u\in W}^{1,1}{\small (\Omega;}\mathbb{R}^{d}{\small )}\\
{\small \chi_E\in BV(\Omega;\{0,1\})}%
\end{array}
}{\inf}\left\{  {\int_{\Omega}\left(  \chi_{E}W_{1}(\nabla u)+(1-\chi
_{E})W_{2}\right)  (\nabla u)dx+P(E;\Omega):u=u_{0}}\text{ {on }}%
{\partial\Omega}\right\}  \label{originalpb}%
\end{equation}
where $\chi_{E}$ is the characteristic function of $E\subset\Omega\ $ which
has finite perimeter, see \eqref{perimeter} below.

Note that by \eqref{perimeter}   and the definition of
total variation, $P\left(  E;\Omega\right)  =\left\vert D\chi_{E}\right\vert
\left(  \Omega\right)  $ and we are lead to the subsequent minimum problem%
\begin{equation}\nonumber
\inf\limits_{%
\begin{array}
[c]{l}%
{\small u\in W}^{1,1}{\small (\Omega;{\mathbb R}}^{d}{\small )}\\
{\small \chi}_{E}{\small \in BV(\Omega;\{0,1\})}%
\end{array}
}\left\{
{\displaystyle\int_{\Omega}}
\left(  \chi_{E}W_{1}+\left(  1-\chi_{E}\right)  W_{2}\right)  \left(
\nabla u\right)  dx+\left\vert D\chi_{E}\right\vert \left(
\Omega\right)  :u=u_0\text{ on }\partial\Omega\right\}  . 
\end{equation}

The lack of convexity of the energy requires a relaxation procedure. To this end we start by localizing our energy, first we introduce the functional 
 $F_{\cal OD}:L^{1}(\Omega;\{0,1\})\times L^{1}(\Omega;\mathbb{R}^{d})\times
\mathcal{A}\left(  \Omega\right)  \rightarrow\lbrack0,+\infty]$ defined by
\begin{equation}
F_{\cal OD}(\chi,u;A):=\left\{
\begin{array}
[c]{lll}%
{
{\displaystyle\int_{A}}
\left(  \chi_{E}W_{1}\left(  \nabla u\right)  +(1-\chi_{E})W_{2}\left(  \nabla
u\right)  \right)  dx+\left\vert D\chi_{E}\right\vert (A)} &  & \text{in
}BV(A;\{0,1\})\times W^{1,1}(A;\mathbb{R}^{d}),\text{\bigskip}\\
+\infty &  & \text{otherwise.}%
\end{array}
\right.  \label{Je}%
\end{equation}
Then we consider the relaxed localized energy of $\left(  \ref{Je}\right)  $ given by%
\begin{equation}\nonumber
\begin{array}
[c]{c}%
\mathcal{F_{\cal OD}}\left(  \chi,u;A\right)  :=\inf\left\{  \underset{n\rightarrow
\infty}{\lim\inf}%
{\displaystyle\int_{A}}
{\left(  \chi_{n}W_{1}\left(  \nabla u_{n}\right)  +(1-\chi_{n})W_{2}\left(
\nabla u_{n}\right)  \right)  dx+}\left\vert D\chi_{n}\right\vert \left(
A\right)  {:}\left\{  u_{n}\right\}  \subset W^{1,1}\left(  A;\mathbb{R}%
^{d}\right)  ,\right. \\
\qquad\qquad\qquad\qquad\qquad\qquad\left.  \left\{  \chi_{n}\right\}  \subset
BV\left(  A;\left\{  0,1\right\}  \right)  ,~u_{n}\rightarrow u\text{ in
}L^{1}\left(  A;\mathbb{R}^{d}\right)  \text{ and }\chi_{n}\overset{\ast
}{\rightharpoonup}\chi\text{ in }BV\left(  A;\left\{  0,1\right\}  \right)
\right\}.
\end{array}
\end{equation}

Let $V:\{0,1\}\times\mathbb{R}^{d\times N}\rightarrow(0,+\infty)$ be given by
\begin{equation}
V\left(  q, z\right)  :=q W_{1}(z)+\left(  1-q\right)  W_{2}\left(
z\right)  \label{Vbar}%
\end{equation}
and $\overline{F_{\cal OD}}:BV(\Omega;\{0,1\})\times BV(\Omega;\mathbb{R}^{d}%
)\times\mathcal{A}\left(  \Omega\right)  \rightarrow\lbrack0,+\infty]$ be defined as
\begin{equation}\label{representation}%
\overline{F_{\cal OD}}\left(  \chi,u;A\right)  :=\int_{A}QV\left(  \chi,\nabla u\right)
dx+\int_{A}QV^{\infty}\left(  \chi,\frac{dD^{c}u}{d\left\vert D^{c}%
u\right\vert }\right)  d\left\vert D^{c}u\right\vert +\int_{J_{\left(
\chi,u\right)  }\cap A}K_{2}\left(  \chi^{+},\chi^{-},u^{+},u^{-},\nu\right)
d\mathcal{H}^{N-1}
\end{equation}
where $QV$ is the quasiconvex envelope of $V$ given in $\left(  \ref{Qfbar}%
\right)  ,$ $QV^{\infty}$ is the recession function of $QV,$ namely,%
\begin{equation}\label{QVinfty}
QV^{\infty}\left(  q, z\right)  :=\lim_{t\rightarrow\infty}\frac{QV\left(
q,t z\right)  }{t},
\end{equation}
and%
\begin{equation}
{K_{2}(a,b,c,d,\nu):=\inf}\left\{  {%
{\displaystyle\int_{Q_{\nu}}}
QV^{\infty}(\chi(x),\nabla u(x))dx+|D\chi|(Q_{\nu}):}\left(  \chi,u\right)
{{\in\mathcal{A}_2(a,b,c,d,\nu)}}\right\}  {{,}} \label{K2}%
\end{equation}
where%
\begin{align}
\mathcal{A}_2\left(  a,b,c,d,\nu\right)   &  :=\left\{  \left(  \chi,u\right)
\in BV\left(  Q_{\nu};\left\{  0,1\right\}  \right)  \times W^{1,1}\left(
Q_{\nu};\mathbb{R}^{d}\right)  :\right. \label{AFR}\\
&  \left.  (\chi(y),u\left(  y\right))  =(a,c)\text{ if }y\cdot\nu=\frac{1}{2},~(\chi(y),u\left(
y\right))  =(b,d)\text{ if }y\cdot\nu=-\frac{1}{2},\right.\nonumber\\
&  \left.  (\chi, u)\text{ are
}1-\text{periodic in }\nu_{1},\dots,\nu_{N-1} \hbox{ directions}\right\}  ,\nonumber
\end{align}
for $\left(  a,b,c,d,\nu\right)  \in\left\{  0,1\right\}  \times\left\{
0,1\right\}  \times\mathbb{R}^{d}\times\mathbb{R}^{d}\times S^{N-1},$ with
$\left\{  \nu_{1},\nu_{2},\dots,\nu_{N-1},\nu\right\} $ an orthonormal
basis of $\mathbb{R}^{N}$ and $Q_\nu$ the unit cube, centered at the origin, with one direction parallel to $\nu$.

In Section \ref{appl} we obtain the following integral representation.

\begin{theorem}
\label{mainthm}Let $\Omega\subset\mathbb{R}^{N}$ be a bounded open set and let
$W_{i}:\mathbb{R}^{{d}\times{N}}\rightarrow\lbrack0,+\infty)$, $i=1,2$, be continuous functions
satisfying \eqref{H1}. Let $\overline{F_{\cal OD}}$ be the functional defined in $\left(
\ref{representation}\right)  $. Then for every $(\chi,u)\in L^1(\Omega;\{0,1\})\times
L^1(\Omega;\mathbb{R}^{d})$%
\begin{equation}\nonumber
\mathcal{F_{\cal OD}}(\chi,u;A)=\left\{
\begin{array}{ll}
\overline{F_{\cal OD}}(\chi,u;A) &\hbox{ if } (\chi, u)\in BV(\Omega;\{0,1\}) \times BV(\Omega;\mathbb R^d) ,\\
\\
+\infty & \hbox{ otherwise.}
\end{array}
\right.
\end{equation}

\end{theorem}
This result will be achieved as a particular case of a more general theorem dealing with special functions of bounded variation which are piecewise constants. 

In fact we provide an integral representation for the relaxation of the functional $F:L^{1}(\Omega;\mathbb{R}^{m})\times L^{1}(\Omega;\mathbb{R}^{d}%
)\times\mathcal{A}\left(  \Omega\right)  \rightarrow[ 0,+\infty]$ defined by
\begin{equation}
\label{FG}F(v,u;A):=\left\{
\begin{array}
[c]{lll}%
\displaystyle{\int_Af(v, \nabla u)dx+\int_{A\cap J_v} g(v^{+}, v^{-}, \nu
_{v})d\mathcal{H}^{N-1}} & \hbox{ in }SBV_{0}(A;\mathbb{R}^{m}) \times
W^{1,1}(A;\mathbb{R}^{d}),\\
&  & \\
+\infty & \hbox{ otherwise,} &
\end{array}
\right.
\end{equation}
where $SBV_0(A;\mathbb R^m)$ is defined in \eqref{SBV0} (see Section \ref{pre}) and  $f: \mathbb{R}^{m} \times\mathbb{R}^{d \times N}\to[0,+\infty[$, $g
:\mathbb{R}^{m} \times\mathbb{R}^{m} \times S^{N-1} \to[0, +\infty[$ satisfy
the following hypotheses:

\begin{itemize}
\item[($F_{1}$)] $f$ is continuous;

\item[$(F_{2})$] there exist $0<\beta'\leq\beta$ such that
\[
\beta'|z| \leq f(q, z) \leq\beta(1+ |z|),
\]
for every $(q, z) \in\mathbb{R}^{m}\times\mathbb{R}^{d \times N};$

\item[($F_{3}$)] there exists $L>0$ such that
\[
\left| f(q_{1},z)- f(q_{2},z)\right| \leq L|q_{1}-q_{2}|(1+ |z|)
\]
for every $q_{1},q_{2} \in\mathbb{R}^{m}$ and $z \in\mathbb{R}^{d \times N};$

\item[($F_{4}$)] there exist $\alpha\in(0,1),$ and $C,L>0$ such that
\[
t|z|>L\ \Rightarrow\ \left| f^{\infty}(q, z)- \frac{f(q, t z)}{t}\right|
\leq C \frac{|z|^{1-\alpha}}{t^{\alpha}},\ \hbox{ for every } (q, z)
\in\mathbb{R}^{m} \times\mathbb{R}^{d \times N}, t \in\mathbb{R},
\]

\end{itemize}

with $f^{\infty}$ the recession function of $f$ with respect to the last
variable, defined as%

\begin{equation}
\label{finfty}\displaystyle{f^{\infty}(q, z):= \limsup_{t \to\infty}%
\frac{f(q,t z)}{t},}%
\end{equation}
for every $(q, z)\in\mathbb{R}^{m} \times\mathbb{R}^{d \times N};$

\begin{itemize}
\item[($G_{1}$)] $g$ is continuous;

\item[($G_{2}$)] there exists a constant $C>0$ such that
\[
\frac{1}{C}(1+|\lambda-\theta|)\leq g(\lambda, \theta, \nu)\leq C (1+|\lambda-
\theta|),
\]
for every $(\lambda, \theta, \nu)\in\mathbb{R}^{m} \times\mathbb{R}^{m} \times
S^{N-1}$,

\item[($G_{3}$)] $g(\lambda, \theta, \nu)= g(\theta, \lambda, -\nu)$, for
every $(\lambda, \theta, \nu)\in\mathbb{R}^{m} \times\mathbb{R}^{m} \times
S^{N-1}$.
\end{itemize}

The relaxed localized energy of \eqref{FG} is given by%
\begin{equation}%
\begin{array}
[c]{c}%
\mathcal{F}\left(  v,u;A\right)  :=\inf\left\{  \displaystyle{\liminf_{n
\to\infty} \left( \int_{A} f(v_{n}, \nabla u_{n})dx +\int_{J_{v_{n}}\cap A}
g({v_{n}}^{+}, {v_{n}}^{-}, \nu_{v_{n}}) d\mathcal{H}^{N-1}\right) :\left\{
u_{n}\right\}  \subset W^{1,1}\left(  A;\mathbb{R}^{d}\right)  },\right. \\
\qquad\qquad\qquad\qquad\qquad\qquad\left.  \{ v_{n}\} \subset SBV_{0}\left(
A;\mathbb{R}^{m}\right)  ,~u_{n}\rightarrow u\text{ in }L^{1}\left(
A;\mathbb{R}^{d}\right)  \text{ and }v_{n} \to v \hbox{ in }L^1(A;\mathbb R^m)  \right\}.
\end{array}
\label{calFG}%
\end{equation}


Let $\overline{F_{0}}:SBV_{0}(\Omega;\mathbb{R}^{m})\times BV(\Omega;\mathbb{R}%
^{d})\times\mathcal{A}\left(  \Omega\right)  \rightarrow[0,+\infty]$ be given by
\begin{equation}
\label{representationFG}\overline{F_{0}}\left(  v,u;A\right)  :=\int_{A}Qf\left(
v,\nabla u\right)  dx+\int_{A}Qf^{\infty}\left(  v,\frac{dD^{c}u}{d\left\vert
D^{c}u\right\vert }\right)  d\left\vert D^{c}u\right\vert +\int_{J_{\left(
v,u\right)  }\cap A}K_{3}\left(v^{+},v^{-},u^{+},u^{-},\nu\right)
d\mathcal{H}^{N-1},
\end{equation}
where $Qf$ is the quasiconvex envelope of $f$ given in
\eqref{Qfbar}, $Qf^{\infty}$ is the recession function of
$Qf$, and $K_{3}: \mathbb{R}^{m} \times\mathbb{R}^{m} \times\mathbb{R}^{d}
\times\mathbb{R}^{d} \times S^{N-1} \to[0, +\infty[$ is defined as

\begin{equation}
\label{K3}
\begin{array}{ll}
K_{3}(a,b,c,d,\nu):=
\\
\displaystyle{\inf\left\{  \int_{Q_{\nu}}Qf^{\infty}(v(x),
\nabla u(x))dx+ \int_{J_{v}\cap Q_{\nu}}g(v^{+}(x), v^{-}(x), \nu(x))d
\mathcal{H}^{N-1}: (v,u)\in\mathcal{A}_3(a,b,c,d,\nu)\right\}}
\end{array}
\end{equation}
where
\begin{equation}
\label{A3}%
\begin{array}
[c]{ll}%
\displaystyle{\mathcal{A}_3\left(  a,b,c,d,\nu\right) :=} &\displaystyle{\left\{ (v,u)\in(SBV_{0}(Q_{\nu
};\mathbb{R}^{m})\cap L^{\infty}(Q_{\nu};\mathbb{R}^{m}))\times W^{1,1}(Q_{\nu};\mathbb{R}^{d}): \right.}
\\
\\
&\displaystyle{(v(y), u(y))= (a,c) \hbox{ if } y \cdot\nu=\frac
{1}{2},  (v(y), u(y))= (b,d) \hbox{ if } y\cdot\nu=-\frac{1}{2},}
\\
\\
&\displaystyle{\left. (v, u) \hbox{ are } 1-\hbox{periodic in }\nu_{1}, \dots,
\nu_{N-1} \hbox{ directions} \right\},}
\end{array}
\end{equation}
with $\left\{  \nu_{1},\nu_{2},\dots,\nu_{N-1},\nu\right\}  $ an
orthonormal basis of $\mathbb{R}^{N}.$

In the following  we present the main result.
\begin{theorem}
\label{mainthmgen} Let $\Omega\subset\mathbb{R}^{N}$ be a bounded open set and
let $f:\mathbb{R}^{m} \times\mathbb{R}^{d \times N}\to[0, +\infty[$ be a
function satisfying $(F_{1})- (F_{4})$ and $g:\mathbb{R}^{m} \times
\mathbb{R}^{m} \times S^{N-1}\to[0, +\infty[$ satisfying $(G_{1})-(G_{3})$.
Let $F$ be the functional defined in \eqref{FG}. Then for every $(v,u) \in
L^{1}(\Omega;\mathbb{R}^{m})\times L^{1}(\Omega;\mathbb{R}^{d}) $
\[
\mathcal{F}(v,u; \Omega)=\left\{
\begin{array}
[c]{ll}%
\overline{F_{0}}(v,u;\Omega) & \hbox{ if }(v,u)\in SBV_{0}(\Omega;\mathbb{R}^{m}) \times
BV(\Omega;\mathbb{R}^{d}),\\
\\
+\infty & \hbox{ otherwise.}
\end{array}
\right.
\]

\end{theorem}

The paper is organized as follows. Section \ref{pre} is devoted to preliminary results dealing with functions of bounded variation, perimeters and special functions of bounded variation which are piecewise constant. The properties of the energy densities and several auxiliary results involved in the proofs of representation Theorems \ref{mainthm} and \ref{mainthmgen} are discussed in Section \ref{auxres}. The proof of the lower bound for ${\cal F}$ in \eqref{calFG} is presented in Sections \ref{lb}, while Section \ref{ub} contains the upper bound and the proof of Theorem \ref{mainthmgen}. The applications to optimal design problems as in \cite{AB} and the comparison with previous related relaxation results as in \cite{FM2},  such as Theorem \ref{mainthm}, are discussed in Section \ref{appl}.


\section{Preliminaries}\label{pre}

We give a brief survey of functions of bounded variation and sets of finite perimeter.

In the following $\Omega\subset\mathbb{R}^{N}$ is an open bounded set and we
denote by $\mathcal{A}\left(  \Omega\right)  $ the family of all open subsets of
$\Omega$. The $N$-dimensional Lebesgue measure is designated as
$\mathcal{L}^{N}$, while $\mathcal{H}^{N-1}$ denotes the $\left(  N-1\right)
$-dimensional Hausdorff measure. The unit cube in $\mathbb{R}^{N}$, $\left(
-\frac{1}{2},\frac{1}{2}\right)  ^{N}$, is denoted by $Q$ and we set $Q\left(
x_{0},\varepsilon\right)  :=x_{0}+\varepsilon Q$ for $\varepsilon>0$. For every $\nu \in S^{N-1}$ we
define $Q_{\nu}:=R_{\nu}\left(  Q\right)  $, where $R_{\nu}$ is a rotation
such that $R_{\nu}\left(  e_{N}\right)  =\nu$. The constant $C$ may vary from
line to line.

\label{perimeterBV}


We denote by $\mathcal{M}(\Omega)$ the space of all signed Radon measures in
$\Omega$ with bounded total variation. By the Riesz Representation Theorem,
$\mathcal{M}(\Omega)$ can be identified to the dual of the separable space
$\mathcal{C}_{0}(\Omega)$ of continuous functions on $\Omega$ vanishing on the
boundary $\partial\Omega$. If $\lambda\in\mathcal{M}(\Omega)$ and $\mu
\in\mathcal{M}(\Omega)$ is a nonnegative Radon measure, we denote by
$\frac{d\lambda}{d \mu}$ the Radon-Nikod\'{y}m derivative of $\lambda$ with respect
to $\mu$. 


The following version of Besicovitch Differentiation Theorem was proven by Ambrosio and Dal Maso \cite[Proposition 2.2]{Ambrosio-Dal Maso}.
\begin{theorem}\label{thm2.6BBBF}
If $\lambda$ and $\mu$ are Radon measures in $\Omega$, $\mu \geq 0$, then there exists a Borel measure set $E \subset \Omega$ such that $\mu(E)=0$, and for every $x \in {\rm supp}\mu-E$
$$
\displaystyle{\frac{d \lambda}{d \mu}(x):= \lim_{\e \to 0^+}\frac{\lambda (x+ \e C)}{\mu (x+ \e C)}}
$$
exists and is finite whenever $C$ is a bounded, convex, open set containing the origin.
\end{theorem}

We recall that the exceptional set E above does not depend on C. An immediate corollary is the generalization of Lebesgue-Besicovitch Differentiation
Theorem given below.
\begin{theorem}\label{thm2.8FM2}
If $\mu$ is a nonnegative Radon measure and if $f \in L^1_{\rm loc}(\mathbb R^N,\mu)$ then
$$
\lim_{\e \to 0^+} \frac{1}{\mu(x+ \e C)}\int_{x+ \e C} | f(y) - f ( x ) | d\mu(y) =0
$$
for $\mu$- a.e. $ x\in \mathbb R^N$ and for every, bounded, convex, open set $C$ containing the
origin.
\end{theorem}

\begin{definition}
A function $w\in L^{1}(\Omega;{\mathbb{R}}^{d})$ is said to be of
\emph{bounded variation}, and we write $w\in BV(\Omega;{\mathbb{R}}^{d})$, if
all its first distributional derivatives $D_{j}w_{i}$ belong to $\mathcal{M}%
(\Omega)$ for $1\leq i\leq d$ and $1\leq j\leq N$.
\end{definition}

The matrix-valued measure whose entries are $D_{j}w_{i}$ is denoted by $Dw$
and $|Dw|$ stands for its total variation.
We observe that if $w\in BV(\Omega;\mathbb{R}^{d})$ then $w\mapsto|Dw|(\Omega)$ is lower
semicontinuous in $BV(\Omega;\mathbb{R}^{d})$ with respect to the
$L_{\mathrm{loc}}^{1}(\Omega;\mathbb{R}^{d})$ topology.

We briefly recall some facts about functions of bounded variation. For more
details we refer the reader to \cite{AFP2}, \cite{EG}, \cite{G} and \cite{Z}.

%


\begin{definition}
\label{def3.14AFPstrict} Let $w, w_{n} \in BV(\Omega;\mathbb{R}^{d})$. The
sequence $\{w_{n}\}$ strictly converges in $BV(\Omega;\mathbb{R}^{d})$ to $w$
if $\{w_{n}\}$ converges to $w$ in $L^{1}(\Omega;\mathbb{R}^{d})$ and $\{|D
w_{n}|(\Omega)\}$ converges to $|D w|(\Omega)$ as $n \to\infty$.
\end{definition}

\begin{definition}
Given $w\in BV\left(  \Omega;\mathbb{R}^{d}\right)  $ the \emph{approximate
upper}\textit{\ }\emph{limit }and the \emph{approximate lower limit} of each
component $w^{i}$, $i=1,\dots,d$, are defined by
\[
\left(  w^{i}\right)  ^{+}\left(  x\right)  :=\inf\left\{  t\in\mathbb{R}%
:\,\lim_{\varepsilon\rightarrow0^{+}}\frac{\mathcal{L}^{N}\left(  \left\{
y\in\Omega\cap Q\left(  x,\varepsilon\right)  :\,w^{i}\left(  y\right)
>t\right\}  \right)  }{\varepsilon^{N}}=0\right\}
\]
and
\[
\left(  w^{i}\right)  ^{-}\left(  x\right)  :=\sup\left\{  t\in\mathbb{R}%
:\,\lim_{\varepsilon\rightarrow0^{+}}\frac{\mathcal{L}^{N}\left(  \left\{
y\in\Omega\cap Q\left(  x,\varepsilon\right)  :\,w^{i}\left(  y\right)
<t\right\}  \right)  }{\varepsilon^{N}}=0\right\}  ,
\]
respectively. The \emph{jump set}\textit{\ }of $w$ is given by
\[
J_{w}:=\bigcup_{i=1}^{d}\left\{  x\in\Omega:\,\left(  w^{i}\right)
^{-}\left(  x\right)  <\left(  w^{i}\right)  ^{+}\left(  x\right)  \right\}
.
\]

\end{definition}

It can be shown that $J_{w}$ and the complement of the set of Lebesgue points
of $w$ differ, at most, by a set of $\mathcal{H}^{N-1}$ measure zero.
Moreover, $J_{w}$ is $\left(  N-1\right)  $-rectifiable, i.e., there are
$C^{1} $ hypersurfaces $\Gamma_{i}$ such that 
$\mathcal{H}^{N-1}\left(  J_{w}\setminus\cup_{i=1}^{\infty}\Gamma_{i}\right)=0.$

\begin{proposition}\label{thm2.3BBBF}
If $w\in BV\left(  \Omega;\mathbb{R}^{d}\right)  $ then

\begin{enumerate}
\item[i)] for $\mathcal{L}^{N}-$a.e. $x\in\Omega$%
\begin{equation}
\lim_{\varepsilon\rightarrow0^{+}}\frac{1}{\varepsilon}\left\{  \frac
{1}{\mathcal{\varepsilon}^{N}}\int_{Q\left(  x,\varepsilon\right)
}\left\vert w(y)  -w(x)  -\nabla w\left(
x\right)  \cdot( y-x)  \right\vert ^{\frac{N}{N-1}%
}dy\right\}  ^{\frac{N-1}{N}}=0; \label{approximate differentiability}%
\end{equation}

\item[ii)] for $\mathcal{H}^{N-1}$-a.e. $x\in J_{w}$ there exist $w^{+}\left(
x\right)  ,$ $w^{-}\left(  x\right)  \in\mathbb{R}^{d}$ and $\nu\left(
x\right)  \in S^{N-1}$ normal to $J_{w}$ at $x,$ such that
\[
\lim_{\varepsilon\rightarrow0^{+}}\frac{1}{\varepsilon^{N}}\int_{Q_{\nu}%
^{+}\left(  x,\varepsilon\right)  }\left\vert w\left(  y\right)  -w^{+}\left(
x\right)  \right\vert dy=0,\qquad\lim_{\varepsilon\rightarrow0^{+}}\frac
{1}{\varepsilon^{N}}\int_{Q_{\nu}^{-}\left(  x,\varepsilon\right)  }\left\vert
w\left(  y\right)  -w^{-}\left(  x\right)  \right\vert dy=0,
\]

where $Q_{\nu}^{+}\left(  x,\varepsilon\right)  :=\left\{  y\in Q_{\nu}\left(
x,\varepsilon\right)  :\,\left\langle y-x,\nu\right\rangle >0\right\}  $ and
$Q_{\nu}^{-}\left(  x,\varepsilon\right)  :=\left\{  y\in Q_{\nu}\left(
x,\varepsilon\right)  :\,\left\langle y-x,\nu\right\rangle <0\right\}  $;

\item[iii)] for $\mathcal{H}^{N-1}$-a.e. $x\in\Omega\backslash J_{w}$%
\[
\lim_{\varepsilon\rightarrow0^{+}}\frac{1}{\mathcal{\varepsilon}^{N}}%
\int_{Q\left(  x,\varepsilon\right)  }\left\vert w(y)
-w\left(  x\right)  \right\vert dy=0.
\]

\end{enumerate}
\end{proposition}

We observe that in the vector-valued case in general $\left(  w^{i}\right)
^{\pm}\neq\left(  w^{\pm}\right)  ^{i}.$ In the sequel $w^{+}$ and $w^{-}$
denote the vectors introduced in $ii)$ above.

Choosing a normal $\nu_{w}\left(  x\right)  $ to $J_{w}$ at $x,$ we denote
the \emph{jump} of $w$ across $J_{w}$ by $\left[  w\right]  :=w^{+}-w^{-}.$
The distributional derivative of $w\in BV\left(  \Omega;\mathbb{R}^{d}\right)
$ admits the decomposition
\[
Dw=\nabla w\mathcal{L}^{N}\lfloor\Omega+\left(  \left[  w\right]  \otimes
\nu_{w}\right)  \mathcal{H}^{N-1}\lfloor J_{w}+D^{c} w,
\]
where $\nabla w$ represents the density of the absolutely continuous part of
the Radon measure $Dw$ with respect to the Lebesgue measure. The
\emph{Hausdorff}, or \emph{jump}, \emph{part} of $Dw$ is represented by
$\left(  \left[  w\right]  \otimes\nu_{w}\right)  \mathcal{H}^{N-1}\lfloor
J_{w}$ and $D^{c} w $ is the \emph{Cantor part} of $Dw$. The measure $D^{c} w
$ is singular with respect to the Lebesgue measure and it is diffuse, i.e.,
every Borel set $B\subset\Omega$ with $\mathcal{H}^{N-1}\left(  B\right)
<\infty$ has Cantor measure zero.

The following result, that will be exploited in the sequel, can be found in \cite[Lemma 2.6]{FM2}.
\begin{lemma}\label{lemma2.5BBBF}
Let $w \in BV(\Omega;\mathbb R^d)$, for ${\cal H}^{N-1}$ a.e. $x$ in $J_w$,
$$
\displaystyle{\lim_{\e \to 0^+} \frac{1}{\e^{N-1}} \int_{J_w \cap  Q_{\nu(x)}(x, \e)} |w^+(y)- w^-(y)| d {\cal H}^{N-1} = |w^+(x)- w^-(x)|.}
$$
\end{lemma}

In the following we give some preliminary notions related with sets of finite
perimeter. For a detailed treatment we refer to \cite{AFP2}.

\begin{definition}
\label{Setsoffiniteperimeter} Let $E$ be an $\mathcal{L}^{N}$- measurable
subset of $\mathbb{R}^{N}$. For any open set $\Omega\subset\mathbb{R}^{N}$ the
perimeter of $E$ in $\Omega$, denoted by $P(E;\Omega)$, is the variation of
$\chi_{E}$ in $\Omega$, i.e.
\begin{equation}
\label{perimeter}P(E;\Omega):=\sup\left\{  \int_{E} \mathrm{div}\varphi dx:
\varphi\in C^{1}_{c}(\Omega;\mathbb{R}^{d}), \|\varphi\|_{L^{\infty}}%
\leq1\right\}  .
\end{equation}
We say that $E$ is a set of finite perimeter in $\Omega$ if $P(E;\Omega) <+
\infty.$
\end{definition}

Recalling that if $\mathcal{L}^{N}(E \cap\Omega)$ is finite, then $\chi_{E}
\in L^{1}(\Omega)$, by \cite[Proposition 3.6]{AFP2}, it results
that $E$ has finite perimeter in $\Omega$ if and only if $\chi_{E} \in
BV(\Omega)$ and $P(E;\Omega)$ coincides with $|D\chi_{E}|(\Omega)$, the total
variation in $\Omega$ of the distributional derivative of $\chi_{E}$.
Moreover,  a
generalized Gauss-Green formula holds:
\begin{equation}\nonumber
{\int_{E}\mathrm{div}\varphi dx=\int_{\Omega}<\nu_{E},\varphi>d|D\chi
_{E}|\;\;\forall\,\varphi\in C_{c}^{1}(\Omega;\mathbb{R}^{d})},
\end{equation}
where $D\chi_{E}=\nu_{E}|D\chi_{E}|$ is the polar decomposition of $D\chi_{E}$.

We also recall that, when dealing with sets of finite measure, a sequence of
sets $\{E_{n}\}$ converges to $E$ in measure in $\Omega$ if $\mathcal{L}%
^{N}(\Omega\cap(E_{n}\Delta E))$ converges to $0$ as $n\rightarrow\infty$,
where $\Delta$ stands for the symmetric difference. Analogously, the local
convergence in measure corresponds to the above convergence in measure for any
open set $A\subset\subset\Omega$. These convergences are equivalent to $L^{1}(\Omega)$ and $L^{1}%
_{\mathrm{loc}}(\Omega)$ convergences of the characteristic functions. We also
remind that the local convergence in measure in $\Omega$ is equivalent to
convergence in measure in domains $\Omega$ with finite measure.








Denoting by ${\cal P}(\Omega)$ the family of all sets with finite perimeters in $\Omega$ we recall the Fleming-Rishel formula (see \cite[formula 4.59]{F}):
for every $\Phi \in W^{1,1}(\Omega)$ the set
$
\{t \in \mathbb R: \{\Phi >t\} \not \in {\cal P}(\Omega)\}
$
is negligible in $\mathbb R$ and
\begin{equation}\label{FR}
\displaystyle{\int_\Omega h |\nabla \Phi|dx = \int_{-\infty}^{+ \infty} \int_{\partial^\ast \{\Phi >t\}} h d {\cal H}^{N-1}dt}
\end{equation}
for every  bounded Borel function $h :\Omega \to \mathbb R$, where $\partial ^\ast \{\Phi >t\}$ denotes the essential boundary of $\{\Phi >t\}$ (cf. \cite[Definition 3.60]{AFP2}).

At this point we deal with functions of bounded variation whose Cantor part is null.

\begin{definition}
\label{SBV} A function $v \in BV(\Omega;\mathbb{R}^{m})$ is said to be a
special function of bounded variation, and we write $v \in SBV(\Omega
;\mathbb{R}^{m})$, if $D^{c} v={\underline{0}}$,  i.e.
\[
Dv=\nabla v\mathcal{L}^{N}\lfloor\Omega+ ([v]\otimes\nu_{v})\mathcal{H}%
^{N-1}\lfloor J_{v}.
\]
\end{definition}

The space $SBV_{0}(\Omega;\mathbb{R}^{m})$ is defined by
\begin{equation}\label{SBV0}
SBV_{0}(\Omega;\mathbb{R}^{m}):=\left\{ v \in SBV(\Omega;\mathbb{R}^{m}):
\nabla v=0, \hbox{ and } \mathcal{H}^{N-1}(J_{v})< + \infty\right\} .
\end{equation}
\noindent Clearly, any characteristic function of a set of finite perimeter is
in $SBV_{0}(\Omega)$.

We recall that a sequence of sets $\{E_i\}$ is a Borel partition of a Borel set $B \in {\cal B}(\mathbb R^N)$ if and only if 
$$
E_i \in{\cal B}(\mathbb R^N) \hbox{ for every }i,  E_i \cap E_j = \emptyset \hbox{ for every } i \not= j \hbox{ and } \cup_{i=1}^\infty E_i =B.
$$ 

The above requirements could be weakened requiring that $|E_i \cap E_j|=0$, for $i \not =j$ and $|B \Delta \cup_{i=1}^\infty E_i|=0$.
Such a sequence $\{E_i\}$ is said to be a Caccioppoli partition if and only if each $E_i$ is a set of finite perimeter.

The following result, whose proof can be found in \cite{CT}, expresses the relations between Caccioppoli partitions and $SBV_0$ functions.

\begin{lemma}\label{lemma31BCD}
If $v \in SBV_0(\Omega;\mathbb R^m)$ then there exist a Borel partition $\{E_i\}$ of $\Omega$ and a sequence $\{v_i\}\subset \mathbb R^m$ such that
$$
v=\sum_{i=1}^\infty v_i \chi_{E_i} \hbox{ a e. } x \in \Omega,
$$
$${\cal H}^{N-1}(J_v \cap \Omega)=\frac{1}{2}\sum_{i=1}^\infty {\cal H}^{N-1}(\partial^\ast E_i \cap \Omega)=\frac{1}{2}\sum_{i\not = j=1}^\infty {\cal H}^{N-1}(\partial^\ast E_i \cap \partial^\ast E_j\cap \Omega),$$
$$ (v^+, v^-, \nu_{v}) \equiv (v^i, v^j, \nu_i) \hbox{ a.e. } x \in \partial^\ast E_i \cap \partial E^\ast_j \cap \Omega,$$
 $\nu_i$ being the unit normal to $\partial^\ast E_i \cap \partial E^\ast_j $,
\end{lemma}

\medskip

In the sequel we identify $(v,u)\in SBV_{0}(\Omega;\mathbb{R}^{m})\times
BV(\Omega;\mathbb{R}^{d})$ with their precise representatives $({\tilde v},
{\tilde u}). $ See \cite[Definition 3.63 and Corollary 3.80]{AFP2} for the definition.

\begin{remark}
\label{vmeas} Since $SBV_{0}(\Omega;\mathbb{R}^{m}) \subset BV(\Omega
;\mathbb{R}^{m})$, then $(v,u)\in BV(\Omega;\mathbb{R}^{m+d})$ for every
$(v,u) \in SBV_{0}(\Omega;\mathbb{R}^{m}) \times BV(\Omega;\mathbb{R}^{d})$.
Thus $(v,u)$ is $|D^{c} (v,u)|$-measurable, and since $D^{c}(v,u)=
({\underline{0}}, D^{c} u)$, we may say that $v$ is $|D^{c} u|$-measurable.

\end{remark}

The
following compactness result for bounded sequences in $SBV(\Omega;
\mathbb{R}^m)$ is due to Ambrosio (see \cite{A1}, \cite{A2}).

\begin{theorem}
\label{Theorem 2.1} Let $\Phi:[0,+\infty) \to[0,+\infty)$, $\Theta:(0,+\infty]
\to(0,+\infty]$ be two functions, respectively convex and concave, and such
that
\[
\lim_{t\to\infty} \frac{\Phi(t)}{t} = +\infty, \quad\Phi\text{ is
nondecreasing},
\]
\[
\Theta(+\infty) = \lim_{t\to\infty} \Theta(t), \quad\lim_{t\to0^{+}}
\frac{\Theta(t)}{t} = +\infty, \quad\Theta\text{ is non decreasing}.
\]

Let $\{v_n\}$ be a sequence of functions in $SBV(\Omega;\mathbb{R}^m)$
such that
\[
\sup_{n} \left\{ \int_{\Omega} \Phi(|\nabla v_{n}|) \, dx + \int_{J_{v_n}}
\Theta(|[v_n]|)\, d\mathcal{H}^{N-1} + \int_{\Omega} |v_n|\, dx\right\}  <
+\infty.
\]
Then there exists a subsequence $\{v_{n_{k}}\}$ converging in $L^{1}%
(\Omega;\mathbb{R}^m)$ to a function $v \in SBV(\Omega;\mathbb{R}^m)$,
and
\[
\nabla v_{n_{k}} \rightharpoonup\nabla v \quad\text{in} \; L^{1}%
(\Omega;\mathbb{R}^{N\times m}),\quad[v_{n_{k}}]\otimes\nu_{v_{n_{k}}%
}\mathcal{H}^{N-1}\lfloor J_{v_{n_{k}}}\overset{\ast}{\rightharpoonup}
[v]\otimes\nu_v\mathcal{H}^{N-1}\lfloor J_v,
\]
\[
\int_{J_v\cap\Omega} \Theta(|[v]|)\, d\mathcal{H}^{N-1} \leq\liminf
_{n\to+\infty} \int_{J_{v_{n}}\cap\Omega} \Theta(|[v_{n}]|) \, d\mathcal{H}%
^{N-1}.
\]

\end{theorem}

\section{Auxiliary results}\label{auxres}

This section is mainly devoted to describe the properties of the energy densities involved in the integral representation of relaxed functionals \eqref{representation} and \eqref{representationFG}.

Recall that a Borel function $f:\mathbb{R}^{m}\times\mathbb{R}^{d\times
N}\rightarrow\left[  -\infty,+\infty\right]  $ is said to be quasiconvex if
\begin{equation}
f\left( q, z\right)  \leq\frac{1}{\mathcal{L}^{N}\left(  \Omega\right)  }%
\int_{\Omega}f\left( q, z+\nabla\varphi\left(  y\right)  \right)  dy
\label{qcx}%
\end{equation}
for every open bounded set $\Omega\subset\mathbb{R}^{N}$ with $\mathcal{L}%
^{N}\left(  \partial\Omega\right)  =0,$ for every $(q, z)\in\mathbb{R}^{m}
\times\mathbb{R}^{d\times N}$ and every $\varphi\in W_{0}^{1,\infty}\left(
\Omega;\mathbb{R}^{d}\right)  $ whenever the right hand side of $\left(
\ref{qcx}\right)  $ exists as a Lebesgue integral.

The quasiconvex envelope of $f:\mathbb{R}^{m}\times\mathbb{R}^{d\times
N}\rightarrow\left[  0,+\infty\right]  $ is the largest quasiconvex function
below $f$ and it is denoted by $Qf.$ If $f$ is Borel and locally bounded from below
then it can be shown that%

\begin{equation}
Qf\left(  q, z\right)  =\inf\left\{  \int_{Q} f\left(  q, z+\nabla
\varphi\right)  dx:\varphi\in W_{0}^{1,\infty}\left(  Q;\mathbb{R}^{d}\right)
\right\}  , \label{Qfbar}%
\end{equation}
for every $(q, z)\in\mathbb{R}^{m} \times\mathbb{R}^{d\times N}$.

The following result guarantees that the properties of $f$ are inherited by $Qf$. Since the proof  develops along the lines as in \cite[Proposition 2.2]{RZ}, in turn inspired by \cite{D}, we omit it.
\begin{proposition}
\label{continuityQfbar} Let $f:\mathbb{R}^{m}\times\mathbb{R}^{d\times
N}\rightarrow[0,+\infty)$ be a function satisfying $(F_{1})-(F_{3})$, and let $Q
f:\mathbb{R}^{m}\times\mathbb{R}^{d \times N}\rightarrow[0,+\infty)$ be its
quasiconvexification, as in \eqref{Qfbar}. Then $Q f$ satisfies $(F_{1}%
)-(F_{3})$.
\end{proposition}

\begin{remark}
\label{propfinfty}
Let $f:\mathbb{R}^{m} \times\mathbb{R}^{d \times N} \to[0, +\infty)$ be a
function satisfying $(F_{1})-(F_{4})$, with $f^{\infty}$ as in \eqref{finfty}.

\noindent(i) Recall that the recession function $f^{\infty}(q, \cdot)$ is 
positively one homogeneous for every $q \in\mathbb{R}^{m}$.

\medskip\noindent(ii) We observe that, if $f$ satisfies the growth condition
$(F_{2}) $, then $\beta'|z|\leq f^{\infty}(q, z)\le\beta|z|$ holds. Moreover, if $f$ satisfies $(F_3)$, then $f^\infty$ satisfies
$|f^\infty(q, z)- f^\infty(q', z)|\leq L|q-q'||z|$, where $L$ is the constant appearing in $(F_3)$.

\medskip\noindent(iii) As showed in \cite[Remark 2.2 (ii)]{FM2}, if a function
$f:\mathbb{R}^{m} \times\mathbb{R}^{d\times N}\longrightarrow[0,+ \infty)$ is
quasiconvex in the last variable and such that $f(q,
z)\le c(1+|z|)$, for some $c>0$, then, its recession function $f^{\infty
}(q, \cdot)$ is also quasiconvex.

\medskip\noindent(iv) A proof entirely similar to \cite[Proposition
3.4]{BZZ} (see also \cite[Proposition 2.6]{RZ}) ensures that for every
$(q, z) \in\mathbb{R}^{m} \times\mathbb{R}^{d \times N}, $ $Q(f^{\infty
})(q, z)= (Qf)^{\infty}(q, z)$, hence we will adopt the notation $Qf^\infty$. In particular if $f$ satisfies $(F_1)-(F_3)$, Proposition \ref{continuityQfbar} guarantees that $Qf^\infty$ is continuous in both variables. Furthermore, for every $q\in \mathbb R^m$, $Qf^\infty(q, \cdot)$ is Lipschitz continuous in the last variable.

\medskip\noindent(v) $(Qf)^{\infty}$ satisfies the analogous condition to
$(F_{4})$. We also observe, as emphasized in \cite{FM2}, that $(F_{4})$ is
equivalent to say that there exist $C >0$ and $\alpha\in(0,1) $ such that
$$
\displaystyle{\left| f^{\infty}(q, z)-f(q, z)\right|
\leq C (1+ |z|^{1-\alpha})}%
$$
for every $(q, z) \in\mathbb{R}^{m} \times\mathbb{R}^{d \times N}.$

An argument entirely similar to \cite[Proposition 2.7]{RZ} ensures that there
exist $\alpha\in(0,1),$ and $C^{\prime}>0$ such that
\[
\displaystyle{\left|  (Qf)^{\infty}(q, z)- Qf(q, z)\right| \leq
C^{\prime}(1+|z|^{1-\alpha})}
\]
for every $(q, z)\in\mathbb{R}^{m}\times\mathbb{R}^{d \times N}.$
\end{remark}


The following proposition, whose proof can be obtained arguing exactly as in \cite[page 132]{BBBF}, establishes the properties of the density $K_3$. 

\begin{proposition}\label{propK3}
Let $f:\mathbb R^m \times \mathbb R^{d \times N}\to [0, +\infty)$ and $g: \mathbb R^m\times \mathbb R^m \times S^{N-1}\to (0, +\infty)$. Let $K_3$ be the function defined in \eqref{K3}.
If $( F_{1}) -( F_{4}) $ and $\left(  G_{1}\right)  -\left(  G_{3}\right)  $
hold then

\begin{enumerate}
\item[a)] $\left\vert K_3\left(  a,b,c,d,\nu\right)  -K_3\left(
a',b',c',d',\nu\right)  \right\vert \leq
C\left(  \left\vert a-a'\right\vert +\left\vert b-b'\right\vert +\left\vert c-c'\right\vert +\left\vert d-d'\right\vert \right)  $ for every $\left(  a,b,c,d,\nu\right)  ,$ $\left(
a',b',c',d',\nu\right)  \in\mathbb R^m  \times\ \mathbb R^m  \times \mathbb R^{d}\times \mathbb R^{d}\times
S^{N-1};$

\item[b)] $\nu\longmapsto K_3\left(  a,b,c,d,\nu\right)  $ is upper
semicontinuous for every $\left(  a,b,c,d\right)  \in \mathbb R^m\times \mathbb R^{m}\times
\mathbb R^{d}\times \mathbb R^{d};$

\item[c)] $K_3$ is upper semicontinuous in $\mathbb R^{m}\times \mathbb R^{m}\times
\mathbb R^{d}\times \mathbb R^{d}\times S^{N-1};$

\item[d)] $K_3\left(  a,b,c,d,\nu\right)  \leq C\left(  \left\vert
a-b\right\vert +\left\vert c-d\right\vert + 1\right)  $ for every $\nu\in
S^{N-1}.$  More precisely, from the growth conditions $(F_2)$, $(G_2)$ and the definition of $K_3$ we have 
$K_3(a,a,c,d,\nu) \leq C(|c-d|),$  $K_3(a,b,c,c,\nu)\leq C(1+|a-b|)$. 

\end{enumerate}
\end{proposition}


\noindent A Borel measurable function
$g :\mathbb  R^m \times \mathbb R^m \times  S^{N-1} \to \mathbb R $ is BV-elliptic (cf. \cite{A3}, \cite{AFP2} and \cite{BFLM}) if for all  $(a, b, \nu) \in \mathbb R^m \times \mathbb R^m \times S^{N-1}$, and for any finite subset $T$ of  $\mathbb R^m$
\begin{equation}\label{Bvellipticity}
\int_{J_{w}\cap Q_{\nu}} g(w^+, w^-, \nu_w) d{\cal H}^{N-1} \geq  g(a, b,\nu)
\end{equation}
for all $w \in BV (Q_\nu; T)$ such that $w = v_0$ on $\partial Q_\nu$, 
where 
\begin{equation}\label{vab}
v_0:= \left\{
\begin{array}{ll}
a \hbox{ if } x \cdot \nu > 0,\\

b \hbox{ if } x\cdot\nu \leq 0.
\end{array}
\right.
\end{equation}

We are in position to provide some approximation results which allow us to reobtain the relaxed functionals and the related energy densities in terms of suitable relaxation procedures. To this end we start by stating a result very similar to \cite[Proposition 3.5]{BBBF} which allows to achieve $K_3$.

\begin{proposition}\label{prop3.5BBBF}
Let $f:\mathbb R^m \times \mathbb R^{d \times N}\to [0,+\infty)$ and $g:\mathbb R^m \times \mathbb R^m \times S^{N-1} \to (0,+\infty)$ be functions such that $(F_1)-(F_4)$ and $(G_1)-(G_3)$ hold, respectively. Let $K_3$ be the function defined in \eqref{K3} and $(v_0, u_0) $ be given by
\begin{equation}\label{v0u0}
v_0(x):=\left\{
\begin{array}{ll}
a \hbox{ if }x \cdot \nu >0,\\
b \hbox{ if } x \cdot \nu < 0
\end{array}
\right., \;\;\; u_0(x):= \left\{
\begin{array}{ll}
c \hbox{ if }x \cdot \nu >0,\\
d \hbox{ if }x \cdot \nu < 0.
\end{array}
\right.
\end{equation}
Then
$$
\begin{array}{ll}
K_3(a,b,c,d,\nu)&=\displaystyle{\inf_{(v_n, u_n)}\left\{ \liminf_{n\to \infty} \left(\int_{Q_\nu}Qf^\infty(v_n(x), \nabla u_n(x))dx\right.\right.}\displaystyle{+\left. \int_{Q_\nu \cap J_{v_n}}g(v_n^+(x), v_n^-(x), \nu_n(x)) d{\cal H}^{N-1}\right):}
\\
\\
&\displaystyle{ (v_n,u_n) \in SBV_0(Q_\nu;\mathbb R^m)\times W^{1,1}(Q_\nu;\mathbb R^d), (v_n,u_n)\to (v_0,u_0) \hbox{ in }L^1(Q_\nu;\mathbb R^{m+d}) 
\Big\}}\\
\\
&\displaystyle{=: K_3^\ast(a,b,c,d,\nu).}
\end{array}
$$
\end{proposition}

\begin{remark}\label{applicationofprop3.4}
i) It is worthwhile to observe that the above result ensures a sharper result than the one which is stated, namely the same type of arguments in \cite[Proposition 3.5]{BBBF}  allow us to obtain $K_3(a,b,c,d, \nu)$ as a relaxation procedure but with test sequences in ${\cal A}_3(a,b,c,d,\nu)$, converging to $(v_0, u_0)$ in \eqref{v0u0}.

\noindent ii) Notice that by virtue of the growth conditions on $Qf^\infty$ (cf. Remark \ref{propfinfty}) we can replace in \eqref{A3} the space $W^{1,1}(Q_\nu;\mathbb R^d)$ by $W^{1,\infty}(Q_\nu;\mathbb R^d)$.

\noindent iii) Under assumptions $(G_1)-(G_3)$, the function  $K_3$ in \eqref{K3} can be obtained,  either taking  test functions $v$ in $BV(\Omega; T)$ for every $T\subset \mathbb R^m$, with ${\rm card}(T)$ finite, or in $SBV_0(\Omega;\mathbb R^m) \cap L^\infty(\Omega;\mathbb R^m)$.  
This is easy to verify by virtue of Lemma \ref{lemma31BCD}. Namely, one can approximate functions $v$ in $SBV_0(\Omega;\mathbb R^m) \cap L^\infty(\Omega;\mathbb R^m)$ by sequences $\{v_n\}$ in $BV(\Omega; T_n)$ with $T_n \subset \mathbb R^m$ and  ${\rm card}(T_n)$ finite. Moreover $(v_n^+, v_n^-, \nu_{v_n}) \to (v^+, v^-, \nu_v)$ pointwise and we can apply reverse Fatou's lemma to obtain the equivalence between the two possible definitions of $K_3$.

\noindent iv) Observe that the properties of $K_3$ and the assumptions on $f$ and $g$ allow us to replace in the definition of ${\cal A}_3$ (see formula \eqref{A3}) the set $SBV_0(Q;\mathbb R^m)\cap L^\infty(\Omega;\mathbb R^m)$ by $SBV_0(\Omega;\mathbb R^m)$.
\end{remark}

By the proposition below one can replace in \eqref{calFG}, $f$ by its quasiconvexification $Qf$. We will omit the proof, which is quite standard, exploiting the relaxation results in the Sobolev spaces, cf. \cite[Theorem 9.8]{D}. 
\begin{proposition}
\label{propqcx} Let $\Omega\subset\mathbb{R}^{N}$ be a bounded open set,
$f$ and $g$ be as in Theorem \ref{lsctheorem}, $Qf$ as in $\left(  \ref{Qfbar}\right)  $ and let
$\mathcal{F}$ be given by \eqref{calFG} . Then for every
 $A\in
\mathcal{A}\left(  \Omega\right)  $ and  for every $\left(  v,u\right)  \in SBV_0\left( A;\mathbb R^m  \right)
\times BV\left( A;\mathbb{R}^{d}\right),$ 
\[%
\begin{array}
[c]{c}%
\mathcal{F}\left(  v,u;A\right)  =\inf\left\{  \underset{n\rightarrow
\infty}{\lim\inf}%
{\displaystyle\int_{A}}
Qf\left(  {v_{n},\nabla u_{n}}\right)  {dx+} \displaystyle{\int_{A \cap J_{v_n}}}g(v^+_n, v^-_n, \nu_n)d {\cal H}^{N-1} : \right.\\
\quad \quad \quad \quad \left\{ (v_n, u_{n})\right\}  \subset SBV_0(A; \mathbb R^m) \times W^{1,1}\left(
A;\mathbb{R}^{d}\right), \left. ~(v_n,u_{n})\rightarrow (v, u)\text{ in
}L^1(A;\mathbb R^m) \times L^{1}\left(  A;\mathbb{R}^{d}\right) 
\right\}  .
\end{array}
\]

\end{proposition}

The following result is analogous to \cite[Proposition 2.4]{FM1} and it is devoted to replace the test functions in \eqref{calFG} by smooth ones.  We will omit the proof, and just observe that i) follows the arguments in \cite{AF} with the application of Morse's measure covering theorem (c.f. \cite[Theorem 1.147]{FL}) .

\begin{proposition}\label{prop2.4FM1}
Let $f:\mathbb R^m \times \mathbb{R}^{d\times N}\rightarrow [0, +\infty]$  
 be a function satisfying  $(F_1)-(F_3)$ 
and let $Qf$ be given by \eqref{Qfbar}.

\begin{enumerate}
\item[i)] Let $B$ be a ball in $\mathbb{R}^{N}.$ If
\begin{equation}
{\overline F_0}(v, u;B)\leq\underset{n\rightarrow
\infty}{\lim\inf}\left(\int_{B}Qf\left(  v_n,\nabla u_n\right)  dx + \int_{J_{v_n}\cap B}g(v_n^+, v_n^-, \nu_{v_n})d {\cal H}^{N-1}\right)
\label{lsc}%
\end{equation}
holds for every $(v_n,u_{n}), (v,u) \in SBV_0(\Omega;\mathbb R^m)\times W^{1,1}\left(  \Omega;\mathbb{R}^{d}\right)  $
 such
that $(v_n,u_n)\rightarrow (v,u)$ in $L^1\left(  \Omega;\mathbb R^m \right) \times L^{1}\left(  \Omega;\mathbb{R}^{d}\right) $ then it holds for all open bounded sets $\Omega
\subset\mathbb{R}^{N}.$

\item[ii)] For every $(v,u) \in L^1(\Omega;\mathbb R^m)\times L^1(\Omega;\mathbb R^d)$, $\{(v_n,u_n)\} \subset SBV_0( \Omega;\mathbb R^m) \times W^{1,1}(\Omega;\mathbb{R}^d)$
such that $(v_n, u_{n})\rightarrow (v,u)$ in $L^1(\Omega;\mathbb R^m)\times L^{1}\left(  \Omega;\mathbb{R}^{d}\right)$ 
there exists $\left\{  ({\widetilde v}_{n}, \widetilde{u}_{n})\right\}  \subset C_{0}^{\infty}\left(  \mathbb{R}^{N};\mathbb{R}^m\right)   \times   
C_{0}^{\infty}\left(  \mathbb{R}^{N};\mathbb R^d\right)  $ such that $(\widetilde{v}_n, \widetilde{u}%
_{n})\rightarrow (v,u)$ strictly in $BV(\Omega;\mathbb R^m) \times BV\left(  \Omega;\mathbb{R}^{d}\right)  $ and

\[
\displaystyle{\liminf_{n \to \infty} \int_{\Omega}Qf\left(
\widetilde{v}_n,\nabla\widetilde{u}_n\right) dx
=\liminf_{n\to \infty}\int_\Omega Qf(v_n,\nabla u_n) dx.} 
\]
\end{enumerate}
\end{proposition}

In order to achieve the integral representation in \eqref{mainthmgen} for the jump part, we need to modify
$\left\{  \left(  v_{n},u_{n}\right)  \right\}  $ to match the boundary in
such a way the new sequences will be in $\mathcal{A}_3\left(  v^+(x),v^-(x),u^{+}\left(
x_{0}\right)  ,u^{-}\left(  x_{0}\right)  ,\nu\left(  x_{0}\right)  \right)  $
given in  \eqref{A3}, and the energy doesn't increase. This is achieved in the next Lemma that for sake of simplicity is
stated in the unit cube $Q\subset\mathbb{R}^{N}$, and with the normal to the jump set $\nu=e_N$. The proof relies on the techniques of \cite[Lemma 3.5]{BDV}, \cite[Lemma 3.1]{FM2} and \cite[Lemma 4.4]{ABr1}.

\begin{lemma}
\label{Lemma4.1FM}Let $Q:=\left[  0,1\right]  ^{N}$ and%
\[
 v_0\left(  y\right)  :=\left\{
\begin{array}
[c]{lll}%
a &  & \text{if }x_{N}>0,\\
b &  & \text{if }x_{N}< 0,
\end{array}
\right.\qquad u_{0}\left(  y\right)  :=\left\{
\begin{array}
[c]{lll}%
c &  & \text{if }x_{N}> 0,\\
d &  & \text{if }x_{N} <0.
\end{array}
\right. 
\]
Let $\left\{  v_{n}\right\}  \subset
SBV_0\left(  Q;\mathbb R^m  \right) $ and $\{u_{n}\} \subset W^{1,1}\left(
Q;\mathbb{R}^{d}\right)  $, such that $v_n \to v_0$ in $L^{1}\left(
Q;\mathbb R^m  \right)  $ and $u_n\to u_0$  in $L^{1}\left(  Q;\mathbb{R}^{d}\right) .$   

If $\rho$ is a mollifier,
$\rho_{n}:=n^{N}\rho\left(  nx\right)  ,$ then there exists $\left\{  \left(  \zeta_{n},\xi_{n}\right)  \right\}  \in
\mathcal{A}_3\left(  a,b,c,d,e_{N}\right)  $ such that
\[
\zeta_{n}=v_0\text{ on }\partial Q,~\zeta_{n}\rightarrow v_0\text{ in
}L^{1}\left(  Q;\mathbb R^m\right),
\]
\[
\xi_{n}=\rho_{i\left(  n\right)  }\ast u_{0}\text{ on }\partial Q,~~\ \ \xi
_{n}\rightarrow u_{0}\text{ in }L^{1}\left(  Q;\mathbb{R}^{d}\right) 
\]%
and%
\begin{equation}\nonumber
\begin{array}{ll}
\displaystyle{\underset{n\rightarrow\infty}{\lim\sup}\left(  \int_{Q}Qf\left(
\zeta_{n},\nabla\xi_{n}\right)  dx+\int_{J_{\zeta_n}\cap Q}g(\zeta_n^+, \zeta_n^-, \nu_{\zeta_n})d {\cal H}^{N-1}\right)}
\\
\\
\displaystyle{\leq \underset{n\rightarrow\infty}{\lim\inf}\left(  \int_{Q}Qf\left( v_{n},\nabla u_{n}\right)  dx+\int_{J_{v_n}\cap Q}g(v_n^+, v_n^-, \nu_{v_n})d {\cal H}^{N-1}
\right).}
\end{array} 
\end{equation}

\end{lemma}

\begin{proof}
Without loss of generality, we may assume that
\[
\begin{array}{ll}
\displaystyle{\underset{n\rightarrow\infty}{\lim\inf}\left(\int_{Q}Qf\left(  v_{n},\nabla
u_{n}\right)  dx+\int_{J_{v_n}\cap Q}g(v_n^+, v_n^-,\nu_{v_n})d {\cal H}^{N-1}\right)}\\
\\
\displaystyle{=\lim_{n\to \infty}\left(\int_{Q}Qf\left(  v_{n},\nabla
u_{n}\right)  dx+\int_{J_{v_n}\cap Q}g(v_n^+, v_n^-,\nu_{v_n})d {\cal H}^{N-1}\right)  <+\infty.}
\end{array}
\]

The proof is divided in two steps.

\noindent {\bf Step 1.} First we claim that for every $\varepsilon >0$, denoted $ \|(v_0, u_0)\|_{\infty}$ by $M_0$, there exist a sequence $ \{\overline u_n\} \subset W^{1,1}(Q;\mathbb R^d)\cap L^\infty(Q;\mathbb R^d)$ and
a sequence $ \{\overline v_n\} \subset SBV_0(Q;\mathbb R^m)\cap L^\infty(Q;\mathbb R^m)$, and a constant $C>0$ such that $\|{\overline u_n}\|_{\infty}, \|\overline v_n\|_{\infty} \leq C$ for every $n$ and 
\begin{equation}\label{as3.7BBBF}
\begin{array}{ll}
\displaystyle{\underset{n\rightarrow\infty}{\lim\inf}\left(\int_{Q}Qf\left(  {\overline v_{n}},\nabla
\overline u_{n}\right)  dx+\int_{J_{{\overline v_n}}\cap Q}g({\overline v_n}^+, {\overline v_n}^-,\nu_{\overline v_n})d {\cal H}^{N-1}\right)}
\\
\\
\displaystyle{\leq \underset{n\rightarrow\infty}{\lim}\left(\int_{Q}Qf\left(  v_{n},\nabla
u_{n}\right)  dx+\int_{J_{v_n}\cap Q}g(v_n^+, v_n^-,\nu_{v_n})d {\cal H}^{N-1}\right)+ \varepsilon.}
\end{array}
\end{equation} 

To achieve the claim we can apply a truncation argument as in \cite[Lemma 3.5]{BDV}, (c.f. also \cite[Lemma 3.7]{BBBF}). 
For $a_i \in \mathbb R$ to be determined later depending on $\e$ and $M_0$, we define $\phi_i \in W^{1, \infty}_0(\mathbb R^{m+d};\mathbb R^{m+d})$ such that
\begin{equation}\label{Lipschitztruncature}
\phi_i(x)=\left\{
\begin{array}{ll}
x, & |x| < a_i,\\
0, &|x|\geq a_{i+1},
\end{array}
\right.
\end{equation}
$\|\nabla \phi_i\|_{\infty}\leq 1$, with $x \in \mathbb R^{m+d}$, and $x\equiv(x_1,x_2),  x_1 \in \mathbb R^m,  x_2 \in \mathbb R^d$.

For any $n \in \mathbb N$ and for any $i $ as above, let $(v^i_n,u^i_n)\in SBV_0(Q;\mathbb R^m) \times W^{1,1}(Q;\mathbb R^d)\cap L^\infty(Q;\mathbb R^{m+d})$ be  given by
$$
(v^i_n, u^i_n):=\phi_i(v_n,u_n).
$$ 
Considering the bulk part of the energy $F$ in \eqref{FG}, and exploiting Proposition \ref{propqcx} and the growth conditions on $f$ and $Qf$, we have
$$
\begin{array}{ll}
\displaystyle{\int_Q Qf(v^i_n,\nabla u^i_n)dx = \int_{Q \cap \{|(v_n, u_n)|\leq a_i\}}Qf(v_ n,\nabla u_n)dx + \int_{Q\cap \{|(v_n,u_n)|> a_{i+1}\}}Qf(0,0)dx }
\\
\\
\displaystyle{+\int_{Q \cap \{a_i< |(v_n, u_n)|\leq a_{i+1}\}} Qf(v_n^i, \nabla u_n^i)dx }\\
\\
\displaystyle{\leq \int_Q Qf(v_n, \nabla u_n)dx + C |Q \cap \{|(v_n, u_n)|> a_{i+1}\}|+ C_1 \int_{A \cap \{a_i < |(v_n, u_n)|\leq a_{i+1}\}}(1+|\nabla u_n|)dx}. 
\end{array}
$$
Concerning the surface term of the energy in \eqref{FG}, since $((v_n^i)^\pm, (u_n^i)^\pm)=\phi_i(v_n^\pm, u_n^\pm)$, and without loss of generality one can assume that $|(v_n^-, u_n^-)|\leq |(v_n^+,u_n^+)|$ ${\cal H}^{N-1}$- a.e. on $J_{(v_n,u_n)}$,
we have that
$$
\begin{array}{ll}
\displaystyle{\int_{Q\cap J_{v^i_n}}g((v_n^i)^+, (v_n^i)^-, \nu_{v_n^i})d {\cal H}^{N-1}}\\
\\
\displaystyle{\leq \int_{J_{v_n} \setminus \{ a_{i+1} \leq |(v_n^-,u_n^-)|\}\cap Q} g(\phi_i((v_n^i)^+, (u_n^i)^+), \phi_i((v_n^i)^-, (u_n^i)^-), \nu_{(v_n^i, u_n^i)})d {\cal H}^{N-1}.}
\end{array}
$$ 
Arguing as in \cite[Lemma 3.5]{BDV} (cf. also \cite[Remark 3.6]{BDV}), and exploiting the growth conditions on $g$ we can estimate $\displaystyle{\frac{1}{k}\sum_{i=1}^k F(v_n^i, u_n^i; Q)}$ for any fixed $k \in \mathbb N$, and for every $n \in \mathbb N$, with $k$ independent on $n$.
Then
$$
\begin{array}{ll}
\displaystyle{\frac{1}{k}\sum_{i=1}^k F(v_n^i, u_n^i; Q) \leq F(v_n, u_n; Q) + \frac{1}{k} \sum_{i=2}^k\left(C |Q \cap \{|(v_n, u_n)|> a_{i+1}\}|+ C_4 \int_ {J_2^i \cap Q} (1+ |v_n^-|)d {\cal H}^{N-1}\right)}\\
\\
\displaystyle{+\frac{1}{k}\left(c_2 \int_Q (1+|\nabla u_n|)dx + 3 C_4 \int_{J_{v_n}\cap Q} (1+|v_n^+-v_n^-|)d {\cal H}^{N-1}\right),}
\end{array}
$$
where $J^i_2:=\{|v_n^-|\leq a_i, |v_n^+|\geq a_{i+1}\}$.
By the growth conditions there exists  a constant $C$ such that
$$
\displaystyle{\left(c_2 \int_Q (1+|\nabla u_n|)dx + 3 c_4 \int_{J_{v_n}\cap Q} (1+|v_n^+-v_n^-|)d {\cal H}^{N-1}\right)\leq C,}
$$
for every $n \in \mathbb N$. Choose $\displaystyle{k \in \mathbb N}$ such that $\displaystyle{\frac{c}{k}\leq \frac{\varepsilon}{3}}$.
Moreover
$$
\displaystyle{C \geq\int_ {J_2^i \cap Q} |v_n^+- v_n^-|d {\cal H}^{N-1} \geq \int_ {J_2^i \cap Q} (|v_n^+ |- |v_n^-|)d {\cal H}^{N-1} \geq (a_{i+1}-a_i){\cal H}^{N-1}(J_2^i \cap Q)},  
$$
whence
$$
\displaystyle{\int_{J_2^i\cap Q}(1+ |v_n^-|)d{\cal H}^{N-1}\leq C \frac{1+ a_i}{a_{i+1}-a_i}.}
$$
The sequence $\{a_i\}$  can be chosen recursively as follows
$$
\begin{array}{ll}
C_2 |Q \cap \{|(v_n, u_n)| > a_{i}\}|\leq \frac{\varepsilon}{3}, \hbox{ for every }n \in \mathbb N, a_{i+1} \geq M_0,\\
\\
c_4C\frac{1+a_i}{a_{i +1}- a_i} \leq \frac{\varepsilon}{3} \hbox{ for every }i \in \mathbb N,  
\end{array}
$$
which is possible since $\{(v_n, u_n)\}$ is bounded in $L^1$.
Thus we obtain
$$
\frac{1}{k}\sum_{j=1}^k F(v_n^{i_j}, u_n^{i_j}; Q) \leq F(v_n, u_n; Q)+ \varepsilon.
$$
Therefore for every $n \in \mathbb N$ there exists $i(n)\in \{1,\dots, k\}$ such that
$$
F(v_n^{i_n}, u_n^{i_n}; Q) \leq F(v_n, u_n; Q)+ \varepsilon.
$$
 It suffices to define $\overline v_n:= v_n^{i_n}$ and ${\overline u_n}:= u_n^{i_n}$ to achieve \eqref{as3.7BBBF} and observe that $\{\overline u_n\}$ and $\{\overline v_n\}$ are bounded in $L^\infty$, by construction.

\noindent {\bf Step 2.} This step is devoted to the construction of sequences $\{\xi_n\}$ and $\{\zeta_n\}$ as in the statement.  Let ${\overline v_n}$ and ${\overline u_n}$ be as in $i)$. Define
\[
w_{n}\left(  x\right)  :=\left(  \rho_{n}\ast u_{0}\right)  \left(  x\right)
=\int_{B\left(  x,\frac{1}{n}\right)  }\rho_{n}\left(  x-y\right)
u_{0}\left(  y\right)  dy.
\]
As $\rho$ is a mollifier, we have for each tangential direction $i=1,\dots
,N-1,$ $w_{n}\left(  x+e_{i}\right)  =w_{n}\left(  x\right)  $ and so%
\[
w_{n}\left(  y\right)  =\left\{
\begin{array}
[c]{lll}%
c &  & \text{if }x_{N}>\frac{1}{n},\\
 d &  & \text{if }x_{N}<-\frac{1}{n},
\end{array}
\right.  ~\ \ \ \left\Vert \nabla w_{n}\right\Vert _{\infty}=O\left(
n\right)  ,~~\ w_{n}\in\mathcal{A}_1\left( c,d,e_{N}\right),
\] where 
$$
\begin{array}{ll}
{\cal A}_1(c,d,e_N):=\left\{ u\in W^{1,1}(Q_{\nu};\mathbb{R}^{d}): u(y)= c \hbox{ if } y \cdot\nu=\frac
{1}{2},  u(y)= d \hbox{ if } y\cdot\nu=- \frac{1}{2},\right.\\
\\
\left.\hbox{ with }u \;1-\hbox{periodic in }\nu_{1}, \dots,
\nu_{N-1} \hbox{ directions} \right\}. 
\end{array}
$$ 

Let $\alpha_{n}:=\sqrt{\left\Vert {\overline u_{n}}-w_{n}\right\Vert _{L^{1}\left(
Q;\mathbb{R}^{d}\right)  }+\left\Vert {\overline v_{n}}-v_{0}\right\Vert
_{L^{1}\left(  Q\right)  }},~$\newline$k_{n}:=n\left[  1+\left\Vert
{\overline u_{n}}\right\Vert _{W^{1,1}\left(  Q;\mathbb{R}^{d}\right)  }+\left\Vert
w_{n}\right\Vert _{W^{1,1}\left(  Q;\mathbb{R}^{d}\right)  }+\left\Vert
{\overline v_{n}}\right\Vert _{BV\left(  Q\right)  }+\left\Vert v_{0}\right\Vert
_{BV\left(  Q\right)  }+ {\cal H}^{N-1}(J_{{\overline v_n}})\right]  ,~s_{n}:=\frac{\alpha_{n}}{k_{n}}$ where
$\left[  k\right]  $ denotes the largest integer less than or equal to $k.$
Since $\alpha_{n}\rightarrow0^{+},$ we may assume that $0\leq\alpha_{n}<1,$
and set $Q_{0}:=\left(  1-\alpha_{n}\right)  Q,~Q_{i}:=\left(  1-\alpha
_{n}+is_{n}\right)  Q,~i=1,\dots,k_{n}.$

Consider a family of cut-off functions $\varphi_{i}\in C_{0}^{\infty}\left(
Q_{i}\right)  ,$ $0\leq\varphi_{i}\leq1,~\varphi_{i}=1$ in $Q_{i-1}%
,~\left\Vert \nabla\varphi_{i}\right\Vert _{\infty}=O\left(  \frac{1}{s_{n}%
}\right)  $ for $i=1,\dots,k_{n},$ and define
\[
u_{n}^{\left(  i\right)  }\left(  x\right)  :=\left(  1-\varphi_{i}\left(
x\right)  \right)  w_{n}\left(  x\right)  +\varphi_{i}\left(  x\right)
{\overline u_n}\left(  x\right)  .
\]
Since $u_{n}^{\left(  i\right)  }=w_{n}$ on $\partial Q$ we have that
$u_{n}^{\left(  i\right)  }\in\mathcal{A}_1\left(  c,d,e_{N}\right).  $ Clearly,%
\[
\nabla u_{n}^{\left(  i\right)  }=\nabla {\overline u_n}\text{ in }Q_{i-1},\qquad\nabla
u_{n}^{\left(  i\right)  }=\nabla w_{n}\text{ in }Q\backslash Q_{i},
\]
and in $Q_{i}\backslash Q_{i-1}$%
\[
\nabla u_{n}^{\left(  i\right)  }=\nabla w_{n}+\varphi_{i}\left(  \nabla
{\overline u_n}-\nabla w_{n}\right)  +\left(  {\overline u_n}-w_{n}\right)  \otimes\nabla
\varphi_{i}.
\]
For $0<t<1$ define
\[
v_{n,i}^{t}\left(  x\right)  :=\left\{
\begin{array}
[c]{lll}%
v_0\left(  x\right)  &  & \text{if }\varphi_{i}\left(  x\right)  <t,\\
{\overline v_n}\left(  x\right)  &  & \text{if }\varphi_{i}\left(  x\right)  \geq t.
\end{array}
\right.
\]
Clearly, $\lim_{n\rightarrow\infty}\left\Vert v_{n,i}^{t}-v
_{0}\right\Vert _{L^{1}\left(  Q\right)  }=0$ as $n\rightarrow\infty,$ independently on $i$ and $t$. 
For every $n$ and $i,$ by Fleming-Rishel formula  \eqref{FR} it is possible to find
$t_{n,i}\in\left]  0,1\right[  $ such that
\begin{align*}
\left\{  x\in Q:\varphi_{i}\left(  x\right)  <t_{n,i}\right\}   &
\in\mathcal{P}\left(  Q\right)  ,\\
\mathcal{H}^{N-1}\left(  J_{v_{0}}\cap\left\{  x\in Q:\varphi_{i}\left(
x\right)  =t_{n,i}\right\}  \right)   &  =\mathcal{H}^{N-1}\left(  J_{\overline v_n}\cap\left\{  x\in Q:\varphi_{i}\left(  x\right)  =t_{n,i}\right\}
\right)  =0,
\end{align*}
where ${\cal P}(Q)$ denotes the family of sets with finite perimeter in $Q$. 
Let%
\[
v_{n,i}^{t_{n,i}}:=\left\{
\begin{array}
[c]{lll}%
v_0\left(  x\right)  &  & \text{in }Q\cap\left\{  x\in Q:\varphi
_{i}\left(  x\right)  <t_{n,i}\right\}  ,\\
{\overline v_n}\left(  x\right)  &  & \text{in }Q\cap\left\{  x\in Q:\varphi
_{i}\left(  x\right)  \geq t_{n,i}\right\}  .
\end{array}
\right.
\]
Clearly, $\lim_{n\rightarrow\infty}\left\Vert v_{n,i}^{t_{n,i}}-v
_{0}\right\Vert _{L^{1}\left(  Q\right)  }=0$,  $\left\{v^{t_{n,i}}_{n,i}\right\}\subset SBV_0(Q;\mathbb R^m)\cap L^\infty(Q;\mathbb R^m)$ and, from Step 1,  it is uniformly bounded on $n, i$ and $t$.

We have
\begin{align*}
&  \int_{Q}Qf\left( v_{n,i}^{t_{n,i}},\nabla u_{n}^{\left(  i\right)
}\right)  dx+\int_{J_{v_{n,i}^{t_{n,i}}}\cap Q}g( (v_{n,i}^{t_{n,i}})^+,(v_{n,i}^{t_{n,i}})^-, \nu_{v_{n,i}^{t_{n,i}}}) d {\cal H}^{N-1} \\
&  \leq\int_{Q}Qf\left( {\overline v_n},\nabla {\overline u_n}\right)  dx+C\int_{Q_{i}%
\backslash Q_{i-1}}\left(  1+\left\vert {\overline u_n}\left(  x\right)  -w_{n}\left(
x\right)  \right\vert \frac{1}{s_{n}}+\left\vert \nabla {\overline u_n}\left(  x\right)
\right\vert +\left\vert \nabla w_{n}\left(  x\right)  \right\vert \right)
dx\\
&  +C\int_{Q\backslash Q_{i}}\left(  1+\left\vert \nabla w_{n}\left(
x\right)  \right\vert \right)  dx+\int_{
Q\cap\left\{  \varphi_{i}>t_{n,i}\right\}  _{1}} g({\overline v_n}^+, {\overline v_n}^-, \nu_{\overline  v_n})d {\cal H}^{N-1} \\
&  +\left\vert Dv_{n,i}^{t_{n,i}}\right\vert \left(  \left(  Q\cap\left\{
\varphi_{i}>t_{n,i}\right\}  _{0}\right)  \right) + {\cal H}^{N-1} \left(  \left(  Q\cap\left\{
\varphi_{i}>t_{n,i}\right\}  _{0}\right)  \right) +\left\vert Dv
_{n,i}^{t_{n,i}}\right\vert \left(  \partial^{\ast}\left\{  \varphi
_{i}<t_{n,i}\right\}  \right) \\
 & + {\cal H}^{N-1}\left(  \partial^{\ast}\left\{  \varphi
_{i}<t_{n,i}\right\}  \right)\\
&  \leq\int_{Q}Qf\left(  {\overline v_n},\nabla {\overline u_n}\right)  dx+I_{1}+
\int_{Q \cap J_{\overline v_n}} g({\overline v_n}^+, {\overline v_n}^-, \nu_{\overline  v_n})d {\cal H}^{N-1}
 +C \left\vert Dv_0\right\vert
\left(  Q\backslash Q_{i}:\left\{  \varphi_{i}>t_{n,i}\right\}  _{0}\right) \\
& +\frac{C}{s_{n}}\int_{Q_{i}\backslash Q_{i-1}} |{\overline v}_n-v_0|dx+ \frac{1}{s_n}O(s_n),
\end{align*}
where
\[
\left\{  \varphi_{i}>t_{n,i}\right\}  _{1}:=\left\{  x\in Q:\frac{\left\vert
\left\{  x\in Q:\varphi_{i}>t_{n,i}\right\}  \cap B_{\rho}\left(  x\right)
\right\vert }{\left\vert B_{\rho}\left(  x\right)  \right\vert }=1\right\},
\]

\[
\left\{  \varphi_{i}>t_{n,i}\right\}  _{0}:=\left\{  x\in Q:\frac{\left\vert
\left\{  x\in Q:\varphi_{i}>t_{n,i}\right\}  \cap B_{\rho}\left(  x\right)
\right\vert }{\left\vert B_{\rho}\left(  x\right)  \right\vert }=0\right\},
\]
 $I_1:=$  $\displaystyle{C\int_{Q_{i}%
\backslash Q_{i-1}}\left(  1+\left\vert {\overline u_n}\left(  x\right)  -w_{n}\left(
x\right)  \right\vert \frac{1}{s_{n}}+\left\vert \nabla {\overline u_n}\left(  x\right)
\right\vert +\left\vert \nabla w_{n}\left(  x\right)  \right\vert \right)
dx +C\int_{Q\backslash Q_{i}}\left(  1+\left\vert \nabla w_{n}\left(
x\right)  \right\vert \right)  dx}$,
and we have used \eqref{FR}  in the last two terms of the above estimate.

Averaging over all layers $Q_{i}\backslash Q_{i-1}$ 
one obtains
\begin{align*}
&  \frac{1}{k_{n}}\sum_{i=1}^{k_{n}}\left(  \int_{Q}Qf\left(  v
_{n,i}^{t_{n,i}},\nabla u_{n}^{\left(  i\right)  }\right)  dx+
\int_{Q \cap J_{v_{n,i}^{t_{n,i}}}} g((v_{n,i}^{t_{n,i}})^+, (v_{n,i}^{t_{n_i}})^-, \nu_{v_{n_i}^{t_{n,i}}})d {\cal H}^{N-1}\right)
\\
&  \leq\int_{Q}Qf\left(  {\overline v_n},{\overline u_n}\right)  dx
+ \int_{Q \cap J_{v_n}}g({\overline v_n}^+, {\overline v_n}^-, \nu_{\overline v_n})d {\cal H}^{N-1}
+\frac{C}{k_{n}}\int_{Q}\left(  1+\left\vert
\nabla {\overline u_n}\right\vert +\left\vert \nabla {\overline v_n}\right\vert \right)
dx\\
&  +\frac{C}{k_{n}}\int_{Q}\left\vert {\overline u_n}-w_n\right\vert \frac{1}{s_{n}%
}dx+C\int_{Q\backslash Q_{0}}\left(  1+\left\vert \nabla w_n\right\vert
\right)  dx+C\left\vert Dv_{0}\right\vert \left(  Q\backslash Q_{0}\right)+\frac{C}{s_{n}k_{n}}\int_{ Q\backslash
Q_{0}} |{\overline v_n}-v_0|dx  + \frac{C}{k_n}
\\
&  \leq\int_{Q}Qf\left(  {\overline v_n},\nabla {\overline u_n}\right)  dx+
\int_{Q \cap J_{v_n}} g({\overline v_n}^+, {\overline v_n}^-, \nu_{\overline v_n})d{\cal H}^{N-1}
 +\frac{C}{k_{n}}\int_{Q}\left(  1+\left\vert
\nabla {\overline u_n}\right\vert +\left\vert \nabla {\overline v_n}\right\vert \right)
dx\\
&  +\frac{C}{\alpha_{n}}\left\Vert {\overline u_n}-w_n\right\Vert _{L^{1}}+C\int_{Q\backslash Q_{0}}\left(  1+\left\vert \nabla w_n\right\vert
\right)  dx+C\left\vert Dv_{0}\right\vert \left(  Q\backslash
Q_{0}\right) +\frac{C}{\alpha_{n}}\left\Vert {\overline v_n}-v_{0}\right\Vert
_{L^{1}\left(  Q\right)} + \frac{C}{k_n} .
\end{align*}
Since $\left\vert Q\backslash Q_{0}\right\vert =O\left(  \alpha_{n}\right)  $
and $\nabla w_n\left(  x\right)  =0$ if $\left\vert x_{N}\right\vert
>\frac{1}{N}$ we estimate
\[
\int_{Q\backslash Q_{0}}\left(  1+\left\vert \nabla w_n\right\vert \right)
dx\leq O\left(  \alpha_{n}\right)  +\mathcal{H}^{N-1}\left(  Q\backslash
Q_{0}\cap\left\{  x_{N}=0\right\}  \right)  \int_{-\frac{1}{n}}^{\frac{1}{n}%
}O\left(  n\right)  dx_{N}=O\left(  \alpha_{n}\right)  .
\]
The same argument exploited above in order to estimate $\int_{Q \setminus Q_0}dx$ applies to estimate $|Dv_0|(Q\setminus Q_0)$ since $v_0$ is a jump function across $x_N=0$, namely
$|Dv_0|(Q\setminus Q_0)= C \mathcal{H}^{N-1}\left(  Q\backslash
Q_{0}\cap\left\{  x_{N}=0\right\}  \right)$, recalling also that $Q_0 =\alpha_n Q$.

Setting $\varepsilon_{n}:=O\left(  \frac{1}{n}\right)  +C\sqrt{\left\Vert
{\overline u_n}-w_n\right\Vert _{L^{1}\left(  Q;\mathbb{R}^{d}\right)  }+\left\Vert
{\overline v_n}-v_{0}\right\Vert _{L^{1}\left(  Q\right)  }}+O\left(  \alpha
_{n}\right)  $ we have that $\varepsilon_{n}\rightarrow 0^{+}$ and
\begin{align*}
&  \frac{1}{k_{n}}\sum_{i=1}^{k_{n}}\left(  \int_Q Qf\left(  v_{n,i}^{t_{n,i}},\nabla u_n^{(i)}\right)  dx+
\int_{Q \cap J_{v_{n_i}^{t_{n,i}}}} g((v_{n,i}^{t_{n,i}})^+, (v_{n,i}^{t_{n,i}})^-, \nu_{v_{n,i}^{t_{n,i}}})d {\cal H}^{N-1}\right)
 \\
&  \leq\int_{Q}Qf \left(  {\overline v_n},\nabla {\overline u_n}\right)  dx+\int_{Q \cap J_{\overline v_n}}g({\overline v_n}^+, {\overline v_n}^-, \nu_{\overline v_n})d {\cal H}^{N-1}+\varepsilon_{n}
\end{align*}

and so there exists an index $i\left(  n\right)  \in\left\{  1,\dots
,k_{n}\right\}  $ for which
\begin{equation}\nonumber
\begin{array}{ll}
\displaystyle{\int_{Q}Qf\left(  v_{n,i\left(  n\right)  }^{t_{n,i\left(  n\right)  }%
},\nabla u_n^{i\left(  n\right)  }\right)  dx+\int_{Q \cap J_{v_{n_i}^{t_{n,i}}}} g((v_{n,i}^{t_{n,i}})^+, (v_{n,i}^{t_{n,i}})^-, \nu_{v_{n,i}^{t_{n,i}}})d {\cal H}^{N-1}}
\\
\\
\displaystyle{\leq \int_{Q}Qf \left(  {\overline v_n},\nabla {\overline u_n}\right)  dx+\int_{Q \cap J_{\overline v_n}}g({\overline v_n}^+. {\overline v_n}^-, \nu_{\overline v_n})d {\cal H}^{N-1}+\varepsilon_{n}.}
\end{array}
\end{equation}
It suffices to define $\xi_{n}:=u_{n}^{i\left(  n\right)} , \zeta_{n}:=v_{n,i\left(n\right)}^{t_{n,i\left(  n\right) }}$  to get
\begin{equation}\nonumber
\begin{array}{ll}
\displaystyle{\underset{n\rightarrow\infty}{\lim\sup}\left(  \int_{Q}Qf\left(
\zeta_{n},\nabla\xi_{n}\right)  dx+\int_{J_{\zeta_n}\cap Q}g(\zeta_n^+, \zeta_n^-, \nu_{\zeta_n})d {\cal H}^{N-1}\right)}
\\
\\
\displaystyle{\leq \underset{n\rightarrow\infty}{\lim\inf}\left(  \int_{Q}Qf\left( {\overline v_n},\nabla {\overline u_n}\right)  dx+\int_{J_{\overline v_n}\cap Q}g({\overline v_n}^+, {\overline v_n}^-, \nu_{\overline v_n})d {\cal H}^{N-1}
\right),}
\end{array}
\end{equation}
which concludes the proof.
\end{proof}

\begin{remark}\label{asinGlobalMethodBFLM}
\begin{itemize}
\item[i)] Observe that arguing as in the first step of Lemma \ref{Lemma4.1FM}, it results that for every $u \in BV(\Omega;\mathbb R^d)$ and $v \in SBV_0(\Omega;\mathbb R^m)\cap L^\infty(\Omega;\mathbb R^m)$
$$
\begin{array}{ll}
{\cal F}(v, u;A) =&\inf\left\{  \displaystyle{\liminf_{n
\to\infty} \left( \int_{A} f(v_{n}, \nabla u_{n})dx +\int_{J_{v_{n}}\cap A}
g({v_{n}}^{+}, {v_{n}}^{-}, \nu_{v_{n}}) d\mathcal{H}^{N-1}\right) :}\right. \\
\\
 &\{ v_{n}\} \subset SBV_{0}\left(A;\mathbb R^m \right) \cap L^\infty(A;\mathbb R^m),  
\left\{
u_{n}\right\}  \subset W^{1,1}\left(  A;\mathbb{R}^{d}\right) ,\\
\\
& (v_n,u_n) \to (v, u)\text{ in }L^1\left(
A;\mathbb{R}^{m+d}\right), \, \sup_n \|v_n\|_{\infty}< +\infty  \Big\}.
\end{array}
$$

\item[ii)] Similarly, if also $u\in BV(\Omega;\mathbb R^d)\cap L^\infty(\Omega;\mathbb R^d)$, then
 $$
\begin{array}{ll}
{\cal F}(v,u;A) =&\inf\left\{  \displaystyle{\liminf_{n
\to\infty} \left( \int_{A} f(v_{n}, \nabla u_{n})dx +\int_{J_{v_{n}}\cap A}
g({v_{n}}^{+}, {v_{n}}^{-}, \nu_{v_{n}}) d\mathcal{H}^{N-1}\right) :}\right. \\
\\
& \{ v_{n}\} \subset SBV_{0}\left(A;\mathbb R^m \right) \cap L^\infty(A;\mathbb R^m), \left\{
 u_{n}\right\}  \subset W^{1,1}\left(  A;\mathbb{R}^{d}\right)  \cap L^\infty(A;\mathbb R^d),
 \\
  \\
&(v_n,u_n) \to (v, u)\text{ in }L^1\left(
A;\mathbb{R}^{m+d}\right),  \sup_n \|(v_n, u_n)\|_{\infty}< +\infty  \Big\}.
\end{array}
$$
\item[iii)] Notice that an argument entirely similar to \cite[Lemmas 13 and 14]{BFLM} allows us to say that for every $(v,u)\in SBV_0(\Omega;\mathbb R^m)\times BV(\Omega;\mathbb R^d)$, it results
$$
\displaystyle{{\cal F}(v,u; A)=\lim_{j
\to\infty}{\cal F}(\phi_j(v,u);A)},
$$ where $\phi_j$ are the functions defined in \eqref{Lipschitztruncature}.
\end{itemize}
\end{remark}

We conclude this section with a result that will be exploited in the sequel.

\begin{lemma}
\label{Lemma0} Let $X$ be a function space, for any $F:\mathbb R\times X
\rightarrow\left[  0,\infty\right]  $%
\[
\underset{\varepsilon\rightarrow 0^+}{\lim\sup}\inf_{u\in X}F\left(  \varepsilon,u\right)  \leq\inf_{u\in X }\underset{\varepsilon\rightarrow 0^+ }{\lim\sup}F\left(
\varepsilon,u\right)  .
\]

\end{lemma}

\begin{proof}
For any $\widetilde{u}\in X $%
\[
\inf_{u\in X  }F\left(  \varepsilon,u\right)  \leq
F\left(  \varepsilon,\widetilde{u}\right)  .
\]
Thus%
\[
\underset{\varepsilon\rightarrow 0^+}{\lim\sup}\inf_{u\in X  }F\left(  \varepsilon,u\right)  \leq\underset{\varepsilon
\rightarrow 0^+ }{\lim\sup}F\left(  \varepsilon,\widetilde{u}\right)
\]
for every $\widetilde{u}\in X  .$ Applying the
infimum in the previous inequality one obtains%
\[
\inf_{\widetilde{u}\in X}\underset{\varepsilon
\rightarrow0^+}{\lim\sup}\inf_{u\in X}F\left(
\varepsilon,u\right)  \leq\inf_{\widetilde{u}\in X  }\underset{\varepsilon\rightarrow 0^+ }{\lim\sup}F\left(  \varepsilon
,\widetilde{u}\right)  .
\]
Hence%
\[
\underset{\varepsilon\rightarrow 0^+ }{\lim\sup}\inf_{u\in X }F\left(  \varepsilon,u\right)  \leq\inf_{u\in X}\underset{\varepsilon\rightarrow 0^+ }{\lim\sup}F\left(
\varepsilon,u\right)  .
\] \end{proof}

\section{Lower bound}\label{lb}

This section  is devoted to the proof of the lower bound inequality for Theorem \ref{mainthmgen}. Recall that ${\cal F}$ and $\overline{F_0}$  are the functionals introduced in \eqref{calFG} and \eqref{representationFG}. 


\begin{theorem}
\label{lsctheorem} Let $\Omega\subset\mathbb{R}^{N}$ be a bounded open set,
let $f:\mathbb R^m \times \mathbb R^d\rightarrow\lbrack0,+\infty)$ satisfy $(F_1)-(F_4)$ and let $g:\mathbb R^m \times \mathbb R^m \times S^{N-1}\to [0, +\infty)$ satisfy $(G_1)-(G_3)$. 
Then for every 
$\left(  v,u\right)  \in SBV_0 ( \Omega;\mathbb R^m)  \times BV\left(
\Omega;\mathbb{R}^{d}\right)$, and for every sequence $\left\{  (v_{n}, u_n)\right\}  \subset
SBV_0( \Omega; \mathbb R^m) \times W^{1,1}\left(\Omega;\mathbb{R}^{d}\right)  $ such that $(v_n, u_n) \to (v,u)$ in $L^1(\Omega;\mathbb R^m)\times L^1(\Omega; \mathbb R^d)$,  
\begin{equation}
\overline{F_{0}}\left(  v,u;\Omega\right)  \leq\underset{n\rightarrow\infty}{\lim\inf
}F\left(  v_{n},u_{n};\Omega\right)  , \label{lsc0}%
\end{equation}
where ${\overline F_0}$ is given by \eqref{representationFG}  .
\end{theorem}

\begin{proof}[Proof] 
Let $(v,u) \in SBV_0(\Omega;\mathbb R^m)\times BV(\Omega;\mathbb R^d)$.
Without loss of generality, we may assume that  for every $\{(v_n, u_n)\} \subset SBV_0(\Omega;\mathbb R^m)\times BV(\Omega;\mathbb R^d)$ converging to $(v,u)$ in $L^1(\Omega; \mathbb R^m)\times L^1(\Omega;\mathbb R^d)$,
\begin{equation}\nonumber
\begin{array}{ll}
\displaystyle{\underset{n\rightarrow\infty}{\lim\inf}\left(\int_{\Omega} f\left( 
v_n,\nabla u_n\right)dx  +\int_{J_{v_n}\cap \Omega}g(v^+_n, v^-_n, \nu_{v_n})d {\cal H}^{N-1}\right)}\\
\\
\displaystyle{=\lim_{n\rightarrow\infty}\left(\int_{\Omega} f\left(  v_n,\nabla
u_{n}\right)dx  +   \int_{J_{v_n}\cap \Omega}g(v^+_n, v^-_n, \nu_{v_n})d {\cal H}^{N-1} \right) <+\infty.}
\end{array}
\end{equation}

For every Borel set $B \subset \Omega$ define
$$
\displaystyle{\mu_n(B):=\int_B f\left(  v_n,\nabla
u_{n}\right) dx +   \int_{J_{v_n}\cap B}g(v^+_n, v^-_n, \nu_{v_n})d {\cal H}^{N-1} .}
$$ 
Since $\{\mu_n\}$ is a  sequence of nonnegative Radon measures, uniformly bounded in the space of measures, we can extract a subsequence, still denoted by $\{\mu_n\}$,   
 weakly $\ast$ converging in the
sense of measures to some Radon measure $\mu.$ Using Radon-Nikod\'ym theorem we
can decompose $\mu$ as a sum of four mutually singular nonnegative measures,
namely%
\begin{equation}\label{mudecomposition}
\mu=\mu_{a}\mathcal{L}^{N}+\mu_{c}\left\vert D^{c}u\right\vert +\mu_{j}\mathcal{H}^{N-1}\lfloor J_{(v,u)}+\mu_{s},
\end{equation}
where we have been considering $(v,u)$ as a unique field in $BV(\Omega; \mathbb R^{m + d})$ and we have been exploiting the fact that $D^c(v,u)= (\underline{0}, D^c u)$ (cf. Remark \ref{vmeas}). 
By Besicovitch derivation theorem%
\begin{align}
\mu_{a}\left(  x_{0}\right)   &  =\lim_{\varepsilon\rightarrow0^{+}}\frac
{\mu\left(  B\left(  x_{0},\varepsilon\right)  \right)  }{\mathcal{L}%
^{N}\left(  B\left(  x_{0},\varepsilon\right)  \right)  }<+\infty
,~\text{for~\ }\mathcal{L}^{N}-\text{a.e.}~x_{0}\in\Omega,\nonumber\\
\mu_{j}\left(  x_{0}\right)   &  =\lim_{\varepsilon\rightarrow0^{+}}\frac
{\mu\left(  Q_{\nu}\left(  x_{0},\varepsilon\right)  \right)  }{\mathcal{H}%
^{N-1}\left(  Q_{\nu}\left(  x_{0},\varepsilon\right)  \cap J_{(v, u)}\right)
}<+\infty,~\text{for }\mathcal{H}^{N-1}-\text{a.e. }x_{0}\in J_{(v, u)}\cap
\Omega,\label{BDT}\\
\mu_{c}\left(  x_{0}\right)   &  =\lim_{\varepsilon\rightarrow0^{+}}\frac
{\mu\left(  Q\left(  x_{0},\varepsilon\right)  \right)  }{\left\vert
Du\right\vert \left(  Q\left(  x_{0},\varepsilon\right)  \right)  }%
<+\infty,~\text{for }\left\vert D^{c}u\right\vert -\text{a.e. }x_{0}\in
\Omega.\nonumber
\end{align}

We claim that%
\begin{equation}
\mu_{a}\left(  x_{0}\right)  \geq Qf\left(  v\left(  x_{0}\right)  ,\nabla
u\left(  x_{0}\right)  \right)  ,~\text{for \ }\mathcal{L}^{N}-\text{a.e}%
.~x_{0}\in\Omega,\label{lboundbulk}%
\end{equation}%
\begin{equation}
\mu_{j}\left(  x_{0}\right)  \geq K_3\left(  v^+(x_0),v^-(x_0),u^{+}\left(  x_{0}\right)
,u^{-}\left(  x_{0}\right)  ,\nu_{(v,u)}\right)  ,~\text{for }\mathcal{H}%
^{N-1}-\text{a.e}.~x_{0}\in J_{(v,u)}\cap\Omega,\label{lboundjump}%
\end{equation}%
\begin{equation}
\mu_{c}\left(  x_{0}\right)  \geq\left(  Qf\right)  ^{\infty}\left(
v\left(  x_{0}\right)  ,\frac{dD^{c}u}{d\left\vert D^{c}u\right\vert
}\left(  x_{0}\right)  \right)  \text{ for }\left\vert D^{c}u\right\vert
-\text{a.e. }x_{0}\in\Omega, \label{lboundcantor}%
\end{equation}

\noindent where $Qf$ is the density introduced in \eqref{Qfbar}, $Qf^\infty$ is  its recession function as in \eqref{finfty} and $K_3$ is given by \eqref{K3}.
If $\left(  \ref{lboundbulk}\right)  -\left(  \ref{lboundcantor}\right)  $
hold then $\left(  \ref{lsc0}\right)  $ follows immediately. Indeed, since
$\mu_{n}\overset{\ast}{\rightharpoonup}\mu$ in the sense of measures

\begin{align*}
&  \underset{n\rightarrow\infty}{\lim\inf}\left(\int_{\Omega} f\left(  v_n,\nabla
u_{n}\right)dx  +   \int_{J_{v_n}\cap \Omega}g(v^+_n, v^-_n, \nu_{v_n})d {\cal H}^{N-1} \right) \\
&  \geq\underset{n\rightarrow\infty}{\lim\inf}\mu_{n}\left(  \Omega\right)
\geq\mu\left(  \Omega\right)  \geq\int_{\Omega}\mu_{a}~dx+\int_{J_{(v,u)}}\mu
_{j}~d\mathcal{H}^{N-1}+\int_{\Omega}\mu_{c}d|D^{c}u|\\
&  \geq\int_{\Omega}Qf\left(  v\left(  x\right)  ,\nabla u\left(  x\right)
\right)  dx+\int_{J_{u}\cap\Omega}K_3\left(  v^+(x),v^-(x),u^{+}(x)
,u^{-}( x)  ,\nu_{(v,u)}\right)  d\mathcal{H}^{N-1}\\
&  +\int_{\Omega}\left(  Qf\right)  ^{\infty}\left(  v\left(  x\right)
,\frac{dD^{c}u}{d\left\vert D^{c}u\right\vert }\left(  x\right)  \right)
d\left\vert D^{c}u\right\vert
\end{align*}
where we have used the fact that $\mu_s$ is nonnegative.

We prove \eqref{lboundbulk}$-$\eqref{lboundcantor} using
the blow-up method introduced in \cite{FM1}.

\noindent\textbf{Step 1.} Let $x_{0}\in\Omega$ be a Lebesgue point for $\nabla u$ and $v$, such that
$x_{0}\notin J_{(v,u)},$ \eqref{approximate differentiability} applied to $u$,
and $\left(  \ref{BDT}\right)  _{1}$ hold.

We observe that 
$$
\begin{array}{ll}
\displaystyle{ \underset{n\rightarrow\infty}{\lim\inf}\left(\int_{\Omega} f\left(  v_n,\nabla
u_{n}\right)dx  +   \int_{J_{v_n}\cap \Omega}g(v^+_n, v^-_n, \nu_{v_n})d {\cal H}^{N-1} \right) }\\
\displaystyle{\geq \underset{n\rightarrow\infty}{\lim\inf}\int_{\Omega} f\left(  v_n,\nabla u_{n}\right)dx \geq \underset{n\rightarrow\infty}{\lim\inf}\int_{\Omega} Qf\left(  v_n,\nabla u_{n}\right)dx.}
\end{array}
$$

Note that by Proposition \ref{continuityQfbar} $Qf$ satisfies $(F_1)-(F_3)$. By Proposition \ref{prop2.4FM1} we may assume that $\left\{ (v_n, u_{n})\right\}  \subset
C_{0}^{\infty}\left(  \mathbb{R}^{N};\mathbb R^m \right) \times C_{0}^{\infty}\left(  \mathbb{R}^{N};\mathbb{R}^{d}\right) $ and
applying \cite[formula (2.10) in Theorem 2.19]{FM2}, to the functional $G: (v,u)\in W^{1,1}(\Omega;\mathbb R^{m+d}) \to \int_{\Omega}Q f(v,\nabla u)dx$ we obtain \eqref{lboundbulk}.

\noindent{\bf Step 2. }Now we prove $\left(  \ref{lboundjump}\right)  .$

Remind that $J_{( v,u)  }=J_v\cup J_{u}$ and $\nu_{(v,u)}= \nu_v$ for every $(v,u) \in SBV_0(\Omega;\mathbb R^m) \times W^{1,1}(\Omega;\mathbb R^d)$.
By Lemma \ref{lemma2.5BBBF}, Proposition \ref{thm2.3BBBF} ii) and Theorem
\ref{thm2.6BBBF} 
we may fix $x_{0}\in J_{( v,u)}\cap\Omega$
such that%
\begin{equation}\label{4.13}
\begin{array}{ll}
\displaystyle{\lim_{\varepsilon\rightarrow0^{+}}\frac{1}{\varepsilon^{N-1}}\int_{J_{(v,u)}  \cap Q_{\nu}\left(  x_{0},\varepsilon\right)  }\left( \left\vert v^{+}\left(  x\right)  -v^{-}\left(  x_{0}\right)
\right\vert +
\left\vert u^{+}\left(  x\right)  -u^{-}\left(  x_{0}\right)  \right\vert
 \right)  d\mathcal{H}^{N-1}}\\
\\
 \displaystyle{ =\left\vert v^{+}\left(  x_{0}\right)  -v^{-}\left(  x_{0}\right)
\right\vert +\left\vert u^{+}\left(  x_{0}\right)  -u^{-}\left(
x_{0}\right)  \right\vert,}
\end{array}
\end{equation}

\noindent 
\begin{equation}\label{4.14}
\begin{array}{ll}
\displaystyle{\lim_{\varepsilon\rightarrow0^{+}}\frac{1}{\varepsilon^{N}}\int_{\left\{
x\in Q_{\nu}\left(  x_{0},\varepsilon\right)  :\left(  x-x_{0}\right)  \cdot\nu\left(
x\right)  >0\right\}  }\left\vert v\left(  x\right)  -v^{+}\left(  x_0\right)
\right\vert ^{\frac{N}{N-1}}dx }\\
\\
\displaystyle{+\lim_{\varepsilon\rightarrow0^{+}}\frac{1}{\varepsilon^{N}}\int_{\left\{
x\in Q_{\nu}\left(  x_{0},\varepsilon\right)  :\left(  x-x_{0}\right)  \cdot\nu\left(
x\right)  >0\right\}  }\left\vert u\left(  x\right)  -u^{+}\left(
x_0\right)  \right\vert ^{\frac{N}{N-1}}dx=    0},
\end{array}
\end{equation}

\noindent
\begin{equation}\label{4.15}
\begin{array}{ll}
\displaystyle{\lim_{\varepsilon\rightarrow0^{+}}\frac{1}{\varepsilon^N}\int_{\left\{
x\in Q_{\nu}\left(  x_{0},\varepsilon\right)  :\left(  x-x_{0}\right)  \cdot\nu\left(
x\right)  <0\right\}  }\left\vert v\left(  x\right)  -v^{-}\left(  x_0\right)
\right\vert ^{\frac{N}{N-1}}dx}\\
\\
\displaystyle{+\lim_{\varepsilon\rightarrow0^{+}}\frac{1}{\varepsilon^N}\int_{\left\{
x\in Q_{\nu}\left(  x_{0},\varepsilon\right)  :\left(  x-x_{0}\right)  \cdot\nu\left(
x\right)  <0\right\}  }\left\vert u\left(  x\right)  -u^{-}\left(
x_0\right)  \right\vert ^{\frac{N}{N-1}}dx=0,}
\end{array}
\end{equation}

\noindent
\begin{equation}
\displaystyle{\mu_{j}\left(  x_{0}\right)  =\lim_{\varepsilon\rightarrow0^{+}}\frac
{\mu(  x_0+\varepsilon Q_{\nu( x_0)})}{\mathcal{H}^{N-1}\lfloor J_{( v,u)  }(  x_0+\varepsilon Q_{\nu( x_0)})  }\text{ \ exists and it is
finite.}} \label{4.16}%
\end{equation}

For simplicity of notation we write $Q:=Q_{\nu\left(  x_{0}\right)  }.$ Then
by $\left(  \ref{4.16}\right)  $,
\begin{equation}\label{4.17}
\mu_{j}(  x_0)  =\lim_{\varepsilon
\rightarrow 0^+}\frac{1}{\varepsilon^{N-1}}\int_{x_0+\varepsilon Q}%
d\mu\left(  x\right)  .
\end{equation}
Without loss of generality, we may choose $\varepsilon>0$ such that
$\mu\left(  \partial\left(  x_{0}+\varepsilon Q\right)  \right)  =0.$ Since $Qf \leq f$, we have%
\begin{align*}
\mu_{j}\left(  x_{0}\right)   &  \geq\lim_{\varepsilon\rightarrow0^{+}}%
\lim_{n\rightarrow\infty}\frac{1}{\varepsilon^{N-1}}\left(  \int%
_{x_{0}+\varepsilon Q}Qf\left(  v_n\left(  x\right)  ,\nabla u_{n}\left(
x\right)  \right)  dx+\int_{J_{v_n}}g(v_n^+, v_n^-, \nu_{v_n}) d {\cal H}^{N-1} \right)  \\
&  =\lim_{\varepsilon\rightarrow0^{+}}\lim_{n\rightarrow\infty}\varepsilon
\int_{Q}Qf\left(  v_n\left(  x_{0}+\varepsilon y\right)  ,\nabla
u_{n}\left(  x_{0}+\varepsilon y\right)  \right)  dy\\
&  +\int_{Q\cap J\left( v_n, u_n\right)  -\frac{x_0}{\varepsilon}} 
g\left( v_n^+ (  x_0+\varepsilon y), v_n^-(x_0+ \varepsilon y) ,
\nu_{(v_n, u_n)} (x_0+\varepsilon y)\right) d\mathcal{H}^{N-1}\left(  y\right)  .
\end{align*}
Define
\begin{equation}\label{vne}
\begin{array}{cc}
v_{n,\varepsilon}\left(  y\right)  :=v_{n}\left(  x_{0}+\varepsilon
y\right) , \;
u_{n,\varepsilon}\left(  y\right)   :=u_{n}\left(  x_{0}+\varepsilon
y\right),\;
\nu_{n,\varepsilon}\left(  y\right)  :=\nu_{\left(  v_n,u_n\right)
}\left(  x_{0}+\varepsilon y\right)  ,
\end{array}
\end{equation}
and%
\begin{equation}\label{u0v0}
\begin{array}{cc}
v_{0}\left(  y\right)  :=\left\{
\begin{array}
[c]{ccc}
v^{+}(x_0)  &  & \text{if }y\cdot\nu\left(  x_{0}\right)  >0,\\
v^{-}(x_0)    &  & \text{   if }y\cdot\nu\left(  x_{0}\right)  <0, \;\;
\end{array}
\right.
u_{0}\left(  y\right)  :=\left\{
\begin{array}
[c]{ccc}%
u^{+}\left(  x_{0}\right)   &  & \text{if }y\cdot\nu\left(  x_{0}\right)
>0,\\
u^{-}\left(  x_{0}\right)   &  & \text{if }y\cdot\nu\left(  x_{0}\right)
<0.
\end{array}
\right.
\end{array}
\end{equation}
Since $(v_n,u_n)\rightarrow (v,u)$ in $L^{1}\left(  \Omega;\mathbb{R}^{m+d}\right)$, by \eqref{4.14} and \eqref{4.15}
one obtains%
\begin{equation}\label{blabla}
\begin{array}{ll}
&  \displaystyle{\lim_{\varepsilon\rightarrow0^{+}}\lim_{n\rightarrow\infty}\int%
_{Q}\left\vert v_{n,\varepsilon}\left(  y\right)  -v_{0}\left(  y\right)
\right\vert dy=\lim_{\varepsilon\rightarrow0^{+}}\frac{1}{\varepsilon^{N}%
}\left(  \int_{\left\{  x\in x_{0}+\varepsilon\partial Q:\left(
x-x_{0}\right)  \cdot\nu\left(  x_{0}\right)  >0\right\}  }\left\vert v\left(
x\right)  -v^{+}\left(  x_{0}\right)  \right\vert dx\right.}  \\
& \displaystyle{ \left.  +\int_{\left\{  x\in x_{0}+\varepsilon\partial Q:\left(
x-x_{0}\right)  \cdot\nu\left(  x_{0}\right)  <0\right\}  }\left\vert v\left(
x\right)  -v^{-}\left(  x_{0}\right)  \right\vert dx\right)  =0}
\end{array}
\end{equation}
and%
\begin{equation}\label{blabla2}
\begin{array}{ll}
&  \displaystyle{\lim_{\varepsilon\rightarrow0^{+}}\lim_{n\rightarrow\infty}\int%
_{Q}\left\vert u_{n,\varepsilon}\left(  y\right)  -u_0\left(
y\right)  \right\vert dy=\lim_{\varepsilon\rightarrow0^{+}}\frac
{1}{\varepsilon^{N}}\left(  \int_{\left\{  x\in x_{0}+\varepsilon\partial
Q:\left(  x-x_{0}\right)  \cdot\nu\left(  x_{0}\right)  >0\right\}
}\left\vert u\left(  x\right)  -u^{+}\left(  x_{0}\right)  \right\vert
dx\right.}  \\
&  \displaystyle{+\left.  \int_{\left\{  x\in x_{0}+\varepsilon\partial Q:\left(
x-x_{0}\right)  \cdot\nu\left(  x_{0}\right)  <0\right\}  }\left\vert
u\left(  x\right)  -u^{-}\left(  x_{0}\right)  \right\vert dx\right)
=0.}
\end{array}
\end{equation}
Thus%
\begin{align*}
\mu_{j}\left(  x_{0}\right)   &  \geq\lim_{\varepsilon\rightarrow0^{+}}%
\lim_{n\rightarrow\infty}\left(\int_{Q}Qf^{\infty}\left(  v_{n,\varepsilon
}\left(  y\right)  ,\nabla u_{n,\varepsilon}\left(  y\right)  \right)
dy+\int_{Q\cap J\left(  v_{n,\e} ,u_{n_\e}\right)  } g(v_{n,\e}^+, v_{n,\e}^-, \nu_{v_{n,\e}}) d {\cal H}^{N-1}(y)\right.\\
&  \left.+\int_{Q}\left(\varepsilon Qf\left(  v_{n,\varepsilon}\left(  y\right)
,\frac{1}{\varepsilon}\nabla u_{n,\varepsilon}\left(  y\right)  \right)
-Qf^{\infty}\left(  v_{n,\varepsilon},\nabla u_{n,\varepsilon}\right) \right) dy \right).
\end{align*}

Exploiting $(v)$ in Remark \ref{propfinfty} we can argue as in the estimates \cite[(3.3)-(3.5)]{FM2}, thus obtaining
\[
\mu_{j}\left(  x_{0}\right)  \geq\underset{\varepsilon\rightarrow0^{+}%
}{\lim\inf}~\underset{n\rightarrow\infty}{\lim\inf}\left(  \int_{Q}Qf^{\infty
}\left( v_{n,\varepsilon}\left(  y\right)  ,\nabla u_{n,\varepsilon
}\left(  y\right)  \right)  dy+\int_{Q\cap J\left(  v_{n,\varepsilon}%
,u_{n,\varepsilon}\right)  }g(v_{n,\e}^+, v_{n,\e}^-, \nu_{v_{n,\e}}) d {\cal H}^{N-1}\left(  y\right)  \right)  .
\]
Since $(v_{n,\varepsilon}, u_{n,\e})\rightarrow (v_0, u_0)$ in $L^{1}\left(  Q;\mathbb{R}%
^{m+d}\right)  $ as $n\rightarrow\infty$ and $\varepsilon\rightarrow
0^{+},$ by a standard diagonalization argument, as in \cite[Theorem 4.1 Steps 2 and 3]{BBBF}, we obtain a sequence $(\bar{v}_k,\bar{u}_k)$ converging to $(v_0, u_0)$ in $L^1(Q;\mathbb R^{m+d})$ as $k \to  \infty$ such that
\[
\mu_{j}\left(  x_{0}\right)  \geq\lim_{k\rightarrow\infty}\left(  \int%
_{Q}Qf^{\infty}\left(  \bar{v}_{k}\left(  y\right)  ,\nabla \bar{u}_{k}\left(
y\right)  \right)  dy+\int_{Q\cap J_{\left(  v_k,w_k \right)  }%
} g(\bar{v}_k^+, \bar{v}_k^-, \nu_{\bar{v}_k}) d {\cal H}^{N-1}\left(  y\right)  \right).
\]

Applying Lemma \ref{Lemma4.1FM} with $Qf$ replaced by $Qf^\infty$ and using $(v)$ in Remark \ref{propfinfty} we may find $\left\{  \left(  \zeta_{k}%
,\xi_{k}\right)  \right\}  \in\mathcal{A}_3\left(  v^+(x_0),v^-(x_0),u^{+}\left(
x_{0}\right)  ,u^{-}\left(  x_{0}\right)  ,\nu\left(  x_{0}\right)  \right)  $
such that

\begin{equation}\nonumber
\begin{array}{ll}
\displaystyle{\mu_{j}(x_0)  \geq\lim_{k\rightarrow\infty}\left(  \int_Q Qf^{\infty}\left( \zeta_{k},\nabla\xi_{k}\right)  dx+ \int_{Q \cap J_{(\zeta_k, \xi_k)} }g(\zeta_k^+, \zeta_k^-, \nu_{\zeta_k}) d {\cal H}^{N-1} \right)} 
\\
\\
\displaystyle{\geq K_3\left(
v^+(x_0), v^-(x_0),u^{+}\left(  x_{0}\right)  ,u^{-}\left(  x_{0}\right)  ,\nu\left(
x_{0}\right)  \right) }.
\end{array}
\end{equation}

\noindent\textbf{Step 3.} Here we show \eqref{lboundcantor}.

\noindent Let $(v,u)  \in SBV_0(\Omega;\mathbb R^m)\times BV\left(  \Omega;\mathbb{R}^{d}\right)  $,  note, as already emphasized in Remark \ref{vmeas}, that
$|D^c(v,u)|= |D^c u|$. For $\left\vert D^{c}u\right\vert -$a.e.
$x_{0}\in\Omega$ we have%
\[
\lim_{\varepsilon\rightarrow0^{+}}\frac{\left\vert D(v,u)\right\vert \left(
Q\left(  x_{0},\varepsilon\right)  \right)  }{\left\vert D^{c}(v,u)\right\vert
\left(  Q\left(  x_{0},\varepsilon\right)  \right)  }= \lim_{\varepsilon\rightarrow0^{+}}\frac{\left\vert D(v,u)\right\vert \left(
Q\left(  x_{0},\varepsilon\right)  \right)  }{\left\vert D^{c}u\right\vert
\left(  Q\left(  x_{0},\varepsilon\right)  \right)  }=1.
\]
And so by Theorems 2.4. $iii)$  and 2.11 in \cite{FM2}, and by Theorem \ref{thm2.6BBBF} for
$\left\vert D^{c}u\right\vert -$a.e. $x_{0}\in\Omega$ the following hold%
\[
\mu_c\left(  x_{0}\right)  =\lim_{\varepsilon\rightarrow0^{+}}\frac{\mu\left(
Q\left(  x_{0},\varepsilon\right)  \right)  }{\left\vert Du\right\vert \left(
Q\left(  x_{0},\varepsilon\right)  \right)  },
\]%
\[
\lim_{\varepsilon\rightarrow0^{+}}\frac{1}{\varepsilon^{N}}\int_{Q\left(
x_{0},\varepsilon\right)  }\left(\left\vert u\left(  x\right)  -u\left(
x_{0}\right)  \right\vert +\left\vert v\left(  x\right)  -v\left(
x_{0}\right)  \right\vert\right) dx=0,
\]
for $\mathcal{H}^{N-1}-x_{0}\in\Omega\backslash J_{(v,u)},$%
\[
A\left(  x_{0}\right)  =\lim_{\varepsilon\rightarrow0^{+}}\frac{\left(
D(v,u)\right)  \left(  Q\left(  x_{0},\varepsilon\right)  \right)  }{\left\vert
D(v,u)\right\vert \left(  Q\left(  x_{0},\varepsilon\right)  \right)
},~~\left\Vert A\left(  x_{0}\right)  \right\Vert =1,~~A\left(  x_{0}\right)
=a\otimes\nu,
\]%
with $a \in \mathbb R^d$ and $\nu \in S^{N-1}$,
\[
\lim_{\varepsilon\rightarrow0^{+}}\frac{\left\vert D(v,u)\right\vert \left(
Q\left(  x_{0},\varepsilon\right)  \right)  }{\varepsilon^{N-1}}%
=\lim_{\varepsilon\rightarrow0^{+}}\frac{\left\vert Du\right\vert \left(
Q\left(  x_{0},\varepsilon\right)  \right)  }{\varepsilon^{N-1}}=0,
\]%
\[
\lim_{\varepsilon\rightarrow0^{+}}\frac{\left\vert D(v,u)\right\vert \left(
Q\left(  x_{0},\varepsilon\right)  \right)  }{\varepsilon^{N}}=\lim
_{\varepsilon\rightarrow0^{+}}\frac{\left\vert Du\right\vert \left(  Q\left(
x_{0},\varepsilon\right)  \right)  }{\varepsilon^{N}}=\infty.
\]


Arguing as in the end of Step 1, by Proposition \ref{prop2.4FM1} (ii), we may assume that $\left\{ (\widetilde{v}_n, \widetilde{u}_n)\right\}  \subset C_{0}^{\infty}(\mathbb R^N;\mathbb R^{m+d})$.
Applying \cite[formula (2.12) in Theorem 2.19]{FM2}, to the functional $G: (v,u)\in W^{1,1}(\Omega;\mathbb R^{m+d}) \to \int_{\Omega}Qf(v,\nabla u)dx$ we obtain for $\left\vert
D^{c}(v,u)\right\vert -$a.e. $x_{0}\in\Omega$
\begin{equation}\nonumber
\mu_c(  x_{0})    \geq(  Qf)  ^{\infty}\left(
v(  x_{0})  ,\frac{dD^{c} u  }{d\left\vert
D^{c}u \right\vert }  (x_0)  \right)
\end{equation}
and that concludes the proof.
\end{proof}


\section{Upper bound}\label{ub}

This section is devoted to prove that ${\cal F}\leq {\overline F_0}$. 

\begin{theorem}
\label{thupperbound}Let $\Omega\subset\mathbb{R}^{N}$ be a bounded open set, let $f:\mathbb R^d \times \mathbb R^m\rightarrow\lbrack0,+\infty)$, be a function satisfying $(F_1)$ - $(F_4)$, and let $g: \mathbb R^m \times \mathbb R^m \times S^{N-1}\to [0,+\infty[$ be a function satisfying $(G_1)$ - $(G_3)$. 

Then for every $\left( v,u\right)  \in
SBV_0\left(  \Omega;\mathbb R^m \right)  \times BV\left(
\Omega;\mathbb{R}^{d}\right)  ,$ for every $A\in\mathcal{A}\left(
\Omega\right)$, there exist sequences $\left\{  v_n\right\}  \subset
SBV_0\left(  \Omega;\mathbb R^m\right) ,\left\{  u_{n}\right\}
\subset W^{1,1}\left(  \Omega;\mathbb{R}^{d}\right)  $ such that $v_n \to v$ in $L^1\left(  \Omega;\mathbb R^m\right)$,  $u_{n}\rightarrow u$ 
in $L^{1}\left(  \Omega;\mathbb{R}^{d}\right)  $, and %
\begin{equation}\nonumber
\underset{n\rightarrow\infty}{\lim\inf}F\left(  v_n,u_n;A\right)  \leq
{\overline F_0}\left(  v,u;A\right)  . 
\end{equation}

\end{theorem}

\noindent Before proving the upper bound  we recall our strategy, which was first proposed in \cite{AMT} and further developped in \cite{FM2}. Namely, first we will show that  
${\cal F}(v, u;\cdot )$ is a variational functional with respect to the $L^1$ topology and
\begin{equation}\nonumber
{\cal F}(v,u;\cdot)  \leq \mathcal{L}^{N}+ |Dv|+|Du| + {\cal H}^{N-1}\lfloor{J_v}. 
\end{equation}
Next by Besicovitch differentiation Theorem, a blow-up argument will provide an upper bound estimate in terms of ${\overline F}_0$, first for bulk and Cantor parts, then also for the jump part, when the target functions $(v,u)$ are bounded. Finally the same approximation as in \cite[Theorem 4.9]{AMT}, will give the estimate for every $(v,u)\in SBV_0(\Omega;\mathbb R^m)\times BV(\Omega;\mathbb R^d)$.

We recall that ${\cal F}(v,u ; \cdot)$ is said to be a variational functional with respect to the
$L^1$ topology if
\begin{itemize}
\item[(i)] ${\cal F}(\cdot, \cdot ;A)$ is local, i.e.,
${\cal F}(v,u;A) = {\cal F}(v', u';A)$
for every $v,v' \in SBV_0(A;\mathbb R^m)$, $ u, u' \in BV(A; \mathbb R^d)$ satisfying $u = u'$ , $v= v'$ a.e. in $A$. 

\item[(ii)] ${\cal F}(\cdot, \cdot;A)$ is sequentially lower semicontinuous, i.e., if $v_n, v \in  BV(A; \mathbb R^m)$, $u_n, u \in BV(A;\mathbb R^d)$
and $v_n \to v$  in $L^1(A; \mathbb R^m)$,  $u_n \to u$ in $L^1(A;\mathbb R^d)$ then
${\cal F}(v, u; A) \leq \liminf_{n \to \infty} {\cal F}(v_n,u_n ; A)$.

\item[(iii)] ${\cal F}(\cdot, \cdot ;A)$ is the trace on $\{A \subset \Omega: A \hbox{ is open}\}$  
of a Borel measure on ${\cal B}(\Omega) $ the family of all Borel subsets of $\Omega$.

\end{itemize}

Since the lower semicontinuity and the locality of ${\cal F}(\cdot, \cdot; A)$ follow by its definition, it remains to prove $(iii)$. This is the target of the following lemma, where  $(iii)$ will be obtained via a refinement of De Giorgi-Letta criterion, cf. \cite[Corollary 5.2]{DMFL}.

\begin{lemma}
\label{measure} Let $\Omega\subset\mathbb{R}^{N}$ be an open bounded set with
Lipschitz boundary and let  $f$ and $g$ be as in Theorem \ref{thupperbound}. For  every $\left( v,u\right)  \in
SBV_0\left(  \Omega;\mathbb R^m \right)  \times BV\left(
\Omega;\mathbb{R}^{d}\right) $, the set
function $\mathcal{F}\left(  v,u;\cdot\right)  $  in \eqref{calFG} is the trace of a Radon
measure absolutely continuous with respect to $\mathcal{L}^{N}+ |Dv|+\left\vert
Du\right\vert  + {\cal H}^{N-1}\lfloor{J_v}.$
\end{lemma}

\begin{proof}
An argument very similar to \cite[ Lemma 2.6 and Remark 2.7]{BFMGlobal} and \cite[Lemma 4.7]{BZZ} entails
\[
\mathcal{F}(v,u; A)\leq C\left(  \mathcal{L}^{N}(A)+|D v
|(A)+|Du|(A)+ {\cal H}^{N-1}\lfloor{J_v}(A)\right)  .
\]

\noindent By \cite[Corollary 5.2]{DMFL} to obtain $(iii)$ it suffices to prove that
\begin{equation}\nonumber
\displaystyle\mathcal{F}{(v,u;A)\leq\mathcal{F}(v,u;B)+\mathcal{F}%
(v,u;A\setminus\overline{U})}
\end{equation}
\noindent for all $A,U,B\in\mathcal{A}(\Omega)$ with $U\subset\subset B\subset\subset A$,  $u\in BV(\Omega;\mathbb{R}^{d})$ and $v\in SBV_0(\Omega;\mathbb R^m)$.

We start by assuming that $v \in SBV_0(\Omega;\mathbb R^m)\cap L^\infty(\Omega;\mathbb R^m)$.

Fix $\eta>0$ and find $\{w_{n}\}\subset W^{1,1}\left(  (A\setminus\overline
{U});\mathbb{R}^{d}\right)  $, $\{v
_{n}\}\subset SBV_0(A\setminus\overline{U};\mathbb R^m)\cap L^\infty(A\setminus\overline{U};\mathbb R^m)$   
(cf. Remark \ref{asinGlobalMethodBFLM}) such that $w_n\rightarrow
u$ in $L^{1}((A\setminus\overline{U});\mathbb{R}^{d})$ and $v
_{n}\rightarrow v$ in $L^{1}((A\setminus\overline{U});\mathbb R^m)$ and%
\begin{equation}
{\limsup_{n\rightarrow\infty}\int_{A\setminus\overline{U}}f\left( v_n,  \nabla w_n  \right)  dx+\int_{A \setminus {\overline U} \cap J_{v_n}} g(v_n^+. v_n^-, \nu_{v_n})d {\cal H}^{N-1}  \leq\mathcal{F}(v,u;A\setminus{\overline{U}})+\eta.}
\label{E1}%
\end{equation}

Extract a subsequence still denoted by $n$ such that the above upper limit is
a limit.

Let $B_0$ be an open subset of $\Omega$ with Lipschitz boundary
 such that $U\subset\subset B_{0}%
\subset\subset B$. Then there exist $\{u_{n}\}\subset W^{1,1}(B_{0}%
;\mathbb{R}^{d})$ and $\{\overline{v}_{n}\}\subset SBV_0\left(  B_{0}%
;\mathbb R^m\right) \cap L^\infty(B_0;\mathbb R^m) $ 
(cf. (i) in Remark \ref{asinGlobalMethodBFLM}) such that $u_{n}\rightarrow u$ in $L^{1}(B_{0}%
;\mathbb{R}^{d})$ and $\overline{v}_{n}\rightarrow v$ in $L^{1}%
(B_{0};\mathbb R^m)$ and%
\begin{equation}
\mathcal{F}{(v,u;B_{0})=\lim_{n\rightarrow\infty}}\left( \int_{B_0}f({\overline v_n},\nabla u_n)dx + \int_{J_{\overline v_n}\cap B_0}g({\overline v_n}^+, {\overline v_n}^-, \nu_{\overline v_n})d {\cal H}^{N-1} \right). \label{E2}%
\end{equation}

\noindent For every $(\overline{v}, w) \in SBV_0(A;\mathbb R^m)\cap L^\infty(A;\mathbb R^m) \times W^{1,1}(A;\mathbb R^d)$,  consider $\displaystyle{\mathcal{G}_{n}({\overline v}, w;A):=\int_{A}\left(  1+|\nabla w|\right)  dx}$ $\displaystyle{+ (1+[{\overline v}]){\cal H}^{N-1}\lfloor{(J_{\overline v} \cap A)}}$.

\noindent Due to the coercivity \eqref{H1}, we may extract a bounded subsequence not
relabelled, from the sequence of measures $\nu_{n}:=\mathcal{G}_{n}(v_n, w_n;\cdot)+\mathcal{G}_{n}({\overline v}_n, u_n;\cdot)$ restricted to $B_{0}\setminus
\overline{U}$, converging in the sense of distributions to some Radon measure
$\nu$, defined on $B_{0}\setminus\overline{U}$.
Analogously, for every $w  \in SBV_0(A;\mathbb R^m)\cap L^\infty(A;\mathbb R^m)$ we could define a sequence of measures ${\cal H}_n(w;E)$:=$\int_{J_w \cap E}d {\cal H}^{N-1}$.

\noindent For every $t >0$ , let $B_{t}:= \left\{  x\in B_{0} | \mathrm{dist}%
(x, \partial B_{0}) > t\right\}  $. Define, for $0 < \delta< \eta$,
the subsets $L_{\delta}:= B_{\eta- 2 \delta} \setminus\overline{B_{\eta+
\delta}}.$ Consider a smooth cut-off function $\varphi_{\delta}\in C^{\infty
}_{0}(B_{\eta-\delta};[0,1])$ such that $\varphi_\delta(x)= 1$ on $B_{\eta}%
$. As the thickness of the strip is of order $\delta$, we have an upper bound
of the form $\|\nabla\varphi_{\delta}\|_{L^{\infty}(B_{\eta-\delta})}
\leq\frac{C}{\delta}.$

\noindent Define ${\overline w_n}(x):=\varphi_{\delta}(x)u_{n}(x)+(1-\varphi_{\delta}(x))w_{n}(x)$. Clearly, $\left\{  {\overline w_n}\right\}  $ converges to $u$ in $L^{1}(A)$ as
$n\rightarrow\infty$, and%
\[
\nabla {\overline w_n}=\varphi_{\delta}\nabla u_{n}+(1-\varphi_{\delta})\nabla
w_{n}+\nabla\varphi_{\delta}\otimes(u_{n}-w_{n}).
\]

\noindent Arguing as in \cite[Lemma 4.4]{ABr1}, we may consider a sharp transition for the  
\noindent $SBV_0$ functions, namely 
\noindent let $\{v_{n}\}$
and $\{{\overline v_n}\}$ be as above, then for every  $0<t <1$ we
may define  $\tilde{v}_{n}^{t}$ such that $\tilde{v}_{n}^{t}%
\rightarrow v$ in $L^{1}(A)$ as $n\rightarrow\infty$, and%

\[
\tilde{v}_{n}^{t}(x):=\left\{
\begin{array}
[c]{ll}%
v_n(x) & \hbox{ in }\{x:\varphi_\delta(x)<t\},\\
{\overline v_n}(x) & \hbox{ in }\{x:\varphi_\delta(x)\geq t\}.
\end{array}
\right.
\]
\noindent Clearly $\tilde{v}_{n}^{t}(x)\in\{v_{n}(x),\overline{v}%
_{n}(x)\}$ almost everywhere in $A$,
 and since $\mathcal{H}^{N-1}(J_{v_n}), {\mathcal H}^{N-1}(J_{\overline v_n}) < +\infty$
for all but at most countable $t \in \left] 0,1\right[ $ it results
that
\[
\mathcal{H}^{N-1}\left( J_{v_{n}}\cap \left\{ x\in A:\varphi _{\delta
}\left( x\right) =t\right\} \right) =\mathcal{H}^{N-1}\left( J_{\overline{v}%
_{n}}\cap \left\{ x\in A:\varphi _{\delta }\left( x\right) =t\right\}
\right) =0.
\]

Moreover, using coarea formula \eqref{FR} and the mean value theorem it is possible to find a $t$ for which the integral over the level set is comparable with the double integral with $t$ varying between $0$ and $1$.
Thus we have
$$
\int_{\partial^{\ast}\{\varphi_\delta < t\}} d {\cal H}^{N-1}\leq  \frac{C}{\delta} {\cal L}^N(B_{\eta - \delta} \setminus B_\eta) \leq C.
$$
An analogous reasoning provides for the same $t$ that
\begin{equation}\label{secondmeanvalue}
{\int_{\partial^{\ast}\{\varphi_\delta<t\}}|[{\tilde v_n}^t]|d {\cal H}^{N-1}} \leq\frac{C}{\delta
}\int_{B_{\eta-\delta}\setminus B_{\eta}}|v_{n}(x)-\overline{v}_{n}(x)|dx.
\end{equation}
Thus, as  for the $\{{\cal G}_n\}$ above, we may extract a bounded subsequence not
relabelled, from the sequence of measures ${\cal H}_n ({\tilde v}_n^t, \cdot)$,  restricted to $B_0\setminus \overline U \cap \partial^\ast\{\varphi_\delta <t\}$, converging in the sense of distributions to some Radon measure
$\nu_1$, defined on $B_{0}\setminus \overline U$.

By \eqref{H1} we have the estimate%
\[
\begin{array}
[c]{l}
\displaystyle{\int_A f\left(  {\tilde v_n}^t,  \nabla {\overline w_n}\right)  dx+ \int_{A \cap J_{\tilde v_n^t}} g(({\tilde v}_n^t)^+, ({\tilde v_n}^t)^-, \nu_{{\tilde v_n}^t})d {\cal H}^{N-1}}\\
\\
\leq{\displaystyle\int_{B_{\eta}}
f({\overline v_n}, \nabla u_n)dx+ 
\int_{J_{\overline v_n}\cap B_\eta}g({\overline v_n}^+, {\overline v_n}^-, \nu_{\overline v_n })d {\cal H}^{N-1}}\\
\\
+{\displaystyle\int_{(A\setminus\overline{B_{\eta-\delta}})}
f\left( v_n,  \nabla w_n\right)dx + \int_{J_{v_n}\cap(A\setminus\overline{B_{\eta-\delta}}) }g( v_n^+, v_n^-, \nu_{v_n })d {\cal H}^{N-1}} \\
\\
+C\left( \mathcal{G}_n(v_n, w_n; L_{\delta})+{\mathcal G}_n({\overline v_n}, u_n;L_{\delta})\right)  +
\frac{1}{\delta} \displaystyle{\int_{L_{\delta}}
|w_n-u_n|dx+\int_{\partial^{\ast}\{\varphi_\delta<t\}}|[{\tilde v_n}^t]|d {\cal H}^{N-1}+ {\cal H}_n ({\tilde v}_n^t; L_\delta \cap \partial^{\ast}\{\varphi_\delta<t\})}
\\
\\
\leq {\displaystyle\int_{B_0}
f({\overline v_n}, \nabla u_n)dx+ 
\int_{J_{\overline v_n}\cap B_0}g({\overline v_n}^+, {\overline v_n}^-, \nu_{\overline v_n })d {\cal H}^{N-1}}\\
\\
+{\displaystyle\int_{(A\setminus\overline{U})}
f\left( v_n,  \nabla w_n\right)dx + \int_{J_{v_n}\cap(A\setminus\overline{U}) }g( v_n^+, v_n^-, \nu_{v_n })d {\cal H}^{N-1}} \\
\\
+C\left( \mathcal{G}_n(v_n, w_n, L_{\delta})+{\mathcal G}_n({\overline v_n}, u_n,L_{\delta})\right)  +
\frac{1}{\delta} \displaystyle{\int_{L_{\delta}}
|w_n-u_n|dx+\int_{\partial^{\ast}\{\varphi_\delta<t\}}|[{\tilde v_n}^t]|d {\cal H}^{N-1}+ {\cal H}_n ({\tilde v}_n^t; L_\delta \cap \partial^{\ast}\{\varphi_\delta<t\})}
 \end{array}
\]

Passing to the limit as $n\rightarrow\infty$, and applying \eqref{E1}, 
\eqref{E2}, \eqref{secondmeanvalue}  and the $L^1$ convergence of
$\{v_n\}$ and $\{\overline v_n\}$ to $v$, 
it results that
\begin{align*}
\mathcal{F}{(v,u;A)}  &  {\leq\mathcal{F}}({v,u;B_{0})+\mathcal{F}%
(v,u;A\setminus{\overline{U}})+\eta+C\nu(\overline{L_{\delta}}) + C \nu_1(\overline{L_\delta})
+\limsup_{n\rightarrow\infty}\int_{\partial^{\ast}\{\varphi_\delta<t\}}|[{\tilde v_n}^t]|d {\cal H}^{N-1}}\\
&  \leq\mathcal{F}{(v,u;B)+\mathcal{F}(v,u;A\setminus{\overline{U}%
})+\eta+C\nu(\overline{L_{\delta}})+ C\nu_1(\overline{L_{\delta}}).}%
\end{align*}
\noindent Letting $\delta$ go to $0$ we obtain
\[
\displaystyle\mathcal{F}{(v,u;A)\leq\mathcal{F}(v,u;B)+\mathcal{F}%
(v,u;(A\setminus{\overline{U}}))+\eta+C\nu(\partial B_{\eta})+ C\nu_1(\partial B_{\eta}).}%
\]

\noindent It suffices to choose a subsequence $\{\eta_{i}\}$ such that $\eta_i
\to 0^{+}$ and $\nu(\partial B_{\eta_i})=\nu_1(\partial B_{\eta_i})=0$, to conclude the proof of subadditivity in the case $v \in SBV_0\cap L^\infty$.


\noindent In the general case, by virtue of Remark \ref{asinGlobalMethodBFLM}, we can argue as in the last part of Theorem 10 in \cite{BFLM}.\end{proof}

\medskip

\medskip
\begin{proof}
[Proof of Theorem \ref{thupperbound}]
We assume  first that $(v,u)\in (SBV_0(\Omega;\mathbb R^m)\times BV(\Omega;\mathbb R^d)) \cap L^\infty(\Omega;\mathbb R^{m+d})$.

\noindent\textbf{Step 1.} In order to prove the upper bound,  we start by recalling that by Proposition \ref{propqcx} we can  replace $Qf$ by $f$  in \eqref{calFG}. First we deal with the bulk part. 


Since the $\mathcal{F}\left(  v,u;\cdot\right)  $ is a measure absolutely
continuous with respect to $\mathcal{L}^{N}+\left\vert Du\right\vert
+(1+ [v]){\cal H}^{N-1}\lfloor{J_v} $ we claim that
\[
\frac{d\mathcal{F}\left(  v,u;\cdot\right)  }{d\mathcal{L}^{N}}\left(
x_{0}\right)  \leq Qf\left(  v\left(  x_{0}\right)  ,\nabla u\left(
x_{0}\right)  \right)
\]
for $\mathcal{L}^{N}-$a.e. $x_{0}\in\Omega$ where $x_{0}$ is a Lebesgue point
of $v$ and $u$ such that
\begin{equation}%
\begin{array}
[c]{l}%
\lim\limits_{\varepsilon\rightarrow0^{+}}\frac{1}{\varepsilon}\left\{
\frac{1}{\varepsilon^{N}}\int_{B\left(  x_{0},\varepsilon\right)  }\left\vert
u\left(  x\right)  -u\left(  x_{0}\right)  -\nabla u\left(  x_{0}\right)
\left(  x-x_{0}\right)  \right\vert ^{\frac{N}{N-1}}dx\right\}  ^{\frac
{N-1}{N}}=0,\medskip\\
\lim\limits_{\varepsilon\rightarrow0^{+}}\frac{1}{\varepsilon}\left\{
\frac{1}{\varepsilon^{N}}\int_{B\left(  x_{0},\varepsilon\right)  }\left\vert
v\left(  x\right)  -v\left(  x_{0}\right)  \right\vert ^{\frac{N}{N-1}%
}dx\right\}  ^{\frac{N-1}{N}}=0,\medskip\\
\mu_{a}\left(  x_{0}\right)  =\lim\limits_{\varepsilon\rightarrow0^{+}}%
\frac{\mu\left(  B\left(  x_{0},\varepsilon\right)  \right)  }{\mathcal{L}%
^{N}\left(  B\left(  x_{0},\varepsilon\right)  \right)  }<\infty.
\end{array}
\label{upper1}%
\end{equation}

Let $U:=\left( v,u\right) .$ By $\left(  \ref{upper1}\right)  $ and Theorems \ref{thm2.6BBBF} and \ref{thm2.8FM2}
for $\mathcal{L}^{N}-$a.e. $x_{0}%
\in\Omega$ we have%
\begin{align}
&  \lim_{\varepsilon\rightarrow0^{+}}\frac{1}{\mathcal{L}^{N}\left(  B\left(
x_{0},\varepsilon\right)  \right)  }\int_{B\left(  x_{0},\varepsilon\right)
}\left\vert U\left(  x\right)  -U\left(  x_{0}\right)  \right\vert \left(
1+\left\vert \nabla U\left(  x\right)  \right\vert \right)  dx=0,\nonumber\\
&  \lim_{\varepsilon\rightarrow0^{+}}\frac{\left\vert D_{s}U\right\vert
\left(  B\left(  x_{0},\varepsilon\right)  \right)  }{\mathcal{L}^{N}\left(
B\left(  x_{0},\varepsilon\right)  \right)  }=0,\nonumber\\
&  \lim_{\varepsilon\rightarrow0^{+}}\frac{\left\vert DU\right\vert \left(
B\left(  x_{0},\varepsilon\right)  \right)  }{\mathcal{L}^{N}\left(  B\left(
x_{0},\varepsilon\right)  \right)  }\text{ exists and it is finite,}%
\label{upper3}\\
&  \lim_{\varepsilon\rightarrow0^{+}}\frac{1}{\mathcal{L}^{N}\left(  B\left(
x_{0},\varepsilon\right)  \right)  }\int_{B\left(  x_{0},\varepsilon\right)
}Qf\left( v\left(  x_{0}\right)  ,\nabla u\left(  x\right)  \right)
dx=Qf\left( v\left(  x_{0}\right)  ,\nabla u\left(  x_{0}\right)  \right)
,\nonumber\\
&  \frac{d\mathcal{F}\left(  v,u;\cdot\right)  }{d\mathcal{L}^{N}}\left(
x_{0}\right)  \text{ exists and it is finite.}\nonumber
\end{align}

We  observe that the assumptions imposed on $f$ and Proposition \ref{continuityQfbar} 
allow us to apply for  every $v \in SBV_0(\Omega;\mathbb R^m)$ the Global Method (cf. \cite[Theorem 4.1.4]{BFMGlobal}) to the functional
$u \in W^{1,1}(\Omega;\mathbb R^d) \times {\cal A}(\Omega)\to G(u;A):= \int_A Qf(v(x),\nabla u(x))dx$, thus obtaining an integral representation for the relaxed functional 
\begin{equation}\label{auxrelax}
\displaystyle{{\cal G}(u;A) = \inf\left\{\liminf_{n \to \infty} G(u_n; A): u_n \to u \hbox{ in } L^1(A;\mathbb R^d)\right\}}
\end{equation}
for every $ (u,A)\in BV(\Omega;\mathbb R^d)\times {\cal A}(\Omega)$.

Recall that the growth condition $(G_2)$, the lower semicontinuity with respect to the $L^1$-topology of the functional $v \in SBV_0(\Omega;\mathbb R^m)\mapsto ( (1+[v]){\cal H}^{N-1}\lfloor{(J_v \cap A)}$ entails
\begin{equation}\label{FestG}
{\cal F}(v, u;A)\leq {\cal G}(u;A)+ (1+[v]){\cal H}^{N-1}\lfloor{(J_v \cap A)}, 
\end{equation}

Differentiating with respect to ${\cal L}^N$ at $x_0$ and exploiting \eqref{upper1} and \eqref{upper3} we obtain that
$$
\displaystyle{\frac{d{\cal F}((v, u);\cdot)}{d {\cal L}^N}(x_0) \leq f_0(x_0, \nabla u(x_0))},
$$
where for every $x_0 \in \Omega$ and $\xi \in \mathbb R^d$, $f_0(x_0,\xi)$ is given as in \cite[formula (4.1.5)]{BFMGlobal}, namely
\begin{equation}\label{f00}
f_0(x_0,\xi):=\limsup_{\varepsilon \to 0^+} \inf_{\begin{array}{ll}
z\in W^{1,1}(Q;\mathbb R^d)\\
z(y)= \xi y \hbox{ on }\partial Q
\end{array}}\left\{\int_Q Qf(v(x_0+\varepsilon y), \nabla z(y))dy\right\}.
\end{equation}
To conclude the proof  we claim that $f_0(x_0,\xi) \leq Qf(v(x_0),\xi)$ for every $x_0 \in \Omega$ satisfying \eqref{upper1} and \eqref{upper3} and $\xi \in \mathbb R^d$.

By virtue of Lemma \ref{Lemma0} we have that
$$
\begin{array}{ll}
\displaystyle{\limsup_{\varepsilon \to 0^+} \inf_{\begin{array}{ll}
z\in W^{1,1}(Q;\mathbb R^d)\\
z(y)= \xi y \hbox{ on }\partial Q
\end{array}}\left\{\int_Q Qf(v(x_0+\varepsilon y), \nabla z(y))dy\right\}} \\
\leq
\displaystyle{\inf_{\begin{array}{ll}
z\in W^{1,1}(Q;\mathbb R^d)\\
z(y)= \xi y \hbox{ on }\partial Q 
\end{array}}\left\{\limsup_{\varepsilon \to 0^+}\int_Q Qf(v(x_0+\varepsilon y), \nabla z(y))dy\right\}.}
\end{array}
$$
Computing the $\limsup$ on the right hand side, we have
$$
\begin{array}{ll}
\displaystyle{\limsup_{\e \to 0^+} \int_Q Qf(v(x_0+ \varepsilon y), \nabla z(y))dy}\\
\displaystyle{=\limsup_{\e \to 0^+} \left(\int_Q Qf(v(x_0+ \varepsilon y), \nabla z(y))dy - \int_Q Qf(v(x_0), \nabla z(y))dy\right) + \int_Q Qf(v(x_0), \nabla z(y))dy.}
\end{array}
$$
Since $x_0$ is a Lebesgue point for $v$, and recalling that $v \in SBV_0(Q;\mathbb R^m)\cap L^\infty(Q;\mathbb R^m)$, by Lebesgue dominated convergence theorem and $(F_3)$ applied to $Qf$ (see Proposition \ref{continuityQfbar}),  we have that
$$
\begin{array}{ll}
\displaystyle{\limsup_{\e \to 0^+} \left(\int_Q Qf(v(x_0+ \varepsilon y), \nabla z(y))dy - \int_Q Qf(v(x_0), \nabla z(y))dy\right)}\\
\displaystyle{ \leq \limsup_{\e \to 0^+} \int_Q L|v(x_0+ \varepsilon y)- v(x_0)|(1+ |\nabla z(y)|)dy =0.}
\end{array}
$$
Hence 
$$
\displaystyle{\limsup_{\e \to 0^+} \int_Q Qf(v(x_0+ \varepsilon y), \nabla z(y))dy=\int_Q Qf(v(x_0), \nabla z(y))dy.}
$$
By the quasiconvexity of $Qf(v(x_0),\cdot)$, and \eqref{f00} one obtains
$$
f_0(x_0,\xi) \leq Qf(v(x_0), \xi),
$$
which concludes the proof, when replacing $\xi$ by $\nabla u(x_0)$.

\noindent\textbf{Step 2.} We prove the upper bound for the Cantor part.

By Radon-Nikod\'ym theorem we can write
\begin{equation}\label{CantorU}
\left\vert DU\right\vert =\left\vert D^{c}u\right\vert +\sigma
\end{equation}
where $U:=(v,u) \in (SBV_0(\Omega;\mathbb R^m) \times BV(\Omega;\mathbb R^d))\cap L^\infty(\Omega;\mathbb R^{m+d})$, $\sigma$ and $\left\vert D^{c}u\right\vert $ are mutually singular Radon
measures.

Observe that $U \equiv \left(v,u\right)  $ is $\left\vert D^{c}u\right\vert -$measurable,
$D v$ is singular with respect to $\left\vert D^{c}u\right\vert $ and by
Theorems \ref{thm2.6BBBF}, \ref{thm2.8FM2}, and \cite[Theorem 2.11]{FM2} for $\left\vert D^{c}u\right\vert -$a.e.
$x\in B\left(  x_{0},\varepsilon\right)  $%
\begin{equation}\label{cantor}
\begin{array}{ll}
\displaystyle{\lim_{\varepsilon\rightarrow0^{+}}\frac{\mu\left(  B\left(  x_{0}%
,\varepsilon\right)  \right)  }{\left\vert D^{c}u\right\vert \left(  B\left(
x_{0},\varepsilon\right)  \right)  }=0,}
\\
\\
\displaystyle{\lim_{\varepsilon\rightarrow0^{+}}\frac{\left\vert Du\right\vert \left(
B\left(  x_{0},\varepsilon\right)  \right)  }{\left\vert D^{c}u\right\vert
\left(  B\left(  x_{0},\varepsilon\right)  \right)  }\text{ exists and is
finite }}
\\
\\
\displaystyle{\lim_{\varepsilon\rightarrow0^{+}}\frac{\varepsilon^{N}}{\left\vert D^{c}u\right\vert\left(  B\left(  x_{0},\varepsilon\right)  \right)  }=0,}\\
\\
\displaystyle{\lim_{\varepsilon\rightarrow0^{+}}\frac{1}{{\cal L}^N 
\left(  B\left(  x_{0},\varepsilon\right)  \right)  }\int_{B\left(
x_{0},\varepsilon\right)  }\left(\left\vert u\left(  x\right)  -u\left(
x_{0}\right)  \right\vert +\left\vert v\left(  x\right)  -v\left(
x_{0}\right)  \right\vert\right) d x =0.}
\end{array}
\end{equation}

Moreover,
\begin{equation}\label{ABrankone}
\displaystyle{A\left(  x\right)  :=\lim_{\varepsilon\rightarrow0^{+}}\frac{D^{c}u\left(
B\left(  x,\varepsilon\right)  \right)  }{\left\vert D^{c}u\right\vert \left(
B\left(  x,\varepsilon\right)  \right)  },\text{~\ \ }\lim_{\varepsilon
\rightarrow0^{+}}\frac{D^{c}U\left(  B\left(  x,\varepsilon\right)  \right)
}{\left\vert D^{c}U\right\vert \left(  B\left(  x,\varepsilon\right)  \right)
}=:D\left(  x\right)}
\end{equation}
exist and they are rank-one matrices of norm 1, in particular
\begin{equation}\label{as3.27BFMglobal}
\displaystyle{A(x)= a_u(x) \otimes \nu_u(x)},
\end{equation} 
where $(a_u(x),\nu_u(x)) \in \mathbb R^d \times S^{N-1}$.
By Theorem \ref{thm2.8FM2} we have
\[
\lim_{\varepsilon\rightarrow0^{+}}\frac{1}{\left\vert D^{c}u\right\vert
\left(  B\left(  x_{0},\varepsilon\right)  \right)  }\int_{B\left(
x_{0},\varepsilon\right)  }f^{\infty}\left(  v\left(  x_{0}\right)
,A\left(  x\right)  \right)  d\left\vert D^{c}u\right\vert =f^{\infty}\left(
v\left(  x_{0}\right)  ,A\left(  x_{0}\right)  \right)  .
\]

We recall as in Step 1, that 
via the Global Method (cf. \cite[Theorem 4.1.4]{BFMGlobal}) we can obtain an integral representation for the functional
 ${\cal G}(u; A)$ in \eqref{auxrelax} for every $ (v,u)\in BV(\Omega;\mathbb R^{m+d})$.
Moreover by  Proposition \ref{propqcx}, we can replace $f$ by $Qf$ in \eqref{calFG} and \eqref{FestG} holds. 

Differentiating with respect to $|D^c u|$ at $x_0$ and exploiting 
\eqref{CantorU} and \eqref{cantor} we deduce
$$
\displaystyle{\frac{d{\cal F}((v,u);\cdot)}{d |D^c u|}(x_0) \leq h(x_0, a_u, \nu_u)},
$$
where $\nu_u(x)$ agrees with the unit vector that, together with $a_u$, satisfies \eqref{as3.27BFMglobal} for $|D^c u|$-a.e. $x \in \Omega \setminus J_u$, and
where $h(x_0,a, \nu)$ is given as in \cite[formula (4.1.7)]{BFMGlobal}, namely
\begin{equation}\label{f0}
h(x_0,a,\nu):=\limsup_{k \to \infty}\limsup_{\varepsilon \to 0^+} \inf_{\begin{array}{ll}
z\in W^{1,1}(Q^{(k)}_\nu;\mathbb R^d)\\
z(y)= a(\nu \cdot y) \hbox{ on }\partial Q^{(k)}_\nu
\end{array}}\left\{\frac{1}{k^{N-1}}\int_{Q^{(k)}_\nu} Qf^\infty(v(x_0+\varepsilon y), \nabla z(y))dy\right\},
\end{equation}
where $a \in \mathbb R^d$, $\nu \in S^{N-1}$, $Q_\nu^{(k)}:= R_\nu \left(\left(-\frac{k}{2},\frac{k}{2}\right)^{N-1}\times \left(-\frac{1}{2},\frac{1}{2}\right)\right),$
and $R_\nu$ is a rotation such that $R_\nu(e_N)=\nu$.

We also recall that by (iv) in Remark \ref{propfinfty}, $Q(f^\infty)= (Qf)^\infty= Qf^\infty$.

To conclude the proof it is enough to show that
$$
h(x_0, a,\nu) \leq Qf^\infty (v(x_0), a \otimes \nu).
$$

By Lemma \ref{Lemma0}
\begin{equation}\label{hCantor}
\begin{array}{ll}
\displaystyle{h(x_0, a, \nu) \leq \limsup_{k \to \infty}\inf_{\begin{array}{ll}
z\in W^{1,1}(Q^{(k)}_\nu;\mathbb R^d)\\
z(y)= a(\nu \cdot y) \hbox{ on }\partial Q^{(k)}_\nu
\end{array}} \left\{\limsup_{\varepsilon \to 0^+} \frac{1}{k^{N-1}}\int_{Q_ \nu^{(k)}} Q f^\infty(v(x_0+\varepsilon y), \nabla z(y))dy\right\}.}
\end{array}
\end{equation}
In order to compute $\displaystyle{\limsup_{\varepsilon \to 0^+} \frac{1}{k^{N-1}}\int_{Q_\nu^{(k)}}Q f^\infty(v(x_0+\varepsilon y), \nabla z(y))dy}$, we add and subtract inside the integral $ Qf^\infty(v(x_0), \nabla z(y))$. Then, as in Step 1, exploiting the fact that $x_0$ is a Lebesgue point for $v\in SBV_0(\Omega;\mathbb R^m)\cap L^\infty(\Omega;\mathbb R^m)$, and that $Qf^\infty$ satisfies $(F_3)$ (see  Remark \ref{propfinfty} where $(F_3)$ has been deduced for $f^\infty$ and Proposition \ref{continuityQfbar}), via Lebesgue dominated convergence theorem, we can conclude that
$$
\displaystyle{\limsup_{\varepsilon \to 0^+}\frac{1}{k^{N-1}}\int_{Q_\nu^{(k)}} Qf^\infty((v(x_0+\varepsilon y), \nabla z(y)) dy)= \frac{1}{k^{N-1}} \int_{Q_\nu^{(k)}}Qf^\infty(v(x_0), \nabla z(y))dy.}
$$
Finally the quasiconvexity of $Qf^\infty$ (deduced via Remark \ref{propfinfty} and Proposition \ref{continuityQfbar}) provides
$$
\displaystyle{ Qf^\infty (v(x_0), a \otimes \nu) = \inf_{\begin{array}{ll}
z\in W^{1,1}(Q^{(k)}_\nu;\mathbb R^d)\\
z(y)= a(\nu \cdot y) \hbox{ on }\partial Q^{(k)}_\nu
\end{array}} \left\{\frac{1}{k^{N-1}}\int_{Q_ \nu^{(k)}} Q f^\infty(v(x_0), \nabla z(y))dy\right\}},
$$ 
which, together with \eqref{hCantor} concludes the proof of the upper bound for the Cantor part when $(v,u) \in (SBV_0(\Omega;\mathbb R^m)\times BV(\Omega;\mathbb R^d))\cap L^\infty(\Omega;\mathbb R^{m+d})$.

\noindent\textbf{Step 3.} We prove the upper bound for the jump. Namely, we
claim that
\begin{equation}\label{claimjump}
\mathcal{F}\left(  U;J_U  \right)\equiv{\cal F}(v,u, J_{(v,u)})  \leq\int_{J_ U}K_3\left(  v^+,v^-,u^+,u^-,\nu\right)  d\mathcal{H}^{N-1}%
\end{equation}
for every $U\equiv \left( v,u\right)  \in \left( SBV_0\left(  \Omega;\mathbb R^m\right)\times  BV\left(  \Omega;\mathbb{R}^{d}\right) \right)  \cap L^{\infty}\left(\Omega;\mathbb{R}^{m+ d}\right)  $.

The proof is divided into three parts according to the assumptions on the limit function $U.$

\emph{Case 1- }$U\left(  x\right)  :=\left(  a,c\right)  \chi_{E}\left(
x\right)  +\left(  b,d\right)  \left(  1-\chi_{E}\left(  x\right)  \right)  $
with $P( E, \Omega)  <\infty.$

\emph{Case 2- }$U\left(  x\right)  :=\sum_{i=1}^{\infty}(a_{i},c_i)\chi_{E_{i}%
}\left(  x\right)  $ where $\left\{  E_{i}\right\}  _{i=1}^{\infty}$ forms a
partition of $\Omega$ into sets of finite perimeter and $(a_i,c_i) \in \mathbb R^m \times \mathbb R^d$.

\emph{Case 3- }$U\in (SBV_0(  \Omega;\mathbb R^m)\times  BV(
\Omega;\mathbb{R}^{d}))  \cap L^{\infty}\left(  \Omega;\mathbb{R}%
^{m+d}\right)  .$

{\it Case 1-} We start by proving that for every open set $A\subset\Omega$%
\[
\mathcal{F}\left(  U;A\right)\equiv{\cal F}(v,u;A)  \leq\int_{A}Qf\left(  v\left(  x\right)
,0\right)  dx+\int_{J_U  \cap A}K_3\left(  a,b,c,d,\nu%
\right)  d\mathcal{H}^{N-1}.
\]

\begin{enumerate}
\item[a)] Assume first that%
\[
v(x)  :=\left\{
\begin{array}
[c]{cc}%
a & \text{if }x\cdot\nu>0,\\
b   & \text{if }x\cdot\nu<0,
\end{array}
\right.  \text{ and }u\left(  x\right)  :=\left\{
\begin{array}
[c]{cc}%
c & \text{if }x\cdot\nu>0,\\
d  & \text{if }x\cdot\nu<0.
\end{array}
\right.
\]
We start with the case when $A=a+\lambda Q$ is an open cube with two faces
orthogonal to $\nu$,  for simplicity we also assume that $\nu =e_N$ and $Q_\nu$ will be denoted simply by $Q$. Our proof develops as in  \cite[Proposition 4.1 and Lemma 4.2]{FR}, cf. also \cite[Proposition 5.1]{BBBF}, thus we will present just the main steps.

Suppose first that $a=0$ and $\lambda =1$. 
By Proposition \ref{prop3.5BBBF} (cf. also Remark \ref{applicationofprop3.4}), there exists $(v_n,u_n) \in {\cal A}_3(a,b,c,d,\nu)$ such that $(v_n,u_n)\to (v,u)$ in $L^1(Q;\mathbb R^{m+d})$ and
\begin{equation}\label{5.9BBBF}
\displaystyle{K_3(a,b,c,d,\nu) =\lim_{n\to \infty}\left(\int_Q Qf^\infty(v_n(x), \nabla u_n(x))dx + \int_{J_{v_n}\cap Q} g(v_n^+(x), v_n^-(x), \nu_n(x))d {\cal H}^{N-1}\right).}
\end{equation}

We denote by $Q'$ the set $\left\{  x\in Q:x_{N}=0\right\}  .$ For
$k\in\mathbb{N}$ we label the elements of $\left(  \mathbb{Z\cap}\left[
-k,k\right]  \right)  ^{N-1}\times\left\{  0\right\}  $ by $\left\{
a_{i}\right\}  _{i=1}^{\left(  2k+1\right)  ^{N-1}}$ and we observe that%
\[
\left(  2k+1\right)  \overline{Q'}=%
{\displaystyle\bigcup\limits_{i=1}^{\left(  2k+1\right)  ^{N-1}}}
\left(  a_{i}+\overline{Q'}\right)
\]
with
\[
\left(  a_{i}+Q'\right)  \cap\left(  a_{j}+Q'\right)  =\emptyset\text{ for
}i\neq j.
\]

We define%
\[
z_{n,k}\left(  x\right)  :=\left\{
\begin{array}
[c]{lll}%
a &  & \text{if }x_{N}>\frac{1}{2\left(  2k+1\right)  },\\
v_n\left(  \left(  2k+1\right)  x\right)   &  & \text{if }\left\vert
x_{N}\right\vert <\frac{1}{2\left(  2k+1\right)  },\\
b &  & \text{if }x_{N}<-\frac{1}{2\left(  2k+1\right)  }.
\end{array}
\right.
\]
and%
\[
w_{n,k}\left(  x\right)  :=\left\{
\begin{array}
[c]{lll}%
c &  & \text{if }x_{N}>\frac{1}{2\left(  2k+1\right)  },\\
u_n\left(  \left(  2k+1\right)  x\right)   &  & \text{if }\left\vert
x_{N}\right\vert <\frac{1}{2\left(  2k+1\right)  },\\
d &  & \text{if }x_{N}<-\frac{1}{2\left(  2k+1\right)  }.
\end{array}
\right.
\]
By the periodicity of the functions $v_n$ and $u_n$, it is easily seen that 
$$
\displaystyle{\lim_{n \to \infty}\lim_{k \to \infty}\|z_{n,k}-v\|_{L^{1}\left(
Q;\mathbb R^m\right) }=0 }, \;\;\;\;\;\;\;\;\;\;\displaystyle{\lim_{n \to \infty}\lim_{k \to \infty}\|w_{n, k}- u \|_{L^{1}\left(
Q;\mathbb{R}^{d}\right) }=0 .}
$$
\noindent Thus, by a standard diagonalization argument, we have
$$
\begin{array} {ll}
\displaystyle{{\cal F}(v,u; Q)\leq \limsup_{n\to \infty}\limsup_{k\to \infty}\left(\int_Q Qf (z_{n,k}(x),\nabla w_{n,k}(x))dx
+\int_{Q\cap J_{z_{n,k}}} g(z_{n,k}^+(x), z_{n,k}^-(x), \nu_{n,k}(x))d {\cal H}^{N-1}\right).}
\end{array}
$$


Arguing as in \cite[Proposition 5.1]{BBBF} for the bulk part we have
$$
\begin{array}{ll}
\displaystyle{\limsup_{k \to \infty}\int_Q Qf(z_{n,k}(x), \nabla w_{n,k}(x))dx = \int_Q Qf(v(y), 0)dy+\int_Q Qf^\infty (v_n(y),\nabla u_n(y))dy,}
\end{array}
$$
and for the surface term
$$
\begin{array}{ll}
\displaystyle{\int_{Q \cap J_{z_{n,k}}} g(z_{n,k}^+(x), z_{n,k}^-(x), \nu_{n,k}(x))d {\cal H}^{N-1}}
\displaystyle{\leq \int_{Q \cap J_{v_n}} g(v_n^+(y), v_n^-(y), \nu_n(y))d {\cal H}^{N-1}(y).}
\end{array}
$$
Putting together the estimates for bulk and surface terms and exploiting \eqref{5.9BBBF} we obtain that
$$
\begin{array}{ll}
\displaystyle{{\cal F}(v,u; Q) \leq \limsup_{n\to \infty} \left(\int_Q Qf(v, 0)dx + \int_Q Qf^\infty(v_n(y), \nabla u_n(y)) dy \right.}\\
\\
\displaystyle{\left.+\int_{Q \cap J_{v_n}} g(v_n^+(y), v_n^-(y), \nu_n(y))d {\cal H}^{N-1} \right) = \int_Q Qf(v(x), 0)dx  +K_3(a,b,c,d, e_N)}\\
\\
\displaystyle{=\frac{Qf(a,0)+ Qf(b,0)}{2}+ K_3(a,b,c,d,e_N).}
\end{array}
$$

In order to consider sets $A= a + \lambda Q$ with $a \in \mathbb R^N$  and $\lambda>0$ we define
$$
\displaystyle{(Qf)_\lambda(b,B):= Qf\left(b, \frac{B}{\lambda}\right), \;\;\;\ g_\lambda(\xi, \zeta, \nu):=\frac{1}{\lambda}g(\xi, \zeta ,\nu) }
$$
and for every $E \subset \Omega$,
$$
\begin{array}{ll}
\displaystyle{{\cal F}_\lambda(v,u; E):= \inf_{\{(v_n, u_n)\}} \left\{ \liminf_{n \to \infty} \left(\int_E (Qf)_\lambda(v_n(x), \nabla u_n(x) )dx + \int_{E \cap J_{v_n}} g_\lambda(v_n^+(x), v_n^-(x), \nu_n(x)) d {\cal H}^{N-1} \right):\right.}
\\
\\
\displaystyle{ (v_n, u_n) \in SBV_0(E;\mathbb R^m) \times W^{1,1}(E;\mathbb R^d), (v_n, u_n)\to (v,u) \hbox{ in }L^1(E;\mathbb R^{m+d})\Big\}.}\end{array}
$$

It is easily seen that for every $(v,u) \in L^1(\Omega ;\mathbb R^{m+d})$, we have
$$
\displaystyle{{\cal F}(v,u; A)= \lambda^N {\cal F}_\lambda (v_\lambda, u_\lambda; Q)},
$$

where
$$
\displaystyle{v_\lambda(x):=v\left(\frac{x-a}{\lambda}\right)}, \;\;
\displaystyle{u_\lambda(x):=u\left(\frac{x-a}{\lambda}\right)}.
$$
Since $Qf^\infty_\lambda = \frac{1}{\lambda}Qf^\infty$, by the definition of $K_3$ for $f_\lambda$ and $g_\lambda$ we have that
$(K_3)_\lambda(a,b,c,d,\nu)= \frac{1}{\lambda} K_3(a,b,c,d,\nu). $

By the definition of $u_\lambda$ and $v_\lambda$  we have that 
$$
v_\lambda= \left\{
\begin{array}{ll}
a \hbox{ if }x_N >0 ,\\
b \hbox{ if }x_N <0,
\end{array}
\right.\;\;\;\;\;\; u_\lambda= \left\{
\begin{array}{ll}
c \hbox{ if }x_N >0 ,\\
d \hbox{ if }x_N <0.
\end{array}
\right.
$$

So by the previous case it results that

$$
\displaystyle{{\cal F}(v,u; A)\lambda^N ={\cal F}_\lambda (v_\lambda, u_\lambda; Q) \leq \lambda^N \left( \displaystyle{\frac{Qf_\lambda(a,0)+ Qf_\lambda(b,0)}{2}+ (K_3)_\lambda(a,b,c,d,e_N)}\right).}
$$

\item[b)] Now let $U:=(v,u)$ as in $a)$ and let $A$ be any open set. 
The proof of this step is identical to \cite[Section 5. Step 3, case 1., b)]{FM2}. Indeed it is enough to apply the same strategy replacing $u$ and $K$ in \cite{FM2} by $U $ and $K_3$ respectively herein, obtaining

\begin{equation}\label{formulasurface}
\displaystyle{\mathcal{F}\left( v,u;A\right)\leq \int_{A}Qf\left( v\left(  x\right)  ,0\right)  dx+\int_{J_U  \cap A}K_3\left(  a,b,c,d,\nu\right)  d\mathcal{H}^{N-1}.}
\end{equation}

\item[c)] Now suppose that $U$ has a polygonal interface, i.e.
$ U=\left(  a,c\right)  \chi_{E}+\left(  b,d\right)  \left(  1-\chi_{E}\right)$, $E$ is a polyhedral set, i.e., $E$ is a
bounded strongly Lipschitz domain and $\partial E=H_{1}\cup H_{2}\cup\dots\cup
H_{M}$ are closed subsets of hyperplanes of type $\left\{  x\in\mathbb{R}%
^{N}:x\cdot\nu_{i}=\alpha_{i}\right\}  .$ 

The details of the proof are omitted since they are very similar to  \cite[Section 5, Step 3, case 1, c)]{FM2} . We just observe that, given an open set $A $ contained in $\Omega$, the argument relies on an inductive procedure on  $I:=\left\{  i\in\left\{  1,\dots,M\right\}  :\mathcal{H}^{N-1}\left(  H_{i}\cap A\right)  >0\right\}$, starting from the case $I=0$, when  $u\in W^{1,1}\left(
A;\mathbb{R}^{d}\right)  $ and $v\in SBV_0(A;\mathbb R^m)\cap L^\infty(A;\mathbb R^m)$, for which it suffices to consider $u_{n}=u$ and $v_{n}=v$ with \eqref{formulasurface} reducing to
\[
\mathcal{F}\left( v,u;A\right)  \leq\int_{A}Qf\left(v\left(
x\right)  ,0\right)  dx.
\]

The case $\operatorname*{card}I=1$ was studied in part $b)$ where $E$ is a
large cube so that $J_U \cap\Omega$ reduces to the flat
interface $\left\{  x\in\Omega:x\cdot\nu=0\right\}  .$ 

Then the induction step, which first assumes that \eqref{formulasurface} is true if
$\operatorname*{card}I=k,~k\leq M-1$ and then proves that it is still true if
$\operatorname*{card}I=k,$ 
develops exactly as in \cite[Proposition 5.1, Step 2, c)]{BBBF}, the only difference being that the slicing method used to connect the sequence across the interfaces relies on the same techniques as Lemma \ref{Lemma4.1FM}, referred to more general open sets  than cubes (cf. also \cite[Section 5, Step 3, case 1, c]{FM2}). 
Thus one can conclude that
$$
\begin{array}{ll}
{\cal F}(v,u; A)\leq 
\displaystyle{ \int_A Qf(v(x), 0)dx + \int_{J_U \cap A} K_3(a,b,c,d,\nu)d {\cal H}^{N-1}.}
\end{array}
$$

\item[d)] If $E$ is an arbitrary set of finite perimeter, the step develops in strong analogy with \cite [Section 5, Step 3, case 1, f)]{FM2}. 
Essentially, exploiting Proposition \ref{propK3} (b) and the approximation via polyhedral sets with finite perimeter as in \cite[Lemma 3.1]{B}, and application of  Lebesgue's monotone convergence theorem gives 



$$
\displaystyle{{\cal F}(v,u; A) \leq \int_A Qf(v(x), 0)dx + \int_{A \cap J_U}K_3(a,b,c,d,\nu)d{\cal H}^{N-1}},
$$
This last inequality, together with Lemma \ref{measure}, yields
$$
\displaystyle{{\cal F}(v,u; J_{(v,u)}) \leq 
\int_{J_{(v,u)}} K_3(a,b,c,d, \nu)d {\cal H}^{N-1}}
$$
\noindent which gives \eqref{claimjump} when $U\equiv (v,u)=(a,c)\chi_E + (b,d)(1-\chi_E)$ is the characteristic function of a set of finite perimeter.
\end{enumerate}

{\it Case 2-} Arguing as in \cite [Section 5, Step 3, case 2]{FM2}, we refer to  \cite[Proposition 4.8, Step 1]{AMT}, and clearly we obtain for every $(v,u)\in BV(\Omega; T) \times BV(\Omega;T)$, with $T$ a finite subset of $\mathbb R^d$ 
$$
\displaystyle{{\cal F}(v,u; A)= {\cal F}(v,u; A \cap J_{(v,u)}) \leq \int_{J_{(v,u)}}K_3(v^+, v^-, u^+, u^-, \nu_{v,u}(x))d {\cal H}^{N-1}(x).}
$$

{\it Case 3-} For $U\equiv(v,u)\in (SBV_0(\Omega;\mathbb R^m) \times BV(\Omega;\mathbb R^d) )\cap L^\infty(\Omega;\mathbb R^{m+d})$, the proof develops analogously to \cite[Proposition 4.8, Step 2]{AMT} and we add some details for the reader's convenience.

First we observe that the jump set $J_U \equiv J_{(v,u)}$ can be decomposed as $(J_{u} \setminus J_v) \cup (J_v \setminus J_u) \cup (J_u\cap J_v) $, recalling that these sets are mutually disjoint and the tangent hyperplanes to $J_u$ and $J_v$ coincide up to a set of ${\cal H}^{N-1}$- measure $0$.



 
Let $A \in {\cal A}(\Omega)$, such that $A \supset J_U$, we assume $U(x) \in [0,1]^{m+d}$  for a.e. $x \in A$. 
For every $h \in \mathbb N$, $h \geq 2$, it is possible to define a set $B_h:= A \setminus J_U \cup \{ x \in J_U : |U^+(x)-U^-(x)| \leq  \frac{1}{4(m+d)h}\}$, and define the sequence $\{U_h\}\equiv \{(v_h, u_h)\}$ according to \cite[Proposition 4.8, Step 2]{AMT}. Observe that $J_{v_h}\subset J_v$.
Then, by Step 2, we have that
\begin{equation}\label{Cantor0}
\begin{array}{ll}
\displaystyle{{\cal F}(v,u,;A )\leq \liminf_{h \to \infty} {\cal F}(v_h, u_h;A) }
\displaystyle{=\liminf_{h \to \infty} \left(\int_A Qf(v_h, 0)dx + \int_A Qf^\infty \left(v_h, \frac{d D^c u_h}{d |D^c u_h|}\right) d |D^c u_h| \right.}\\
\\
\;\;\;\;\;\;\;\;\;\;\;\;\,\,\,\;\;\;\;\;\;\;\,\;\;\;\;\;\;\;\;\;\;\;\;\;\;\;\,\,\,\;\;\;\;\;\;\;\,\;\;\;\displaystyle{\left. +\int_ {A\cap (J_{u_h}\cup J_{v_h})} K_3(v_h^+, v_h^-, u_h^+, u_h^-, \nu_{v_h, u_h})d {\cal H}^{N-1} \right).}
\end{array}
\end{equation}
We restrict our attention to the surface integral.
Clearly,
$$
\begin{array}{ll}
\displaystyle{\int_ {A\cap (J_{u_h}\cup J_{v_h})} K_3(v_h^+, v_h^-, u_h^+, u_h^-, \nu_{v_h, u_h})d {\cal H}^{N-1}= \int_ {A\cap (J_{u_h}\cup J_{v_h})\cap B_h} K_3(v_h^+, v_h^-, u_h^+, u_h^-, \nu_{v_h, u_h})d {\cal H}^{N-1}}\\
\\
\displaystyle{+ \int_ {A\cap (J_{u_h}\cup J_{v_h})\cap (A \setminus B_h)} K_3(v_h^+, v_h^-, u_h^+, u_h^-, \nu_{v_h, u_h})d {\cal H}^{N-1}.}
\end{array}
$$

By the decomposition of the jump set $J_{(v_h, u_h)}$,  Proposition \ref{propK3} d), the fact that $J_{v_h}\subset J_v$, the same type of estimates as in \cite[page 300]{AMT}, entail (with the constant $C$ varying from place to place)
\begin{equation}\label{stimeuhvh}
\begin{array}{ll}
 \displaystyle{\int_ {A\cap (J_{u_h}\cup J_{v_h})\cap  B_h} K_3(v_h^+, v_h^-, u_h^+, u_h^-, \nu_{v_h, u_h})d {\cal H}^{N-1}= \int_ {A\cap (J_{u_h}\setminus J_{v_h})\cap B_h} K_3(v_h^+, v_h^-, u_h^+, u_h^-, \nu_{v_h, u_h})d {\cal H}^{N-1}}\\
\\
\displaystyle{ +\int_ {A\cap (J_{v_h}\setminus J_{u_h})\cap B_h} K_3(v_h^+, v_h^-, u_h^+, u_h^-, \nu_{v_h, u_h})d {\cal H}^{N-1}+  \int_ {A\cap J_{u_h}\cap J_{v_h} \cap  B_h} K_3(v_h^+, v_h^-, u_h^+, u_h^-, \nu_{v_h, u_h})d {\cal H}^{N-1}}\\
\\
\displaystyle{\leq C \int_ {A\cap (J_{u_h}\setminus J_{v_h})\cap  B_h} |u_h^+-u_h^-|d {\cal H}^{N-1}+ C \int_ {A\cap (J_{v_h}\setminus J_{u_h})\cap B_h} (|v_h^+- v_h^-| +1) d {\cal H}^{N-1}}\\
\\
\displaystyle{+ C  \int_ {A\cap J_{u_h}\cap J_{v_h}\cap B_h} \left( |v_h^+- v_h^-| + | u_h^+- u_h^-| +1 \right)d {\cal H}^{N-1}}\\
\\
\displaystyle{\leq 2C (m+d) |Du| (A \cap B_h)+ C (m+d)|Dv|(A \cap B_h) + C{\cal H}^{N-1}(J_v \cap B_h\cap A),}
\end{array}
\end{equation}.
Moreover, by Proposition \ref{propK3} c), d) and reverse Fatou's lemma we have  
$$
\displaystyle{ \int_ {(J_{v_h}\cup J_{u_h})\cap (A \setminus B_h)} K_3(v_h^+, v_h^-, u_h^+, u_h^-, \nu_{(v_h, u_h)})d {\cal H}^{N-1}\leq  \int_ {A\cap (J_v\cup J_u)} K_3(v^+, v^-, u^+, u^-, \nu_{(v, u)})d {\cal H}^{N-1}.}
$$

Clearly, taking the limit as $h \to \infty$,  from the above inequality and \eqref{stimeuhvh} we may conclude that,
$$
\begin{array}{ll}
\displaystyle{{\cal F}(v,u;A) \leq \int_{A \cap (J_v \cup J_u)}K_3(v^+, v^-, u^+, u^-, \nu_{(v,u)})d{\cal H}^{N-1} }
\\
\\\displaystyle{+C (|Du| (A \setminus (J_v \cup J_u))+ |Dv|(A \setminus (J_u \cup J_v)) +\int_{A} Qf(v, 0)dx, }
\end{array}
$$
where we have exploited the fact that the Cantor term in \eqref{Cantor0} is $0$, from the construction of the $u_h$, and 
$\displaystyle{\liminf_{h \to \infty}{\cal H}^{N-1}(J_v \cap B_h\cap A)= {\cal H}^{N-1}(J_v \cap (A \setminus (J_u \cup J_v))) =0}$.
Now, since ${\cal F}(v,u;\cdot )$ is a Radon measure, the above inequality holds for every Borel set $B$, in particular for the set $B= A \cap (J_v \cup J_u)$ and this gives
$$
\displaystyle{{\cal F}(v,u;J_v\cap J_u) \leq \int_{J_v\cap J_u} K_3(v^+, v^-, u^+, u^-, \nu_{(v,u)})d {\cal H}^{N-1}.}
$$
This concludes the proof of Step 2 when $(v,u) \in SBV_0(\Omega;\mathbb R^m)\times BV(\Omega;\mathbb R^d)\cap L^\infty(\Omega;\mathbb R^{m+d})$.

The general case $(v,u) \in SBV_0(\Omega;\mathbb R^m)\times BV(\Omega;\mathbb R^d)$ follows from (iii) in Remark  \ref{asinGlobalMethodBFLM}, (cf.\cite[Section 5, Step 4.]{FM2}  and \cite[Theorem 4.9]{AMT}). 

\end{proof}

\begin{proof}[Proof of Theorem \ref{mainthmgen}]
It follows from Theorems \ref{lsctheorem} and \ref{thupperbound} 
\end{proof}
\begin{remark}\label{specializeK3}We observe that, as it can be easily conjectured from the proof of Theorems \ref{lsctheorem}, Step 2, and \ref{thupperbound}, Step 3, Case 3. i) and ii), $K_3$ admits the following equivalent representation:
\begin{itemize}
\item[on $J_u\setminus J_v$] $K_3(a,a,c,d,\nu)=Qf^\infty(a,(c-d)\otimes \nu)$,  where $Q f^\infty$ represents the recession function of the quasiconvexification of $f$ as in Remark \ref{propfinfty}.
In fact one inequality is trivial by Definition \ref{K3}, while the other can be obtained through Proposition \ref{prop3.5BBBF},  invoking the quasiconvexity and the growth properties of $Qf^\infty (a, \cdot)$ (cf. Remark \ref{propfinfty}) and analogous arguments to the ones leading to \cite[formula (5.84)]{AFP2}.
 
\item[on $J_v \setminus J_u$] $K_3(a,b,c,c,\nu)= {\cal R}g(a,b,\nu)$ where ${\cal R}g$ represents the $BV$-elliptic envelope of $g$, namely the greatest $BV$-elliptic function less than or equal to $g$, which under the assumptions $(G_1)-(G_3)$ admits the representation 
\begin{equation}\label{Rg}
{\cal R}g(a, b,\nu) = \inf\left\{\int_{J_w\cap Q_\nu} g(w^+, w^-, \nu) d{\cal H}^{N-1} : w \in SBV_0(Q_\nu;\mathbb R^m)\cap  L^\infty(Q_\nu; \mathbb R^m), w = v_0  \hbox{ on }\partial Q_\nu\right\},
\end{equation}
 as in \cite{BDV}, \cite{CF}, \cite{BFLM},
where $v_0$ is defined as in \eqref{vab}. This is a consequence of \eqref{K3} and \eqref{Rg}.

\end{itemize}

We observe that the above characterizations of $K_3$ could be deduced directly reproducing the proof of lower bound and upper bound for Theorem \ref{mainthmgen}, for the jump part on the sets $J_u\setminus J_v$ and $J_v\setminus J_u$, respectively.
\end{remark}

\section{Applications}\label{appl}
This section is devoted to the proof of Theorem \ref{mainthm} which is very similar to that of Theorem \ref{mainthmgen}. In particular we replace Lemma \ref{Lemma4.1FM} and Proposition \ref{propK3} by Lemma \ref{Lemma4.1FMchi} and Proposition \ref{propK2}, respectively. Having in mind the application that we will describe in more details in Remark \ref{remmainthm} we state it with more generality, but  in order to prove Theorem \ref{mainthm} we will consider $m=1$ and $T=\{0,1\}$.

Let $T\subset \mathbb R^m$ be a finite set and let 
\begin{equation} \label{Vgfr}
V:T\times\mathbb{R}^{d\times N}\rightarrow(0,+\infty)
\hbox{ and }  g: T \times T \times S^{N-1} \to[0, +\infty[
\end{equation} 
satisfying $(F_1)$ - $(F_4)$ and $(G_1)$ - $(G_3)$, respectively, and denote by ${\cal A}_{fr}$ the set defined in \eqref{AFR}, where the range $\{0,1\}$ is replaced by $T$. 


For simplicity we will consider $\nu=e_N$ and consequently $Q_\nu=Q= [0,1]^N$.
\begin{lemma}\label{Lemma4.1FMchi}
Let $T \subset \mathbb R^m$ a finite set, and%
\[
 v_0\left(  y\right)  :=\left\{
\begin{array}
[c]{lll}%
a &  & \text{if }x_{N}>0,\\
b &  & \text{if }x_{N} < 0, 
\end{array}
\right.\qquad u_{0}\left(  y\right)  :=\left\{
\begin{array}
[c]{lll}%
c &  & \text{if }x_{N}>0,\\
d &  & \text{if }x_{N}< 0.
\end{array}
\right. 
\]
Let $\left\{  v_{n}\right\}  \subset
BV\left(  \Omega;T  \right) $ and $\{u_{n}\} \subset W^{1,1}\left(
Q;\mathbb{R}^{d}\right)  $, such that $v_n \to v_0$ $L^{1}\left(
Q;\mathbb R^m  \right)  $ and $u_n\to u_0$  in $L^{1}\left(  Q;\mathbb{R}^{d}\right) .$   

If $\rho$ is a mollifier,
$\rho_{n}:=n^{N}\rho\left(  nx\right)  ,$ then there exists a sequence of
functions $\left\{  \left(  \zeta_{n},\xi_{n}\right)  \right\}  \in
\mathcal{A}_{fr}\left(  a,b,c,d,e_{N}\right)  $,
 such that
\[
\zeta_{n}=v_0\text{ on }\partial Q,~\zeta_{n}\rightarrow v_0\text{ in
}L^{1}\left(  Q;\mathbb R^m\right),
\]
\[
\xi_{n}=\rho_{i\left(  n\right)  }\ast u_{0}\text{ on }\partial Q,~~\ \ \xi
_{n}\rightarrow u_{0}\text{ in }L^{1}\left(  Q;\mathbb{R}^{d}\right) 
\]%
and%
\begin{equation}\label{eqlemma4.7}
\begin{array}{ll}
\displaystyle{\underset{n\rightarrow\infty}{\lim\sup}\left(  \int_{Q}QV\left(
\zeta_{n},\nabla\xi_{n}\right)  dx+\int_{J_{\zeta_n}\cap Q}g(\zeta_n^+, \zeta_n^-, \nu_{\zeta_n})d {\cal H}^{N-1}\right)}
\\
\\
\displaystyle{\leq \underset{n\rightarrow\infty}{\lim\inf}\left(  \int_{Q}QV\left( v_{n},\nabla u_{n}\right)  dx+\int_{J_{v_n}\cap Q}g(v_n^+, v_n^-, \nu_{v_n})d {\cal H}^{N-1}
\right),}
\end{array} 
\end{equation}
where $QV$ represents the quasiconvex envelope of $V$ as in \eqref{Qfbar}.
\end{lemma}
We omit the proof since it is entirely similar to the one of Lemma \ref{Lemma4.1FM}. We just observe that there is no need of the first step where  a truncation argument for $v$ was built, since in the present context we deal with functions with finite range.

\medskip
The following result, which contains the properties satisfied by $K_2$  in \eqref{K2}, is analogous to  Proposition \ref{propK3} and it is stated for the reader's convenience. 
\begin{proposition}\label{propK2}
Let $V$ be as in \eqref{Vbar}. Let $K_2$ be the function introduced in \eqref{K2}. The following properties hold.

\begin{enumerate}

\item[a)] $\left\vert K_2\left(  a,b,c,d,\nu\right)  -K_2\left(
a',b',c',d',\nu\right)  \right\vert \leq
C\left(  \left\vert a-a'\right\vert +\left\vert b-b'\right\vert +\left\vert c-c'\right\vert +\left\vert d-d'\right\vert \right)  $ for every $\left(  a,b,c,d,\nu\right)  ,$ $\left(
a',b',c',d',\nu\right)  \in\{0,1\}  \times \{0,1\}  \times \mathbb R^{d}\times \mathbb R^{d}\times S^{N-1};$

\item[b)] $\nu\longmapsto K_2( a,b,c,d,\nu)  $ is upper
semicontinuous for every $( a,b,c,d)  \in \{0,1\}\times \{0,1\}\times
\mathbb R^{d}\times \mathbb R^{d};$

\item[c)] $K_2$ is upper semicontinuous in $\{0,1\}\times \{0,1\}\times
\mathbb R^{d}\times \mathbb R^{d}\times S^{N-1};$
\item[d)] $K_2\left(  a,b,c,d,\nu\right)  \leq C\left(  \left\vert
a-b\right\vert +\left\vert c-d\right\vert\right)  $ for every $\nu\in
S^{N-1}.$
\end{enumerate}
\end{proposition}

\begin{proof}[Proof of Theorem \ref{mainthm}]
The arguments develop as in Theorem \ref{mainthmgen}, essentially replacing $f$ by $V$ in \eqref{Vbar}, $v$ by $\chi$, the surface integral by $|D \chi|$, and using the blow-up argument introduced in \cite{FM1}, thus we will present just the main differences.

\noindent
{\bf Lower bound.}
Let $(\chi,u) \in BV(\Omega;\{0,1\})\times BV(\Omega;\mathbb R^d)$.
Without loss of generality we may assume that  for every $\{(\chi_n, u_n)\} \subset BV(\Omega;\{0,1\})\times BV(\Omega;\mathbb R^d)$ converging to $(\chi,u)$ in $L^1(\Omega; \{0,1\})\times L^1(\Omega;\mathbb R^d)$,
$\displaystyle{\underset{n\rightarrow\infty}{\lim\inf}\left(\int_{\Omega} V\left( 
\chi_n,\nabla u_n\right)dx  +|D \chi_n|(\Omega)\right)}$ is indeed a limit.
For every Borel set $B \subset \Omega$ define
$$
\displaystyle{\mu_n(B):=\int_B V\left( \chi_n,\nabla
u_{n}\right) dx + |D \chi_n|(B).}
$$ 
 The sequence $\{\mu_n\}$ behaves as in Theorem \ref{mainthmgen}, and its weak $\ast$ limit (up to a not relabelled subsequence) $\mu$ can be decomposed as in \eqref{mudecomposition} where, as in the remainder of the proof, $J_{(v,u)}$ has been replaced by $J_{(\chi,u)}$.
Moreover we emphasize that we have been considering $(\chi,u)$ as a unique field in $BV(\Omega; \mathbb R^{1+ d})$ and we have been exploiting the fact that $D^c(\chi,u)= (0, D^c u)$ (cf. Remark \ref{vmeas}).  By Besicovitch derivation theorem we deduce \eqref{BDT}.


We claim that%
\begin{equation}
\mu_{a}\left(  x_{0}\right)  \geq QV\left(  \chi\left(  x_{0}\right)  ,\nabla
u\left(  x_{0}\right)  \right)  ,~\text{for \ }\mathcal{L}^{N}-\text{a.e}%
.~x_{0}\in\Omega,\label{lboundbulkchi}%
\end{equation}%
\begin{equation}
\mu_{j}\left(  x_{0}\right)  \geq K_2\left(  \chi^+(x_0),\chi^-(x_0),u^{+}\left(  x_{0}\right)
,u^{-}\left(  x_{0}\right)  ,\nu_{(\chi,u)}\right)  ,~\text{for }\mathcal{H}%
^{N-1}-\text{a.e}.~x_{0}\in J_{(\chi,u)}\cap\Omega,\label{lboundjumpchi}%
\end{equation}%
\begin{equation}
\mu_{c}\left(  x_{0}\right)  \geq\left(  QV\right)  ^{\infty}\left(
\chi\left(  x_{0}\right)  ,\frac{dD^{c}u}{d\left\vert D^{c}u\right\vert
}\left(  x_{0}\right)  \right)  \text{ for }\left\vert D^{c}u\right\vert
-\text{a.e. }x_{0}\in\Omega. \label{lboundcantorchi}%
\end{equation}
If $\left(  \ref{lboundbulkchi}\right)  -\left(  \ref{lboundcantorchi}\right)  $
hold then the lower bound inequality for Theorem \ref{mainthm} follows.

\noindent\textbf{Step 1.}
Observing that by Proposition \ref{continuityQfbar} $QV$ satisfies $(F_1)-(F_3)$, the proof of \eqref{lboundbulkchi} develops as in Step 1 of Theorem \ref{mainthmgen}, just  applying \cite[formula (2.10) in Theorem 2.19]{FM2}, to the functional $G: (\chi,u)\in W^{1,1}(\Omega;\mathbb R^{1+d}) \to \int_{\Omega}QV(\chi,\nabla u)dx$.

\noindent\textbf{Step 2. }The proof of \eqref{lboundjumpchi} is very similar to  the one of \eqref{lboundjump}. Remind that $J_{( \chi,u)  }=J_\chi\cup J_{u}$ and $\nu_{(\chi,u)}= \nu_\chi$ for every $(\chi,u) \in BV(\Omega;\{0,1\}) \times W^{1,1}(\Omega;\mathbb R^d)$.
The same arguments of Step 2. in Theorem \ref{mainthmgen} allow us to fix $x_{0}\in J_{( \chi,u)}\cap\Omega$
such that \eqref{4.13}, \eqref{4.14}, \eqref{4.15} \eqref{4.16}  and \eqref{4.17} hold.

Recall that we denote $Q_{\nu(x_0)}$ by $Q$ and we may choose $\varepsilon>0$ such that
$\mu\left(  \partial\left(  x_{0}+\varepsilon Q\right)  \right)  =0.$ It results
\begin{align*}
\mu_{j}\left(  x_{0}\right)   &  \geq\lim_{\varepsilon\rightarrow0^{+}}%
\lim_{n\rightarrow\infty}\frac{1}{\varepsilon^{N-1}}\left(  \int%
_{x_{0}+\varepsilon Q}QV\left(  \chi_n\left(  x\right)  ,\nabla u_{n}\left(
x\right)  \right)  dx+|D \chi_n|(x_0+ \e Q) \right)  \\
&  =\lim_{\varepsilon\rightarrow0^{+}}\lim_{n\rightarrow\infty}\left(\varepsilon
\int_{Q}QV\left(  \chi_n\left(  x_{0}+\varepsilon y\right)  ,\nabla
u_{n}\left(  x_{0}+\varepsilon y\right)  \right)  dy +|D \chi_n(x_0+ \varepsilon y)|\left(Q\cap J\left( \chi_n, u_n\right)  -\frac{x_0}{\varepsilon}\right) \right) .
\end{align*}
Define $\chi_{n,\varepsilon}, u_{n, \e}, \nu_{n,\e}$ and $ \chi_0, u_0$ according to \eqref{vne} and \eqref{u0v0}.
Since $(\chi_n,u_n)\rightarrow (\chi,u)$ in $L^{1}\left(  \Omega;\mathbb{R}^{1+d}\right)$ we obtain \eqref{blabla} and  \eqref{blabla2},  with $v_{n,\e}$  and $v_0$ replaced by $\chi_{n,\e}$ and $\chi_0$, respectively.

Thus%
\begin{align*}
\mu_{j}\left(  x_{0}\right)   &  \geq\lim_{\varepsilon\rightarrow0^{+}}%
\lim_{n\rightarrow\infty}\left(\int_{Q}QV^{\infty}\left( \chi_{n,\varepsilon
}\left(  y\right)  ,\nabla u_{n,\varepsilon}\left(  y\right)  \right)
dy+ |D\chi_{n, \e}|(Q)\right.\\
&  \left.+\int_{Q}\varepsilon QV\left( \chi_{n,\varepsilon}\left(  y\right)
,\frac{1}{\varepsilon}\nabla u_{n,\varepsilon}\left(  y\right)  \right)
-QV^{\infty}\left( \chi_{n,\varepsilon},\nabla u_{n,\varepsilon}\right)  dy \right).
\end{align*}

By Remark \ref{propfinfty} $(v)$ we can argue as in the estimates \cite[(3.3)-(3.5)]{FM2},  obtaining
\[
\mu_{j}\left(  x_{0}\right)  \geq\underset{\varepsilon\rightarrow0^{+}%
}{\lim\inf}~\underset{n\rightarrow\infty}{\lim\inf}\left(  \int_{Q}QV^{\infty
}\left( \chi_{n,\varepsilon}\left(  y\right)  ,\nabla u_{n,\varepsilon
}\left(  y\right)  \right)  dy+ |D \chi_{n,\e}|(Q) \right)  .
\]

\noindent Applying Lemma \ref{Lemma4.1FMchi} with $QV$ replaced by $QV^\infty$, $T\subset \mathbb R^m$ replaced by $\{0,1\}$, the surface integral replaced by the total variation, $K_{fr}$ and ${\mathcal A}_{fr}$ replaced by $K_2$ and ${\mathcal A}_2$ respectively, and using Remark \ref{propfinfty}, we may find $\left\{  \left(  \zeta_{k},\xi_{k}\right)  \right\}   $  $ \in\mathcal{A}_2\left(  \chi^+(x_0),\chi^-(x_0),u^{+}\left(
x_{0}\right)  ,u^{-}\left(  x_{0}\right)  ,\nu\left(  x_{0}\right)  \right)  $
such that

\begin{equation}\nonumber
\displaystyle{\mu_{j}(x_0)  \geq\lim_{k\rightarrow\infty}\left(  \int_Q QV^{\infty}\left( \zeta_{k},\nabla\xi_{k}\right)  dx+ |D \zeta_k|(Q) \right) \geq K_2\left(
\chi^+(x_0), \chi^-(x_0),u^{+}\left(  x_{0}\right)  ,u^{-}\left(  x_{0}\right)  ,\nu\left(
x_{0}\right)  \right) }.
\end{equation}

\noindent\textbf{Step 3.} The proof of \eqref{lboundcantorchi} follows identically as in Step 3, Theorem \ref{lsctheorem}, namely applying \cite[formula (2.12) in Theorem 2.19]{FM2} to the functional $G$ introduced in Step 1 herein and this concludes the proof.

\medskip
{\bf Upper Bound.} The proof of the upper bound develops in three steps as the one of Theorem \ref{thupperbound}.  Furthermore Propositions \ref{propqcx} 
can be readapted  replacing $Qf$ by $QV$ and the surface integral by $|D\chi|$.
  
\noindent {\bf Step 1.} For ${\cal L}^N$- a.e. $x_0 \in \Omega$, $x_0$ is a Lebesgue point for $U\equiv (\chi, u)$ such that also \eqref{upper1} and \eqref{upper3} hold for $QV$. In analogy with Theorem \ref{thupperbound} Step 1- we apply for every $\chi \in BV(\Omega;\{0,1\})$, the Global Method \cite[Theorem 4.1.4]{BFMGlobal} to the functional $G: (u,A) \in W^{1,1}(\Omega;\mathbb R^m)\times {\cal A}(\Omega) \to \int_\Omega QV(\chi, \nabla u)dx$, to obtain an integral representation for the functional \eqref{auxrelax} for every $(u, A) \in BV(\Omega;\mathbb R^m)\times {\cal A}(\Omega)$.  Moreover we can write
\begin{equation}\nonumber
\displaystyle{{\cal F}_{\cal OD}(\chi,u; A) \leq {\cal G}(u; A)+ |D \chi|(A).}
\end{equation}
 Differentiating with respect to ${\cal L}^N$ we obtain
$\displaystyle{\frac{d {\cal F}_{\cal OD}(\chi, u;\cdot)}{d {\cal L}^N} \leq V_0(x_0, \nabla u(x_0)), }$
where $V_0$ is the correspective of $f_0$ in \eqref{f00} where $Qf$ has been replaced by $QV$.
Arguing  as in the last part of Theorem \ref{thupperbound} Step 1, applying Lemma \ref{Lemma0}, we deduce that $V_0(x_0,\xi_0) \leq QV(\chi(x_0), \xi_0)$  and this leads to the conclusion when $u \in BV(\Omega;\mathbb R^d)\cap L^\infty(\Omega;\mathbb R^d)$.

 \noindent {\bf Step 2.} The same type of arguments as in Step 1,  applies to the proof of the upper bound for the Cantor part. Radon-Nikod\'ym theorem implies \eqref{CantorU} for every $U\equiv(\chi,u)\in BV(\Omega;\{0,1\})\times (BV(\Omega;\mathbb R^d)\cap L^\infty(\Omega;\mathbb R^d))$, with $|D^c u|$ and  $\sigma$ mutually singular. Moreover \eqref{cantor}, \eqref{ABrankone}, \eqref{as3.27BFMglobal} hold, the Global Method \cite[Theorem 4.1.4]{BFMGlobal} applies to \eqref{auxrelax} and a differentiation with respect to $|D^c u|$  at $x_0$ provides
$\displaystyle{\frac{d {\cal F}_{\cal OD}(\chi, u; \cdot)}{d |D^c u|}(x_0) \leq h(x_0, a_u, \nu_u),}$
where $h(x,a,\nu)$ is given by \eqref{f0}.
Remark \ref{propfinfty} applied to $QV^\infty$, Lemma \ref{Lemma0} and the same techniques employed in the last part of Theorem \ref{thupperbound} Step 2, entail
$$
h(x_0, a,\nu)\leq QV^\infty(\chi(x_0), a \otimes \nu),
$$
and that concludes the proof of the Cantor part for $(\chi,u)\in BV(\Omega;\{0,1\})\times (BV(\Omega;\mathbb R^d)\cap L^\infty(\Omega;\mathbb R^d))$.

\noindent{\bf Step 3.} We claim that
\begin{equation}\label{claimjumpod}
\displaystyle{ {\cal F}_{\cal OD}(U;J_U)\leq \int_{J_U}K_2(\chi^+,\chi^-,u^+, u^-, \nu_{\chi, u})d {\cal H}^{N-1},}
\end{equation}
for every $(\chi,  u)\in BV(\Omega;\{0,1\})\times (BV(\Omega;\mathbb R^d)\cap L^\infty(\Omega;\mathbb R^d))$. 
The proof of \eqref{claimjumpod} is divided in three parts, according to the assumptions on the limit functions $u$. Namely,

\noindent {\it Case 1.} $U(x):=(1,c)\chi_E(x)+ (0,d)(1-
\chi_E(x))$, with $P(E,\Omega)< +\infty$,

\noindent {\it Case 2.} $u(x)= \sum_{i=1}^\infty c_i \chi_{E_i}(x)$, where $\{E_i\}_{i=1}^\infty$ forms a partition of $\Omega$ into sets of finite perimeter and $c_i \in \mathbb R^d$,

\noindent {\it Case 3.} $u(x)\in BV(\Omega;\mathbb R^d)\cap L^\infty(\Omega;\mathbb R^d)$.

For what concerns Case 1, we  consider first the unit open cube $Q \subset \mathbb R^N$, and make the same assumptions on the target function $U$ as in Theorem \ref{thupperbound} Step 3,  Case 1. Then we can invoke an argument analogous to Proposition \ref{prop3.5BBBF}, without invoking any truncation arguments as those in Remark \ref{applicationofprop3.4}. This guarantees that there exists $(\chi_n, u_n) \in {\cal A}_2(1,0,c,d,e_N)$ such that $(\chi_n, u_n)\to (\chi, u)$ in $L^1(Q;\mathbb R^{1+d})$ and 
\begin{equation}\nonumber
\displaystyle{K_2(1,0, c,d,e_N) =\lim_{n \to \infty}\left(\int_Q QV^\infty(\chi_n(x), \nabla u_n(x)) dx + |D \chi_n|(Q)\right).}
\end{equation} 
Then the proof develops exactly as Theorem \ref{thupperbound}, just taking into account that the sequence $z_{n,k}$ therein is built replacing $a, b$  and $v_n$ by $1,0$ and $\chi_n$ respectively, thus leading to
$$
\displaystyle{{\cal F}_{OD}(\chi, u; Q)\leq \frac{QV(1,0)+ QV(0,0)}{2} + K_2(1,0, c,d ,e_N).}
$$ 
For what concerns a more general set $A$ than $Q$, like in Theorem \ref{thupperbound} Step 3, Case 1, we achieve the following representation
$$
\displaystyle{{\cal F}_{\cal OD}(\chi, u; A)\leq \int_A QV(\chi(x), 0) dx + \int_ {J_U} K_2(1,0,c,d,\nu)d {\cal H}^{N-1}.}
$$
Then the strategy follows b), c), d) in Theorem \ref{thupperbound} Step 3,  Case 1, hence we obtain
$$
\displaystyle{{\cal F}_{\cal OD}(\chi, u; J_{\chi, u}) \leq \int_{J_{\chi, u}} K_2(1,0,c,d,\nu)d {\cal H}^{N-1}.}
$$

\noindent {\it Case 2.}  and {\it Case 3.} By the properties of $K_2$ in Proposition \ref{propK2}, the proof  develops as in \cite[Proposition 4.8, Case 2 and Case 3]{AMT}.
This concludes the proof of the upper bound when $(\chi, u)\in BV(\Omega;\{0,1\}) \times (BV(\Omega;\mathbb R^d)\cap L^\infty(\Omega;\mathbb R^d)$).

 The general case, since $\chi \in BV(\Omega;\{0,1\})$ and can be fixed, is identical to \cite[Section 5, Step 4.]{FM2}, where the truncation procedures involves just $u$. 

Putting together {\bf Lower bound} and {\bf Upper bound} we achieve the desired result. 
\end{proof}

\begin{remark}\label{specializeK2}We observe that, as in Remark \ref{specializeK3}, $K_2$ admits the following equivalent representation:
\begin{itemize}
\item[i)]on $J_u\setminus J_\chi$, $K_2(a,a,c,d,\nu)=QV^\infty(a,(c-d)\otimes \nu)$,  with $Q V^\infty$ as in \eqref{QVinfty}.

\item[ii)] on $J_\chi \setminus J_u$, $K_2(a,b,c,c,\nu)= |(a-b)\otimes \nu|$ , i.e. $\displaystyle{\int_{J_\chi}K_2(\chi^+, \chi^- ,u^+,u^+, \nu)d {\cal H}^{N-1}= |D \chi|(\Omega)}$.

\item[iii)] Note that 
$
\displaystyle{K_2(a,b,c,d,\nu)\geq \inf\left\{\int_{Q_\nu}\left(QV^\infty(w(x), \nabla u(x)) + |\nabla w(x)|\right)dx: (w,u) \in {\cal A}(a,b,c,d,\nu)\right\},}
$
where this latter density is the density $K(a,b,c,d,\nu)$ first introduced in \cite{FM2} (cf. also \cite[formula (5.83)]{AFP2}) and
\begin{equation}\nonumber
\begin{array}{ll}
\mathcal{A}\left(  a,b,c,d,\nu\right) :=\left\{  \left(w,u\right)
\in W^{1,1}\left(
Q_{\nu};\mathbb{R}^{1+ d}\right)  :\right. (w(y),u\left(  y\right))  =(a,c)\text{ if }y\cdot\nu=\frac{1}{2}, \\
\\
(w(y),u\left(
y\right))  =(b,d)\text{ if }y\cdot\nu=-\frac{1}{2},
\left.  (w, u)\text{ are
}1-\text{periodic in }\nu_{1},\dots,\nu_{N-1} \hbox{ directions}\right\}.
\end{array}
\end{equation}

\noindent On the other hand, if $W_i$, $i=1,2$ in \eqref{H1} are proportional (as in the model presented in \cite{AB}), i.e. $W_2 = \alpha W_1$, $\alpha >1$, taking $V$ as in \eqref{Vbar}, since for every $q \in [0,1]$ $QV^\infty(q,z)=q QW_1^\infty(z) + \alpha(1-q) QW_1^\infty(z)$, then we claim that $K_2$ is equal to $K$ of \cite{FM2}. Indeed, without loss of generality, assuming $W_1$, quasiconvex and positively $1$-homogeneous, it is enough to observe that for every $(w,u) \in {\cal A}(a,b,c,d,\nu)$,
$$
\begin{array}{ll}
\displaystyle{K(1,0,c,d,\nu)\geq \int_{Q_v}\left(w(x) W_1(\nabla u(x))+ \alpha (1-w(x)) W_1(\nabla u(x)) + |\nabla w(x)| \right)dx \geq \int_{Q_\nu} (W_1(\nabla u(x)) + 1)dx, }
\end{array}
$$ 
where it has been used the fact that $\alpha + (1-\alpha)w \geq 1$ and
 $$\displaystyle{\int_{Q_\nu}|\nabla w|dx \geq \left|\int_{Q_\nu}\nabla w\right| dx = \left|\int_{\partial Q_\nu} w \otimes \nu(x) d {\cal H}^{N-1}\right| = 1.} $$
Taking a sequence of characteristic functions $\{\chi_\e\}$, admissible for $ {\cal A}_2(1,0,c,d,\nu)$ in \eqref{AFR}, such that their value is $1$ in a strip of the cube orthogonal to $\nu$ and of thickness $1-\e$, then, it results
$$
\begin{array}{ll}
\displaystyle{\int_{Q_\nu}W_1(\nabla u(x))dx + 1 = \lim_{\e \to 0^+} \int_{Q_\nu}(\chi_\e W_1 (\nabla u(x) )+ \alpha (1-\chi_\e) W_1(\nabla u(x)) d x + |D \chi_\e|(Q_\nu)  }\\
\\
\displaystyle{\geq K_2(1,0,c,d,\nu),}
\end{array}
$$ 
and this proves our claim. Observe also that if $\alpha \in (0,1)$, then the result remains true, it is enough to express $W_1$ in terms of $W_2$.
\end{itemize}
\end{remark}

\medskip
As emphasized in \cite[Remark 2.4]{AB} one can consider mixtures of more than two conductive materials, hence we observe that Theorem \ref{mainthm} can be extended with minor changes to these models leading to formula \eqref{fr} in the remark below.
 

\begin{remark}\label{remmainthm}
Let $T$ be a finite subset of $\mathbb R^m$, Theorem \ref{mainthm} applies also to energies of the type
 $F_{fr}:L^{1}(\Omega;T)\times L^{1}(\Omega;\mathbb{R}^{d})\times
\mathcal{A}\left(  \Omega\right)  \rightarrow\lbrack0,+\infty]$ defined by
\begin{equation}
F_{fr}(v,u;A):=\left\{
\begin{array}
[c]{lll}
{\displaystyle\int_{A}} V\left(v,  \nabla u\right)  dx+ \displaystyle{\int_{{J_v}\cap A}}g(v^+, v^-,\nu_v)d {\cal H}^{N-1} &  & \text{in
}BV(A;T)\times W^{1,1}(A;\mathbb{R}^{d}),\text{\bigskip}\\
+\infty &  & \text{otherwise.}%
\end{array}
\right.  \label{FFR}%
\end{equation}

Indeed, consider the relaxed localized energy of \eqref{FFR} given by%
\begin{equation}\nonumber
\begin{array}
[c]{c}%
\mathcal{F}_{fr}\left( v,u;A\right)  :=\inf\left\{  \underset{n\rightarrow
\infty}{\lim\inf}%
{\displaystyle\int_{A}}V\left(
v_n,\nabla u_{n}\right)   dx+ \displaystyle{\int_{J_{v_n} \cap A}}g(v_n^+, v_n^-, \nu_{v_n})d {\cal H}^{N-1}:\right. \\
\left. \left\{ (v_n, u_n)\right\}  \subset BV(A;T) \times W^{1,1}\left(  A;\mathbb{R}^{d}\right),
(v_n, u_n) \to (v,u) \text{ in }L^1(A;T)\times L^1(A;\mathbb R^d) 
\right\},
\end{array}
\end{equation}
with  $V$  and $g$ as in \eqref{Vgfr} satisfying
$(F_1)- (F_4)$ and $(G_1)- (G_3)$, respectively.

Moreover define ${\overline F}_{fr}:BV(A;T)\times BV(A;\mathbb{R}^{d}%
)\times\mathcal{A}\left(  \Omega\right)  \rightarrow\lbrack0,+\infty]$ as
\begin{equation}\nonumber
{\overline F}_{fr}\left(  v,u;A\right)  :=\int_{A}QV\left( v,\nabla u\right)
dx+\int_{A}QV^{\infty}\left(  v,\frac{dD^{c}u}{d\left\vert D^{c}%
u\right\vert }\right)  d\left\vert D^{c}u\right\vert +\int_{J_{\left(
v,u\right)  }\cap A}K_{fr}\left( v^{+},v^{-},u^{+},u^{-},\nu\right)
d\mathcal{H}^{N-1}
\end{equation}
where $QV$ is the quasiconvex envelope of $V$ given in $\left(  \ref{Qfbar}%
\right)  ,$ $QV^{\infty}$ is the recession function of $QV,$ introduced in 
\eqref{QVinfty},
and%
\begin{equation}
{K_{fr}(a,b,c,d,\nu):=\inf}\left\{ 
{\displaystyle\int_{Q_{\nu}}}
QV^{\infty}(v,\nabla u(x))dx+\int_{Q_\nu}g(v^+, v^-, \nu_v)d {\cal H}^{N-1}
:\left( v,u\right)
{{\in\mathcal{A}_{fr}(a,b,c,d,\nu)}}\right\}  ,\label{K2FR}%
\end{equation}
where ${\mathcal A}_{fr}$ is the set defined in \eqref{AFR}, with $\{0,1\}$ replaced by the finite set $T\subset \mathbb R^m$.
Thus, we are lead to the following representation:
for every $(v,u)\in L^1(\Omega;T)\times
L^1(\Omega;\mathbb{R}^{d})$%
\begin{equation}\label{fr}
\mathcal{F}_{fr}(v,u;A)=\left\{
\begin{array}{ll}
{\overline F}_{fr}(v,u;A) &\hbox{ if } (v, u)\in BV(A;T) \times BV(A;\mathbb R^d) \\
\\
+\infty & \hbox{ otherwise.}
\end{array}
\right.
\end{equation}


\end{remark}

\begin{remark}\label{K2K3}
In general we cannot expect $K_3= K_{fr}$ since in \eqref{K2FR}, the function $g$ is defined in $T \times T \times S^{N-1}$, with $ T \subset \mathbb R^d$ and ${\rm card}(T)$ finite, while in \eqref{K3}, $g$ is defined in $\mathbb R^d \times \mathbb R^d \times S^{N-1}$. In particular we recall that in $J_v\setminus J_u$, $K_3$ coincides with ${\cal R}g$, the $SBV$-elliptic envelope of $g$ as in \cite{BFLM}, while $K_{fr}$ in \eqref{K2FR} is given by the $BV$-elliptic envelope introduced by Ambrosio and Braides, cf. \cite[Definition 5.13]{AFP2}. Analogously, it is easily seen that  $K_2$ coincides with $|D \chi|$  in $J_\chi \setminus J_u$ .

\end{remark}

\section*{Acknowledgements}
This paper has been written during various visits of the authors at Departamento de Matem\'atica da Universidade de \'Evora and at Dipartimento di Ingegneria Industriale dell' Universit\'a di Salerno, whose kind hospitality and support have been gratefully acknowledged. 
   
The authors are indebted to Irene Fonseca for having suggested this problem and  for the many discussions on the subject.

The work of both authors was partially supported by Funda\c{c}\~{a}o para a Ci\^{e}ncia e a Tecnologia (Portuguese Foundation for Science and Technology) through CIMA-UE, UTA-CMU/MAT/0005/2009 and through GNAMPA project 2013 `Funzionali supremali: esistenza di minimi e condizioni di semicontinuit\'a nel caso vettoriale'.

\end{document}